\documentclass[reqno]{amsart}

    \usepackage[
    left=3cm,
    right=3cm,
    top=3cm, 
    bottom=3cm,
    a4paper]{geometry}

    \let\tempone\itemize
    \let\temptwo\enditemize
    \renewenvironment{itemize}{\tempone\addtolength{\itemsep}{0.5\baselineskip}}{\temptwo}
    
    \let\tempenum\enumerate
    \let\tempenumtwo\endenumerate
    \renewenvironment{enumerate}{\tempenum\addtolength{\itemsep}{0.5\baselineskip}}{\tempenumtwo}

            \makeatletter
            \let\origsection\section
            \renewcommand\section{\@ifstar{\starsection}{\nostarsection}}
            \newcommand\nostarsection[1]
            {\sectionprelude\origsection{#1}\sectionpostlude}
            \newcommand\starsection[1]
            {\sectionprelude\origsection*{#1}\sectionpostlude}
            \newcommand\sectionprelude{%
              \vspace{0.75em} 
            }
            \newcommand\sectionpostlude{%
              \vspace{1em}   
            }
            \makeatother
            
            \makeatletter
            \let\origsubsection\subsection
            \renewcommand\subsection{\@ifstar{\starsubsection}{\nostarsubsection}}
            \newcommand\nostarsubsection[1]
            {\sectionprelude\origsubsection{#1}\subsectionpostlude}
            \newcommand\starsubsection[1]
            {\subsectionprelude\origsubsection*{#1}\subsectionpostlude}
            \newcommand\subsectionprelude{%
              \vspace{0.25em} 
            }
            \newcommand\subsectionpostlude{%
              \vspace{0.0em}   
            }
            \makeatother

\usepackage[bb=boondox]{mathalfa}
\usepackage[T1]{fontenc}
\usepackage{lmodern}
\usepackage[utf8]{inputenc}
\usepackage{bm}
\usepackage{mathrsfs}
\usepackage{tikz-cd}
\usepackage{bbm}
\usepackage{mathtools}
\usepackage[UKenglish]{babel}
\usepackage{mathrsfs}
\usepackage{comment}
\usepackage{tikz-cd}
\usepackage{tikz}

\usetikzlibrary{decorations.markings}
\usetikzlibrary{shapes.geometric, calc}
\usetikzlibrary{hobby}

\usepackage{amsmath,amssymb,amsthm,graphicx,amscd,amsfonts,latexsym}
\usepackage{pstricks}
\usepackage[normalem]{ulem}

\usepackage{lscape}
\usepackage{hyperref}
\usetikzlibrary{arrows}
\usetikzlibrary{decorations}
\usetikzlibrary{shapes}
\usetikzlibrary{decorations.pathreplacing,decorations.markings} 
\usepackage{graphicx}
\usepackage{caption} 
\usepackage{float}
\usepackage{upgreek}

\usetikzlibrary{positioning}
\usetikzlibrary{arrows,decorations.pathmorphing}

\usepackage{ifthen} 
\newcounter{pos} 
\tikzset{									
	initcounter/.code={\setcounter{pos}{0}},
	style between/.style n args={3}{
		postaction={
			initcounter,
			decorate,
			decoration={
				show path construction,
				curveto code={
					\addtocounter{pos}{1}
					\pgfmathtruncatemacro{\min}{#1 - 1}
					\ifthenelse{\thepos < #2 \AND \thepos > \min}{
						\draw[#3]
						(\tikzinputsegmentfirst)
						..
						controls (\tikzinputsegmentsupporta) and (\tikzinputsegmentsupportb)
						..
						(\tikzinputsegmentlast);
					}{}
				}
			}
		},
	},
}

\newcommand{\seq}{\mathbbm{m}}
\renewcommand{\P}{P^\bullet}
\makeatletter
\newcommand{\colim@}[2]{%
  \vtop{\m@th\ialign{##\cr
    \hfil$#1\operator@font colim$\hfil\cr
    \noalign{\nointerlineskip\kern1.5\ex@}#2\cr
    \noalign{\nointerlineskip\kern-\ex@}\cr}}%
}
\newcommand{\colim}{%
  \mathop{\mathpalette\colim@{\rightarrowfill@\textstyle}}\nmlimits@
}
\makeatother

\DeclareMathOperator{\Img}{Im}
\DeclareMathOperator{\module}{mod}

\DeclareMathOperator{\Hom}{Hom}

\DeclareMathOperator{\rk}{rk}

\DeclareMathOperator{\End}{End}
\DeclareMathOperator{\Ext}{Ext}
\DeclareMathOperator{\Pic}{Pic}
\DeclareMathOperator{\Aut}{Aut}
\DeclareMathOperator{\Coh}{Coh}
\newcommand{\Sing}{\mathcal{D}_{\operatorname{sg}}(C_n)}
\DeclareMathOperator{\Int}{Int}

\DeclareMathOperator{\Perf}{Perf}

\DeclareMathOperator{\PGL}{PGL}
\DeclareMathOperator{\PSL}{PSL}

\DeclareMathOperator{\marked}{\mathcal{M}}

\DeclareMathOperator{\MCG}{\mathcal{M}\mathcal{C}\mathcal{G}}

\DeclareMathOperator{\PMCG}{\mathcal{P}\mathcal{M}\mathcal{C}\mathcal{G}}
\DeclareMathOperator{\characteristic}{char}

\DeclareMathOperator{\SExt}{\mathcal{E}\!\mathit{xt}}
\newcommand{\CVb}[1]{\operatorname{\Omega}_{\text{Vect}}^{#1} }

\DeclareMathOperator{\Mat}{Mat}
\DeclareMathOperator{\Dinv}{\mathcal{D}_{\text{loop}}}
\DeclareMathOperator{\Dpart}{\mathcal{D}_{\partial}}
\DeclareMathOperator{\supp}{supp}
\DeclareMathOperator{\Vect}{Vect}

\DeclareMathOperator*{\homom}{\mathcal{H}\!\mathit{om}}
        \newcommand\restr[2]{{
		\left.\kern-\nulldelimiterspace 
		#1 
		\vphantom{\big|} 
		\right|_{#2} 
        }}

\tikzset{
	set arrow inside/.code={\pgfqkeys{/tikz/arrow inside}{#1}},
	set arrow inside={end/.initial=>, opt/.initial=},
	/pgf/decoration/Mark/.style={
		mark/.expanded=at position #1 with
		{
			\noexpand\arrow[\pgfkeysvalueof{/tikz/arrow inside/opt}]{\pgfkeysvalueof{/tikz/arrow inside/end}}
		}
	},
	arrow inside/.style 2 args={
		set arrow inside={#1},
		postaction={
			decorate,decoration={
				markings,Mark/.list={#2}
			}
		}
	},
}

\makeatletter
\tikzset{commutative diagrams/.cd,arrow style=tikz,diagrams={>=latex'}}\tikzset{join/.code=\tikzset{after node path={%
			\ifx\tikzchainprevious\pgfutil@empty\else(\tikzchainprevious)%
			edge[every join]#1(\tikzchaincurrent)\fi}}}
\makeatother

\tikzset{>=stealth',every on chain/.append style={join},
	every join/.style={->}}

\tikzset{every loop/.style={min distance=25mm,in=50,out=100,looseness=5}}

\newtheorem{prf}{Proof}[section]

\theoremstyle{remark}

  \newtheoremstyle{ownTheoremStyle}
  {1em}
  {1em}
  {\itshape}
  {}
  {\bfseries}
  {.}
  { }
  {}

  \newtheoremstyle{ownDefinitionStyle}
  {1em}
  {1em}
  {}
  {}
  {\bfseries}
  {.}
  { }
  {}

\theoremstyle{ownTheoremStyle}

\newtheorem{thm}[prf]{Theorem}

\newtheorem{Introthm}{Theorem}

\newtheorem{lem}[prf]{Lemma}
\newtheorem{prp}[prf]{Proposition}

\newtheorem{cor}[prf]{Corollary}

\theoremstyle{ownDefinitionStyle}
\newtheorem{exa}[prf]{Example}
\newtheorem*{convention}{Convention}

\newtheorem{definition}[prf]{Definition}

\newtheorem{rem}[prf]{Remark}

\newtheorem{notation}[prf]{Notation}

\newcommand{\quotient}[2]{{\left.\raisebox{.2em}{$#1$}\middle/\raisebox{-.2em}{$#2$}\right.}}

\numberwithin{equation}{section}

\newcommand{\cX}{\mathcal{X}}

\def\centerarc[#1](#2)(#3:#4:#5) 
{ \draw[#1] ($(#2)+({#5*cos(#3)},{#5*sin(#3)})$) arc (#3:#4:#5); }

\def\triangle[#1][#2][#3][#4][#5][#6][#7]%
{\begin{displaymath}\begin{tikzcd}[ampersand replacement=\&] #1 \arrow{r}{#4} \& #2 \arrow{r}{#5} \& #3 \arrow{r}{#6} \& #1[#7]\end{tikzcd}\end{displaymath}
}

\title[\resizebox{5.6in}{!}{Spherical objects, transitivity and auto-equivalences of Kodaira cycles}]{\MakeUppercase{Spherical objects, transitivity and auto-equivalences of Kodaira cycles via gentle algebras}}
\author{Sebastian Opper}

\address{Sebastian Opper, Charles University, Faculty of Mathematics and Physics, Ke Karlovu 3, 121 16 Praha 2, Czech Republic}
\email{opper@karlin.mff.cuni.cz}

\begin{document}
	
	\begin{abstract}
	This paper studies the class of spherical objects over any Kodaira $n$-cycle of projective lines and provides a parametrization of their isomorphism classes in terms of closed curves on the $n$-punctured torus without self-intersections. Employing recent results on gentle algebras, we derive a topological model for the bounded derived category of any Kodaira cycle. The groups of triangle auto-equivalences of these categories are computed and are shown to act transitively on isomorphism classes of spherical objects. This answers a question by Polishchuk \cite{Polishchuk} and extends earlier results by Burban-Kreussler \cite{BurbanKreusslerDegenerations} and Lekili-Polishchuk \cite{LekiliPolishchuk2016}. The description of auto-equivalences is further used to establish faithfulness of a mapping class group action defined by Sibilla \cite{SibillaMappingClassGroupAction}. The final part describes the closed curves which correspond to vector bundles and simple vector bundles. This leads to an alternative proof of a result by Bodnarchuk-Drozd-Greuel \cite{BodnarchukDrozdGreuel} which states that simple vector bundles on cycles of projective lines are uniquely determined by their multi-degree, rank and determinant. As a by-product we obtain a closed formula for the cyclic sequence of any simple vector bundle on $C_n$ as introduced by Burban-Drozd-Greuel \cite{BurbanDrozdGreuel}.
	\end{abstract}

	\maketitle

	\setcounter{tocdepth}{1}
	\tableofcontents

	\section*{Introduction}
	\noindent Our main objective in this paper is to study the class of \textit{spherical objects} in the derived category of the singular algebraic curve $C_n$ for any $n \geq 1$ which is known as a ``Kodaira cycle of $n$ projective lines'', or, a ``cycle of projective lines'' for short.  A cycle $C_n$ is a rational curve with $n$ irreducible components and may be thought of as a union of $n$ projective lines glued in a cycle as illustrated in Figure \ref{FigureE3}. Spherical objects  are perfect complexes over $C_n$ and form a class which is stable under auto-equivalences of the derived category $\mathcal{D}^b(C_n)$. They were introduced by Seidel and Thomas \cite{SeidelThomas} who showed that every spherical object induces an auto-equivalence of the bounded derived category which referred to as a \textit{spherical twist}. \medskip

	\begin{figure}
    \begin{tikzpicture}[scale=0.8]
        
    
      \begin{scope}[yshift=1.12cm,xshift=-1cm, scale=0.8*1.75, hobby]
        
        	\draw[] plot  [line width=2,  tension=0.75] coordinates { (1,1) (0,0) (-1, -0.35) (-1.5,0) (-1,0.35) (0,0) (1,-1) };
     
        \end{scope}

    
         \begin{scope}[yshift=0.115cm, xshift=2cm, scale=0.9*1.75, hobby]
        
        	\draw[] plot  [line width=2,  tension=0.9] coordinates { ({0.75-0.5*cos(60)},{-0.5*sin(60)}) ({0.75-0.5*cos(60)+1.5*cos(60)-0.375*cos(300)+0.05},{0.75*sin(60)}) ({0.75-0.5*cos(60)},{2*sin(60)})  };
        	
            	\draw[] plot  [line width=2,  tension=0.9] coordinates { ({0.75+1.5*cos(60)-0.5*cos(300)},{1.5*sin(60)-0.5*sin(300)}) ({0.75-0.5*cos(60)+0.5*cos(60)-0.125*cos(300)-0.05},{0.75*sin(60)}) ({0.75+1.5*cos(60)-0.5*cos(300)},{1.5*sin(60)+2*sin(300)}) };
     
        \end{scope}

        \begin{scope}[xshift=5.5cm, scale=1.75, hobby]
    		\draw ({0.75-0.5*cos(0)},{-0.5*sin(0)})--({0.75+2*cos(0)},{2*sin(0)});
    		\draw ({0.75-0.5*cos(60)},{-0.5*sin(60)})--({0.75+2*cos(60)},{2*sin(60)});
    		\draw ({0.75+1.5*cos(60)-0.5*cos(300)},{1.5*sin(60)-0.5*sin(300)})--({0.75+1.5*cos(60)+2*cos(300)},{1.5*sin(60)+2*sin(300)});
    	\end{scope}
	\end{tikzpicture}

		\caption{The curves $C_1$, $C_2$ and $C_3$ (from left to right).} \label{FigureE3}
	\end{figure}
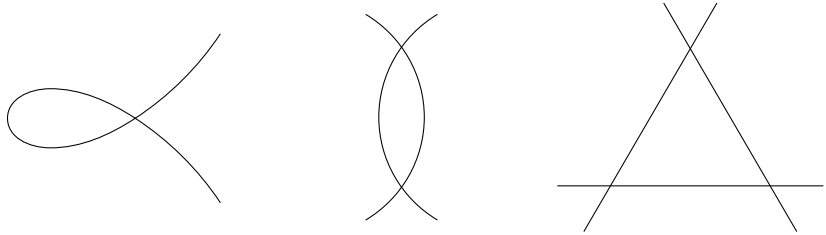

	\noindent In work of Polishchuk \cite{Polishchuk}, spherical objects on  cycles of projective lines (and more general curves of arithmetic genus one) were used to construct solutions of the classical and the associative Yang-Baxter equation. His construction leads to the question whether spherical objects can be classified in a meaningful way. Polishchuk asked further whether the  group of auto-equivalences of $\mathcal{D}^b(C_n)$ acts transitively on the set of isomorphism classes of spherical objects. An affirmative answer to his question is given in Theorem \ref{IntroTheoremTransitivitySphericalOjects}.\medskip
	
	\noindent The problems of classification and transitivity were studied by Burban-Kreussler \cite{BurbanKreusslerDegenerations} and Lekili-Polishchuk \cite{LekiliPolishchuk2016}. Burban and Kreussler established transitivity and provided a classification of spherical objects over $C_1$ showing that any spherical object is isomorphic to a shift of a skyscraper sheaf of a smooth point or a simple vector bundle. The latter were classified by means of their rank, multi-degree and determinant as shown by Bodnarchuk, Drozd and Greuel \cite{BodnarchukDrozdGreuel}. For $n > 1$ however, the class of spherical objects is more complicated as found by Burban and Burban \cite{BurbanBurban} who constructed spherical complexes in $\mathcal{D}^b(C_n)$ which are not quasi-isomorphic to the shift of a sheaf. 
	
	Lekili and Polishchuk \cite{LekiliPolishchuk2016} proved for arbitrary $n$ that the group of auto-equivalences of $\mathcal{D}^b(C_n)$ acts transitively on simple vector bundles. Their proof employed some of their earlier work \cite{LekiliPolishchukMirrorSymmetry} which showed the equivalence of the derived category of a cycle $C_n$ and the wrapped Fukaya category of an $n$-punctured torus after Haiden, Katzarkov and Kontsevich \cite{HaidenKatzarkovKontsevich}.\medskip
	
	\noindent We approach the problem of classification of spherical objects and transitivity from a similar angle. As shown by Burban and Drozd \cite{BurbanDrozdTilting}, there exists an embedding $\Perf(C_n) \hookrightarrow \mathcal{D}^b(\Coh \mathbb{X}_n)$ of the category of perfect complexes over $C_n$ into the derived category of a certain non-commutative curve $\mathbb{X}_n$. They further proved that the latter has a tilting object and is equivalent to the derived category of a finite dimensional algebra $\Lambda_n$ which belongs to the well-understood class of \textit{gentle algebras}.  Although it is difficult to compute the resulting embedding $\Perf(C_n) \hookrightarrow \mathcal{D}^b(\Lambda_n)$ directly, they determined the images of the Picard group and skyscraper sheaves of smooth points. \medskip
	
\noindent	The results of this paper also rely on \cite{OpperPlamondonSchroll} and \cite{OpperDerivedEquivalences} by Plamondon, Schroll and the author which study derived categories of gentle algebras in terms of the geometry of a surface (see also \cite{HaidenKatzarkovKontsevich}, \cite{LekiliPolishchukGentle} and \cite{AmiotPlamondonSchroll} for related work). A result in \cite{OpperDerivedEquivalences} describes the group of auto-equivalences of such categories as an extension of the mapping class group of the associated surface. As a particular instance of the connection between the derived category of a gentle algebra and its surface model, isomorphism classes of objects in $\mathcal{D}^b(\Lambda_n)$ can be fully understood in terms of curves on a torus with $n$ boundary components and marked points on its boundary.  The surface model of $\mathcal{D}^b(\Lambda_n)$  allows us to rephrase many problems about the categories $\Perf(C_n)$ and $\mathcal{D}^b(C_n)$ as considerably simpler problems about curves on a torus.\medskip

\noindent Throughout this introduction we assume that $\Bbbk$ is an algebraically closed field and that $n \geq 1$ is any natural number.\medskip
	
\noindent The first result of this paper provides a geometric interpretation of spherical objects over $C_n$ as closed curves on a punctured torus.

	\begin{Introthm}[Theorem \ref{TheoremClassificationSphericalObjects}]\label{IntroTheoremClassificationSphericalObjects}
	
		There exists a bijection between isomorphism classes of spherical objects on $C_n$, up to shift, and pairs $([\gamma], \lambda)$, where $\lambda \in \Bbbk \!\setminus\! \{0\}$ and $[\gamma]$ is the homotopy class of an unoriented, non-separating simple loop $\gamma$ on the $n$-punctured torus, i.e.\ $\gamma$ is embedded into the $n$-punctured torus in such a way that its complement is connected.	
	\end{Introthm} 
	
\noindent Theorem \ref{IntroTheoremClassificationSphericalObjects} is a special case of a more general correspondence between indecomposable objects of $\mathcal{D}^b(C_n)$ and curves on the $n$-punctured torus. Using the Verdier localization approach from \cite{LekiliPolishchukAuslanderOrders} we extend the topological description of $\mathcal{D}^b(C_n)$ to morphisms (Theorem \ref{TheoremGeometricDescriptionMorphismPuncturedCase}), compositions of morphisms (Remark \ref{RemarkCompositions}) and certain mapping cones (Remark \ref{RemarkMappingConesVerdierQuotient}). In fact, the topological model of $\mathcal{D}^b(C_n)$ presented here can be generalized to similar Verdier quotients of derived categories of other gentle algebras which might be of independent interest, see Remark \ref{RemarkGeneralizations}.\medskip
	
\noindent  The second result describes the loops of simple vector bundles and provides a geometric interpretation of their ranks.  To simplify notation in the next theorem, we identify the torus with the usual quotient $\mathbb{R}^2/\mathbb{Z}^2\cong S^1 \times S^1$ and refer to the first coordinate of the product as the \textit{latitudinal coordinate}.
	
		\begin{Introthm}[Theorem \ref{TheoremImagesOfSimpleVectorBundles}]\label{IntroTheoremSimpleVectorBundles}
		Under the bijection in Theorem \ref{IntroTheoremClassificationSphericalObjects}, a non-separating simple  loop $\gamma$ on a punctured torus corresponds to a simple vector bundle if and only if $\gamma$ is homotopic to a smooth loop such that the latitudinal coordinate of its derivative is nowhere vanishing. The rank of the vector bundle agrees with the number of full turns of $\gamma$ in latitudinal direction.
	\end{Introthm}

\noindent In fact, by dropping the assumption on simplicity on vector bundles and curves, we describe all homotopy classes of curves which represent the image of a vector bundle over $C_n$ and provide an easy geometric interpretation of rank and multi-degree on the surface, c.f.\ Theorem \ref{TheoremImagesOfSimpleVectorBundles}. The representation of vector bundles as curves allows us to give an alternative proof of the result by Bodnarchuk, Drozd and Greuel \cite{BodnarchukDrozdGreuel}  which states that the isomorphism class of any simple vector bundle over $C_n$ is uniquely determined by its rank, its multi-degree and its determinant (Proposition \ref{PropositionNewProof}).\medskip
	
	\noindent Our next result gives a positive answer to the question of Polishchuk:

	\begin{Introthm}[Theorem \ref{TheoremTransitivitySphericalOjects}]\label{IntroTheoremTransitivitySphericalOjects}The group of auto-equivalences of $\mathcal{D}^b(C_n)$ acts transitively on the set of isomorphism classes of spherical objects in $\Perf(C_n)$.
	\end{Introthm}
	
	\noindent The proof of Theorem \ref{IntroTheoremTransitivitySphericalOjects} exploits the relationship between spherical twists in the derived category of a gentle algebra and Dehn twists of its surface.\medskip

\noindent The following Theorem describes the group of auto-equivalences of any cycle.
\begin{Introthm}[Corollary \ref{CorollaryAutGroup}]\label{IntroTheoremAutoGroup}
If $n \geq 2$, then the group of auto-equivalences of $\mathcal{D}^b(C_n)$ is an extension of the mapping class group  of the $n$-punctured torus and the group

\[
 \left(\Bbbk^{\times}\right)^{n} \times \mathbb{Z}  \times \Pic^{\mathbb{0}}(C_n),
\]

\noindent where $\Pic^{\mathbb{0}}(C_n)$ denotes the group of line bundles with vanishing multi-degree. Moreover, $\Aut(\mathcal{D}^b(C_1))$ is an extension of $\PSL_2(\mathbb{Z})$ and 
$\left(\mathbb{Z}_2 \ltimes  \left(\Bbbk^{\times}\right)^{n}\right) \times \mathbb{Z}  \times \Pic^{\mathbb{0}}(C_1)$.
\end{Introthm}

	\noindent Theorem \ref{IntroTheoremAutoGroup} extends earlier results by Burban und Kreussler \cite{BurbanKreusslerGenusOne} for the case $n=1$, see Remark \ref{RemarkBK}. Its proof relies on the topological model for $\mathcal{D}^b(C_n)$ and on the relationship between automorphisms of the arc complex of a punctured surface and its extended mapping class group as found by Irmak and McCarthy \cite{IrmakMcCarthy}. Similar techniques were used in \cite{OpperDerivedEquivalences}. We expect Theorem \ref{IntroTheoremAutoGroup} to generalize to certain Verdier quotients of arbitrary gentle algebras. For further details, the reader is referred to Remark \ref{RemarkGeneralizationsEquivalences}.
	
	 As an application of Theorem \ref{IntroTheoremAutoGroup}, we establish faithfulness of an action of the pure mapping class group of the punctured torus on $\mathcal{D}^b(C_n)$, see Theorem \ref{PropositionActionSplits}. The group action was constructed by Sibilla \cite{SibillaMappingClassGroupAction} who also conjectured its faithfulness.\medskip

	\noindent The final result refines the classification of simple vector bundles obtained by Burban, Drozd and Greuel \cite{BurbanDrozdGreuel}. They show that a simple vector bundle is equivalent to the datum of a non-zero scalar and a cyclic integer sequence which satisfies certain strong constraints. The vector bundle can be reconstructed explicitly from this data. However, while the entries in the cyclic sequence can be easily determined from the constraints in \cite{BurbanDrozdGreuel}, it seems that the order in which they appear was unknown apart from the case $n=1$, see \cite{BurbanStableBundles}. The final theorem provides a closed formula for this sequence. For simplicity it is phrased for non-negative multi-degrees.

	\begin{Introthm}[Corollary \ref{CorollarySimpleVCB}]\label{IntroThmPullback}Let $\mathcal{L}$ be a simple vector bundle on $C_n$ of rank $r$ and multi-degree $(d_1, \dots, d_n)$ with $d_i \geq 0$. Then, the cyclic sequence of $\mathcal{L}$ is given by the sequence of cardinalities $r\mathbb{Z} \cap (d,d']$, where $d$ and $d'$ are any two consecutive entries in the sequence

		\[
		0, d_1, d_1+d_2, \, \dots \, , \sum_{i=1}^n{d_i}, \sum_{i=1}^n{d_i}+d_1, \, \dots, \, r \cdot \sum_{i=1}^n{d_i}.
		\]
		
	\end{Introthm}

\noindent The results of this paper allow us to give an independent proof of the classification of simple vector bundles on $C_n$ which can be found in Proposition \ref{PropositionNewProof}.

\subsection*{Organization of the paper} 

\noindent Section \ref{SectionCategoricalResolutionsCycles} contains a discussion of cycles of projective lines $C_n$, their tilted algebras $\Lambda_n$ and the relevant results from \cite{BurbanDrozdTilting} on the embedding $\Perf(C_n) \hookrightarrow \mathcal{D}^b(\Lambda_n)$.

 In the subsequent section we recall a few basic facts about spherical objects and their spherical twists.
 
  The third section  provides a detailed account of the surface model for the categorical resolution $\mathcal{D}^b(\Lambda_n)$ of $C_n$. In particular, we explain  the relationship between curves and their intersections on the surface model and their connection to indecomposable objects and morphisms between them. 
 
 After this preparation, Section \ref{SectionSphericalObjectsTori}  contains a first discussion of  spherical objects in $\mathcal{D}^b(\Lambda_n)$ but not yet the proof of Theorem \ref{IntroTheoremClassificationSphericalObjects} as this requires further results on the mapping class group of the surface model of $\mathcal{D}^b(\Lambda_n)$ and its connection to auto-equivalences of $\mathcal{D}^b(\Lambda_n)$.  The necessary background material for these is presented in Section \ref{SectionMappingClassGroup} after which Theorem \ref{IntroTheoremClassificationSphericalObjects} (classification of spherical objects) and Theorem \ref{IntroTheoremTransitivitySphericalOjects} (transitivity) are proved in Section \ref{SectionAllTheProofs}.

In Section \ref{SectionSurfaceModelCycle}, the surface model for $\mathcal{D}^b(\Lambda_n)$ is used to derive a similar model for a Verdier quotient of $\mathcal{D}^b(\Lambda_n)$ by the ``boundary objects'' which is equivalent to $\mathcal{D}^b(C_n)$. The description of the auto-equivalence group of $\mathcal{D}^b(C_n)$ as an extension of a mapping class group (Theorem \ref{IntroTheoremAutoGroup}) essentially follows from there by exploiting the relationship between the automorphism group of the \textit{arc complex} and the mapping class group of a surface.

In the final section we determine the curves on the surface model of $\mathcal{D}^b(C_n)$ which represent vector bundles. The pursued approach avoids further direct computations of the embedding $\Perf(C_n) \hookrightarrow \mathcal{D}^b(\Lambda_n)$ and rather exploits general properties of the class of vector bundles in combination with the surface model of $\mathcal{D}^b(\Lambda_n)$. The latter is re-phrased in a suitable combinatorial way akin to the cyclic sequences in \cite{BurbanDrozdGreuel} which describe (simple) vector bundles.

	\subsection*{Acknowledgements}The present work evolved from parts of my Ph.D.\ thesis.
	I like to thank Igor Burban, Wassilij Gnedin and Alexandra Zvonareva for many helpful discussions on the subject and for feedback on earlier drafts of this paper. I also like to thank the anonymous referee for their valuable feedback. While working on this project, I was supported by the DFG grant BU 1866/4-1, the Collaborative Research Centre on ``Symplectic Structures in Geometry, Algebra and Dynamics'' (CRC/TRR 191), by the Czech Science Foundation as part of the project ``Symmetries, dualities and
	approximations in derived algebraic geometry and representation theory'' (20-13778S) and the Charles University Research Center program (UNCE/SCI/022). I am further supported by the Primus grant PRIMUS/23/SCI/006 and was partially supported by the Cooperatio program of Charles University. \smallskip

	\addtocontents{toc}{\protect\setcounter{tocdepth}{0}}
	\section*{Conventions and general notation}
	\noindent We fix an algebraically closed field $\Bbbk$ and denote by $\Bbbk^{\times}$ its group of units. For a ringed space $(X, \mathcal{O}_X)$, we denote by $\Coh X$ its category of coherent $\mathcal{O}_X$-modules, by $\Perf(X)$ its category of perfect complexes, and by $\mathcal{D}^b(X)$ its bounded derived category of coherent sheaves. For a finite-dimensional $\Bbbk$-algebra $A$, we denote by $\mathcal{D}^b(A)$ the bounded derived category of finite-dimensional left $A$-modules. For objects $X , Y $ of a triangulated category $\mathcal{T}$ we write $\Hom^{\bullet}(X,Y)\coloneqq \bigoplus_{i \in \mathbb{Z}}{\Hom(X, Y[i])}$. Given a natural transformation $\eta: \mathcal{F} \rightarrow \mathcal{G}$ between functors $\mathcal{F}, \mathcal{G}: \mathcal{C} \rightarrow \mathcal{D}$ and an object $X \in \mathcal{C}$, we denote by $\eta_X$ the induced map $\mathcal{F}(X) \rightarrow \mathcal{G}(X)$.
	
	For an integer valued variable $x$ and integers $p, q \in \mathbb{Z}$, we often rewrite the condition $p \leq x \leq q$ as $x \in [p,q]$ and similar for all other types of intervals. For an arrow $\alpha$ in a quiver, we write $s(\alpha)$ (resp.\ $t(\alpha)$) for the source (resp.\ the target) of $\alpha$. Compositions of arrows are to be understood from right to left. We write $\mathbb{Z}_m$ for the cyclic group $\mathbb{Z}/{m\mathbb{Z}}$ and $\mathfrak{S}_m$ for the symmetric group of $m$ elements. Finally, for every $m \in \mathbb{N} \setminus \{0\}$,  $\mathbb{Z}^{\mathbb{Z}_m}$ denotes the set of functions $\mathbb{Z}_m \rightarrow \mathbb{Z}$, or, equivalently, the set of cyclic integer sequences of length $m$ .
	
	\addtocontents{toc}{\protect\setcounter{tocdepth}{1}}

	\section{Cycles of projective lines and their categorical resolutions}\label{SectionCategoricalResolutionsCycles}
	
	\noindent We recall the definition of a Kodaira cycle of projective lines and the relevant results about their categorical resolutions from \cite{BurbanDrozdTilting}. In what follows, let $n \geq 1$.
	\begin{definition}\label{DefinitionCycleOfProjectiveLines}An \textbf{$n$-cycle of projective lines} is a reduced rational projective curve $C_n$ of arithmetic genus $1$, that is, a union of $n$ copies of $\mathbb{P}^1$ glued together transversally in a configuration of type $\tilde{A}_{n-1}$.
	\end{definition}
	\noindent By definition, $C_n$ has $n$ irreducible components, henceforth denoted by $\mathbb{P}^1_i$ ($i \in \mathbb{Z}_n$), such that for all $i \in \mathbb{Z}_n$, $\mathbb{P}^1_i$ and $\mathbb{P}^1_{i+1}$ intersect in a nodal singularity. If $n > 1$, $\mathbb{P}^1_i  \cong \mathbb{P}^1$ for each $i \in \mathbb{Z}_n$. The curve $C_1$ is isomorphic to the {Weierstra\ss} nodal cubic and is depicted in Figure \ref{FigureE3} together with $C_2$ and $C_3$.\medskip

	\noindent Burban and Drozd constructed a fully faithful and exact functor \[\begin{tikzcd}\Perf(C_n) \arrow[hookrightarrow]{r} & \mathcal{D}^b(\Coh \mathbb{X}_n),\end{tikzcd}\]
	where $\mathbb{X}_n$ is a certain non-commutative curve. They proved that $\mathcal{D}^b(\Coh \mathbb{X}_n)$ contains a tilting complex $\mathcal{H}$. The opposite of its endomorphism algebra is isomorphic to the algebra $\Lambda_n$ which is by definition the quotient of  the path algebra of the quiver $Q(n)$ shown in Figure \ref{FigureQuiverLambdaN}  by the ideal generated by  the set of paths 
	\begin{equation}\label{EquationSetOfRelations}
	R:=\{b_ia_{i}, d_ic_i  \, | \, i\in [0,n)\}.
\end{equation} In other words, the composition of arrows in Figure \ref{FigureQuiverLambdaN} with different color vanishes.
	
	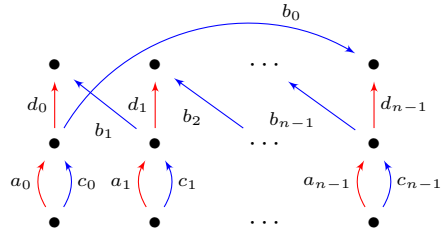
\begin{figure}[H]
		\centering
		\[\begin{tikzcd} \bullet  & \bullet  & \cdots & \bullet\\
		\bullet \arrow[red]{u}[black]{d_0} \arrow[bend angle=45, bend left, near end, blue]{urrr}[black]{b_0} & \bullet \arrow[red]{u}[black]{d_1} \arrow[near start, blue]{ul}[black]{b_1} & \cdots \arrow[blue]{ul}[black]{b_2} & \bullet \arrow[blue]{ul}[black]{b_{n-1}} \arrow[swap, red]{u}[black]{d_{n-1}} \\
		\bullet \arrow[bend left, red]{u}[black]{a_0} \arrow[ bend right, swap, blue]{u}[black]{c_0} & \bullet \arrow[bend left, red]{u}[black]{a_1} \arrow[bend right, swap, blue]{u}[black]{c_1} & \cdots  & \bullet \arrow[bend left, red]{u}[black]{a_{n-1}} \arrow[bend right, swap, blue]{u}[black]{c_{n-1}} \end{tikzcd}\]
		\caption{The underlying quiver $Q(n)$ of the algebra $\Lambda_n$. Arrows of different colors compose to zero in $\Lambda_n$, i.e.\ $b_i a_i=0=d_ic_i$ for all $i \in \mathbb{Z}_n$.}
		\label{FigureQuiverLambdaN}
	\end{figure}
	
	\noindent  For example, $\Lambda_1$ is the quotient of the path algebra of the quiver on the left in Figure \ref{FigureDoubleKronecker} by the ideal $(ba, dc)$ while the quiver of $\Lambda_3$ is shown on the right hand side .

	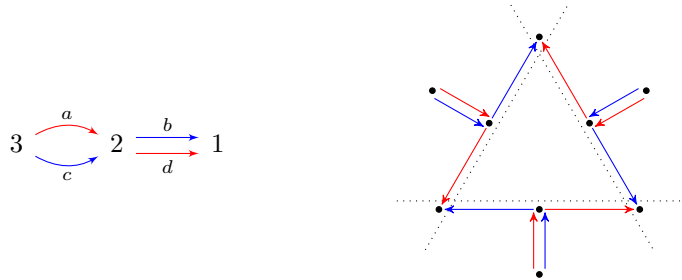
\begin{figure}[H]
		\centering
		
\begin{displaymath}\arraycolsep=3em\begin{array}{cc}
	{\begin{tikzcd}[ampersand replacement=\&] 3 \arrow[bend left, red]{r}[black]{a} \arrow[swap, bend right, blue]{r}[black]{c} \& 2  \arrow[yshift=0.7ex, blue]{r}[black]{b} \arrow[swap, yshift=-0.7ex, red]{r}[black]{d} \& 1 \end{tikzcd}} & {\begin{tikzpicture}[baseline=0.65cm, scale=1.5]


\draw[dotted] ({0.75-0.5*cos(0)},{-0.5*sin(0)})--({0.75+2*cos(0)},{2*sin(0)});
\draw[dotted] ({0.75-0.5*cos(60)},{-0.5*sin(60)})--({0.75+2*cos(60)},{2*sin(60)});

\draw[dotted] ({0.75+1.5*cos(60)-0.5*cos(300)},{1.5*sin(60)-0.5*sin(300)})--({0.75+1.5*cos(60)+2*cos(300)},{1.5*sin(60)+2*sin(300)});


\filldraw ({1.5+sqrt(3)*0.25*cos(150)+0.65*cos(150)},{sqrt(3)*0.25+sqrt(3)*0.25*sin(150)+0.65*sin(150)}) circle (.7pt);

\draw [<-, red] ({1.5+sqrt(3)*0.25*cos(150)+0.1*cos(150)+0.05*cos(60)},{sqrt(3)*0.25+sqrt(3)*0.25*sin(150)+0.1*sin(150)+0.05*sin(60)})--({1.5+sqrt(3)*0.25*cos(150)+0.5*cos(150)+0.1*cos(150)+0.05*cos(60)},{sqrt(3)*0.25+sqrt(3)*0.25*sin(150)+0.1*sin(150)+0.5*sin(150)+0.05*sin(60)});

\draw [<-, blue] ({1.5+sqrt(3)*0.25*cos(150)+0.1*cos(150)-0.05*cos(60)},{sqrt(3)*0.25+sqrt(3)*0.25*sin(150)+0.1*sin(150)-0.05*sin(60)})--({1.5+sqrt(3)*0.25*cos(150)+0.5*cos(150)+0.1*cos(150)-0.05*cos(60)},{sqrt(3)*0.25+sqrt(3)*0.25*sin(150)+0.1*sin(150)+0.5*sin(150)-0.05*sin(60)});

\filldraw ({1.5+sqrt(3)*0.25*cos(150)+0.075*cos(150)},{sqrt(3)*0.25+sqrt(3)*0.25*sin(150)+0.075*sin(150)}) circle (.7pt);

\draw [->, red] ({1.5+sqrt(3)*0.25*cos(150)+0.075*cos(150)+0.05*cos(240)},{sqrt(3)*0.25+sqrt(3)*0.25*sin(150)+0.075*sin(150)+0.05*sin(240)})--({1.5+sqrt(3)*0.25*cos(150)+0.075*cos(150)+sqrt(3)*0.48*cos(240)},{sqrt(3)*0.25+sqrt(3)*0.25*sin(150)+0.075*sin(150)+sqrt(3)*0.48*sin(240)});

\filldraw ({1.5+sqrt(3)*0.587*cos(210)},{sqrt(3)*0.25+sqrt(3)*0.587*sin(210)}) circle (.7pt);

\draw [->, blue] ({1.5+sqrt(3)*0.25*cos(150)+0.075*cos(150)-0.05*cos(240)},{sqrt(3)*0.25+sqrt(3)*0.25*sin(150)+0.075*sin(150)-0.05*sin(240)})--({1.5+sqrt(3)*0.25*cos(150)+0.075*cos(150)-sqrt(3)*0.48*cos(240)},{sqrt(3)*0.25+sqrt(3)*0.25*sin(150)+0.075*sin(150)-sqrt(3)*0.48*sin(240)});


\filldraw ({1.5+sqrt(3)*0.25*cos(150+120)+0.65*cos(150+120)},{sqrt(3)*0.25+sqrt(3)*0.25*sin(150+120)+0.65*sin(150+120)}) circle (.7pt);

\draw [<-, red] ({1.5+sqrt(3)*0.25*cos(150+120)+0.1*cos(150+120)+0.05*cos(60+120)},{sqrt(3)*0.25+sqrt(3)*0.25*sin(150+120)+0.1*sin(150+120)+0.05*sin(60+120)})--({1.5+sqrt(3)*0.25*cos(150+120)+0.5*cos(150+120)+0.1*cos(150+120)+0.05*cos(60+120)},{sqrt(3)*0.25+sqrt(3)*0.25*sin(150+120)+0.1*sin(150+120)+0.5*sin(150+120)+0.05*sin(60+120)});

\draw [<-, blue] ({1.5+sqrt(3)*0.25*cos(150+120)+0.1*cos(150+120)-0.05*cos(60+120)},{sqrt(3)*0.25+sqrt(3)*0.25*sin(150+120)+0.1*sin(150+120)-0.05*sin(60+120)})--({1.5+sqrt(3)*0.25*cos(150+120)+0.5*cos(150+120)+0.1*cos(150+120)-0.05*cos(60+120)},{sqrt(3)*0.25+sqrt(3)*0.25*sin(150+120)+0.1*sin(150+120)+0.5*sin(150+120)-0.05*sin(60+120)});

\filldraw ({1.5+sqrt(3)*0.25*cos(150+120)+0.075*cos(150+120)},{sqrt(3)*0.25+sqrt(3)*0.25*sin(150+120)+0.075*sin(150+120)}) circle (.7pt);

\draw [->, red] ({1.5+sqrt(3)*0.25*cos(150+120)+0.075*cos(150+120)+0.05*cos(240+120)},{sqrt(3)*0.25+sqrt(3)*0.25*sin(150+120)+0.075*sin(150+120)+0.05*sin(240+120)})--({1.5+sqrt(3)*0.25*cos(150+120)+0.075*cos(150+120)+sqrt(3)*0.48*cos(240+120)},{sqrt(3)*0.25+sqrt(3)*0.25*sin(150+120)+0.075*sin(150+120)+sqrt(3)*0.48*sin(240+120)});

\filldraw ({1.5+sqrt(3)*0.587*cos(210+120)},{sqrt(3)*0.25+sqrt(3)*0.587*sin(210+120)}) circle (.7pt);

\draw [->, blue] ({1.5+sqrt(3)*0.25*cos(150+120)+0.075*cos(150+120)-0.05*cos(240+120)},{sqrt(3)*0.25+sqrt(3)*0.25*sin(150+120)+0.075*sin(150+120)-0.05*sin(240+120)})--({1.5+sqrt(3)*0.25*cos(150+120)+0.075*cos(150+120)-sqrt(3)*0.48*cos(240+120)},{sqrt(3)*0.25+sqrt(3)*0.25*sin(150+120)+0.075*sin(150+120)-sqrt(3)*0.48*sin(240+120)});


\filldraw ({1.5+sqrt(3)*0.25*cos(150+240)+0.65*cos(150+240)},{sqrt(3)*0.25+sqrt(3)*0.25*sin(150+240)+0.65*sin(150+240)}) circle (.7pt);

\draw [<-, red] ({1.5+sqrt(3)*0.25*cos(150+240)+0.1*cos(150+240)+0.05*cos(60+240)},{sqrt(3)*0.25+sqrt(3)*0.25*sin(150+240)+0.1*sin(150+240)+0.05*sin(60+240)})--({1.5+sqrt(3)*0.25*cos(150+240)+0.5*cos(150+240)+0.1*cos(150+240)+0.05*cos(60+240)},{sqrt(3)*0.25+sqrt(3)*0.25*sin(150+240)+0.1*sin(150+240)+0.5*sin(150+240)+0.05*sin(60+240)});

\draw [<-, blue] ({1.5+sqrt(3)*0.25*cos(150+240)+0.1*cos(150+240)-0.05*cos(60+240)},{sqrt(3)*0.25+sqrt(3)*0.25*sin(150+240)+0.1*sin(150+240)-0.05*sin(60+240)})--({1.5+sqrt(3)*0.25*cos(150+240)+0.5*cos(150+240)+0.1*cos(150+240)-0.05*cos(60+240)},{sqrt(3)*0.25+sqrt(3)*0.25*sin(150+240)+0.1*sin(150+240)+0.5*sin(150+240)-0.05*sin(60+240)});

\filldraw ({1.5+sqrt(3)*0.25*cos(150+240)+0.075*cos(150+240)},{sqrt(3)*0.25+sqrt(3)*0.25*sin(150+240)+0.075*sin(150+240)}) circle (.7pt);

\draw [->, red] ({1.5+sqrt(3)*0.25*cos(150+240)+0.075*cos(150+240)+0.05*cos(240+240)},{sqrt(3)*0.25+sqrt(3)*0.25*sin(150+240)+0.075*sin(150+240)+0.05*sin(240+240)})--({1.5+sqrt(3)*0.25*cos(150+240)+0.075*cos(150+240)+sqrt(3)*0.48*cos(240+240)},{sqrt(3)*0.25+sqrt(3)*0.25*sin(150+240)+0.075*sin(150+240)+sqrt(3)*0.48*sin(240+240)});

\filldraw ({1.5+sqrt(3)*0.587*cos(210+240)},{sqrt(3)*0.25+sqrt(3)*0.587*sin(210+240)}) circle (.7pt);

\draw [->, blue] ({1.5+sqrt(3)*0.25*cos(150+240)+0.075*cos(150+240)-0.05*cos(240+240)},{sqrt(3)*0.25+sqrt(3)*0.25*sin(150+240)+0.075*sin(150+240)-0.05*sin(240+240)})--({1.5+sqrt(3)*0.25*cos(150+240)+0.075*cos(150+240)-sqrt(3)*0.48*cos(240+240)},{sqrt(3)*0.25+sqrt(3)*0.25*sin(150+240)+0.075*sin(150+240)-sqrt(3)*0.48*sin(240+240)});

\end{tikzpicture}}
	\end{array}
	\end{displaymath}		
		\caption{The quivers $Q(1)$ (left) and $Q(3)$ (right).}
		\label{FigureDoubleKronecker}
	\end{figure}

\noindent For every $n \geq 1$. the algebra $\Lambda_n$ is gentle in the sense of \cite{AssemSkowronski} and has global dimension $2$. In particular, there is an Auslander-Reiten translation $\tau: \mathcal{D}^b(\Lambda_n) \rightarrow \mathcal{D}^b(\Lambda_n)$ and $\tau \cong \nu[-1]$, where $\nu$ denotes the left derived  Nakayama functor.\medskip
	
	\noindent It follows from the preceding results that there exists an embedding of triangulated categories 
	\[\begin{tikzcd}\mathbb{F}: \Perf(C_n) \arrow[hook]{r}& \mathcal{D}^b(\Lambda_n).\end{tikzcd}\]
	
	\noindent Burban and Drozd provided a characterization of its essential image $\Img \mathbb{F}$. 
	\begin{thm}[Corollary 3.5 and Corollary 6.3, \cite{BurbanDrozdTilting}]\label{TheoremBurbanDrozdImage}
		The category $\Img \mathbb{F}$ is the full subcategory of $\tau$-invariant objects, i.e.\ all objects $X \in \mathcal{D}^b(\Lambda_n)$ such that $X \cong \tau X$.
	\end{thm}
	
	\noindent They further computed the image of the Jacobian $\Pic^0(C_1) \cong \Bbbk^{\times}$ (Proposition 7.4, \cite{BurbanDrozdTilting}), i.e.\ the line bundles of degree $0$, as well as the images of the skyscraper sheaves $\Bbbk(x)$ of smooth points $x \in C_1$ (Proposition 7.2, \cite{BurbanDrozdTilting}).
	By twisting their tilting object with a line bundle of degree $1$, the functor $\mathbb{F}$ identifies isomorphism classes of line bundles of degree $0$ with the isomorphism classes of the following family $(\mathcal{O}(\lambda))_{\lambda \in \Bbbk^{\times}}$ of two-term complexes\footnote{This observation was communicated to us by Igor Burban.}.

	\begin{displaymath} \label{EquationComplexStructureSheaf}\mathcal{O}(\lambda)= \begin{tikzcd}[ampersand replacement=\&]\cdots \arrow{r} \& 0 \arrow{r} \& P_1 \arrow{r}{b+\lambda d} \& P_2 \arrow{r} \& 0 \arrow{r} \& \cdots, \end{tikzcd}\end{displaymath}
	\noindent where $P_i$ denotes the indecomposable projective $\Lambda_1$-module associated to the vertex $i$ in $Q(1)$. The set $\left\{\Bbbk(x) \, | \, x \in C_1 \text{ is smooth} \right\}$ is identified with the set of complexes $\{\Bbbk(\lambda) \, | \, \lambda \in \Bbbk^{\times} \}$, where
	
	\begin{displaymath}\label{EquationComplexSkyscraperSheaf} \Bbbk(\lambda) = \begin{tikzcd}[ampersand replacement=\&]\cdots \arrow{r} \& 0 \arrow{r} \& P_2 \arrow{r}{a+\lambda c} \& P_3 \arrow{r} \& 0 \arrow{r} \& \cdots. \end{tikzcd}\end{displaymath}
	
	\noindent Both $\mathcal{O}(\lambda)$ and $\Bbbk(\lambda)$ are concentrated in degrees $-1$ and $0$ and are quasi-isomorphic to $\Lambda_1$-modules. We notice that, up to shift,  they are completely encoded in the quivers in Figure \ref{FigureQuivers}.
	
	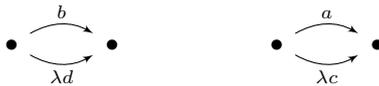
\begin{figure}[H]
		
		\begin{tikzcd}[ampersand replacement=\&]
		\bullet \arrow[bend left]{r}{b} \arrow[bend right, swap]{r}{\lambda d} \& \bullet \& \&  \bullet \arrow[bend left]{r}{a} \arrow[bend right, swap]{r}{\lambda c} \& \bullet 
		\end{tikzcd}
		
		\caption{The quiver description of the complexes $\mathcal{O}(\lambda)$ (left) and $\Bbbk(\lambda)$ (right).} \label{FigureQuivers}
	\end{figure}
	
	\noindent Every vertex of the quivers represents an indecomposable projective $\Lambda_1$-module. Arrows connect projective modules in consecutive cohomological degrees and their labels describe the differentials of the complexes.\medskip
	
	\noindent  The complexes $\mathcal{O}(\lambda)$ and $\Bbbk(\lambda)$ have natural generalizations to objects in $\mathcal{D}^b(\Lambda_n)$. \noindent Their description in terms of quivers is found in Figure \ref{FigurePicardGroupComplex}:

	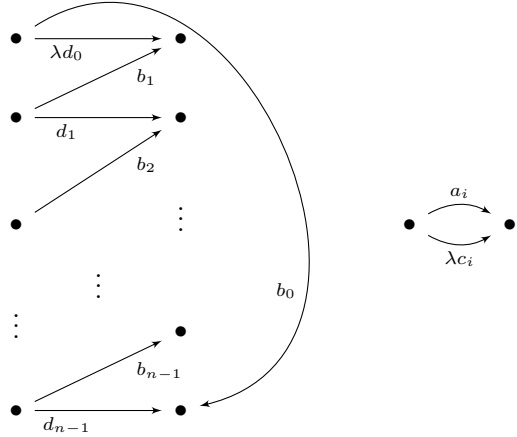
\begin{figure}
		\begin{displaymath}
		\begin{tikzcd}[scale=0.8, ampersand replacement=\&]
		\bullet \arrow[bend angle=100, bend left, swap, min distance=4cm, near end]{ddddrr}{b_0} \arrow[near start, swap]{rr}{d_{n-1}} \& \& \bullet \\
		\bullet \arrow[near end, swap]{urr}{ b_{n-1}} \arrow[near start, swap]{rr}{d_{n-2}} \& \& \bullet \\
		\bullet \arrow[near end, swap]{urr}{ b_{n-2}}   \&  \& \vdots \& \& \& 	\bullet \arrow[bend left]{r}{ a_i} \arrow[bend right, swap]{r}{\lambda c_i} \& \bullet  \\
		\vdots \arrow[phantom]{urr}{\vdots} \&  \& \bullet \\
		\bullet \arrow[phantom, bend angle=100, bend left, swap, min distance=4cm, near end]{uuuurr}{\phantom{b_0}} \arrow[near end, swap]{urr}{b_1} \arrow[near start, swap]{rr}{\lambda d_0} \& \& \bullet
		\end{tikzcd}
		\end{displaymath}
		
		\caption{The images under $\mathbb{F}$ of $\Pic^{\mathbb{0}}(C_n)$ (left) and $\Bbbk(x)$ of a smooth point $x \in \mathbb{P}^1_i$ (right).}  \label{FigurePicardGroupComplex}
	\end{figure}
	\noindent Every vertex in Figure \ref{FigurePicardGroupComplex}  is the placeholder for an indecomposable projective $\Lambda_n$-module and all vertices on the same vertical line represent the direct summands in a fixed cohomological degree which is either $-1$ or $0$ and which increases along arrows. The arrows describe the action of the differential on the complex in that an arrow $\bullet \xrightarrow{\alpha} \bullet$ contributes a direct summand $P_{t(\alpha)}$ in degree $-1$ and $P_{s(\alpha)}$ in degree $0$ and the induced map $\alpha: P_{t(\alpha)} \rightarrow P_{s(\alpha)}$ as a component to the differential. In other words, the complexes are given by ``folding up'' the quivers in Figure \ref{FigurePicardGroupComplex} in vertical direction. Explicitly, the complex of the quiver on the left hand side of Figure \ref{FigurePicardGroupComplex} is given by
	\begin{displaymath}
	\begin{tikzcd}[ampersand replacement=\&]
\cdots \arrow{r} \& 0 \arrow{r} \&	\bigoplus_{i=0}^{n-1}P_{t(d_i)}  \arrow{rrr}{\left(\begin{smallmatrix}\lambda d_0 \\ b_1 & d_1  \\ & b_2  & \ddots  \\ & & \ddots & d_{n-1} \\  & & &  b_0 \end{smallmatrix}\right)} \&\&\& \bigoplus_{i=0}^{n-1}P_{s(d_i)} \arrow{r} \& 0 \arrow{r} \& \cdots.
	\end{tikzcd}
	\end{displaymath}
	
\noindent By abuse of notation, we denote the resulting complexes by $\mathcal{O}(\lambda)$ (left in Figure \ref{FigurePicardGroupComplex}) and $\Bbbk(i, \lambda)$ (right in Figure \ref{FigurePicardGroupComplex}) respectively. Again, these complexes are quasi-isomorphic to $\Lambda_n$-modules. The complex of $\Bbbk(i, \lambda)$ is obtained from the complex of $\Bbbk(\lambda)$ by replacing $a$ and $\lambda c$ with $a_i$ and $\lambda c_i$.\medskip

\noindent It was communicated to us by Igor Burban that the computations in \cite{BurbanDrozdTilting} can be generalized to arbitrary $n \in \mathbb{N}$. We denote by $\Pic^{\mathbb{0}}(C_n)$ the set of all line bundles on $C_n$ of multi-degree\footnote{Given a vector bundle $\mathcal{E}$ over $C_n$ and a normalization map $\pi$, the entry $d_i$ in the multi-degree $(d_i)_{i \in \mathbb{Z}_n}$ of $\mathcal{E}$ is the degree of the restricted pullback bundle $\pi^{\ast}(\mathcal{E})|_{\pi^{-1}(\mathbb{P}_i^{1})}$.} $\mathbb{0}=(0, \dots, 0) \in \mathbb{Z}^n$. As in the case $n=1$, $\Pic^{\mathbb{0}}(C_n) \cong  \Bbbk$.
	
	\begin{thm}[Burban, Burban-Drozd \cite{BurbanDrozdTilting}]\label{TheoremImages}
		\ \begin{enumerate}
			\setlength\itemsep{1ex}

			\item The essential image of $\Pic^{\mathbb{0}}(C_n)$ under $\mathbb{F}$ consists of the isomorphism classes of the complexes $\mathcal{O}(\lambda)$, where $\lambda \in \Bbbk^{\times}$.
			
			\item  The essential image of the skyscraper sheaves of smooth points $x \in \mathbb{P}^1_i$ under $\mathbb{F}$ consists of the isomorphism classes of the complexes $\Bbbk(i, \lambda)$, where $\lambda \in \Bbbk^{\times}$.
		\end{enumerate}
		
	\end{thm}
	
	\noindent The proof of Theorem \ref{TheoremImages} was explained to us by Igor Burban and can be found in the Appendix.
	
 	\ \smallskip

	\section{Serre functors, spherical objects and spherical twists}\label{SectionSphericalObjectSphericalTwists}
	\noindent In this section, we recall the definition of a \emph{spherical object} in a triangulated category and its associated \textit{spherical twist}. Spherical objects and spherical twists were first introduced in \cite{SeidelThomas}.\medskip
	
	\noindent Throughout this section we fix a $\Bbbk$-linear triangulated category $\mathcal{T}$.
	
	\begin{definition}Let $X \in \mathcal{T}$ be such that $\Hom_{\mathcal{T}}^{\bullet}(X,Y)$ is finite dimensional for all $Y \in \mathcal{T}$. Then, a \textbf{Serre dual} of $X$ is an object $\mathcal{S}(X) \in \mathcal{T}$  such that there exists a $\Bbbk$-linear isomorphism of functors
		$\Hom_{\mathcal{T}}(X,-) \cong \Hom_{\mathcal{T}}(-, \mathcal{S}(X))^*$, where $(-)^*$ denotes the duality over the ground field $\Bbbk$. 
	\end{definition}
	\noindent By Yoneda's lemma, the Serre dual of an object is unique up to unique isomorphism, if it exists, and it follows that the mapping $X \mapsto \mathcal{S}(X)$ is functorial on the full subcategory $\mathcal{C}$ spanned by the objects it is defined on. If the functor $\mathcal{S}: \mathcal{C} \rightarrow \mathcal{T}$ restricts to an endo-functor of $\mathcal{C}$ and is  essentially surjective, then any such functor is called a \textbf{Serre functor} of $\mathcal{C}$ relative to $\mathcal{T}$. For purely formal reasons, $\mathcal{C}$ is triangulated and $\mathcal{S}$ is a $\Bbbk$-linear triangle equivalence which commutes with every $\Bbbk$-linear triangle equivalence of $\mathcal{T}$ or $\mathcal{C}$ up to natural isomorphism, see \cite[\S I.1]{ReitenVandenBergh}.
	\begin{exa}\label{ExampleSerreFunctorsCycles}
		The categories $\Perf(C_n)$ and $\mathcal{D}^b(\Lambda_n)$ have Serre functors. In the former case, it is well-known that the left derived tensor product $- \otimes^{\mathbb{L}} \omega[1]$ is a Serre functor of $\Perf(C_n)$ relative to $\mathcal{D}^b(C_n)$, where $\omega$ denotes the canonical sheaf. In fact, $\omega$ is trivial, showing that the Serre functor is isomorphic to the shift functor $[1]$. Since $\Lambda_n$ has finite global dimension, the left derived functor of the Nakayama functor $\nu=\left(\Hom_{\Lambda_n}(-, {_{\Lambda_n}}\Lambda_n)\right)^{\ast}: \Lambda_n \module \longrightarrow \Lambda_n \module$ is a Serre functor of $\mathcal{D}^b(\Lambda_n)=\mathcal{T}=\mathcal{C}$.
	\end{exa}
	
	\begin{definition}Let $d \in \mathbb{N} \setminus \{0\}$. An object $X \in \mathcal{T}$ is \textbf{$d$-Calabi-Yau} if $X[d]$ is a Serre dual of $X$. A $d$-Calabi-Yau object $X \in \mathcal{T}$ is \textbf{$d$-spherical} if there exists an isomorphism of graded rings $\Hom^{\bullet}(X, X)\cong \Bbbk[z]/(z^2)$, where $\deg z=d$.
	\end{definition}
	\noindent It follows from Example \ref{ExampleSerreFunctorsCycles} that every object $X \in \Perf(C_n)$ is $1$-Calabi-Yau. In particular, all of its spherical objects are $1$-spherical.

	\begin{exa}\label{ExampleSphericalObjects}\ \smallskip
		\begin{itemize}
			\setlength\itemsep{1ex}
			\item[1)] Every simple vector bundle $\mathcal{L} \in \Perf(C_n)$ is $1$-spherical. As a sheaf it has no negative self-extensions. Moreover, $\End(\mathcal{L},\mathcal{L})\cong \Bbbk$ and by {Serre} duality,
			\[\Hom(\mathcal{L},\mathcal{L}[i])\cong \Hom(\mathcal{L},\mathcal{L}[1-i])^* \cong \begin{cases}\Bbbk, & \text{if }i=0, 1; \\ 0 & \text{if }i \geq 2.\end{cases}\]
			In particular, all line bundles on $C_n$ are spherical. 
			\item[2)] Let $x \in C_n$ be smooth. By a similar argument as before, the associated skyscraper sheaf is a $1$-spherical object in $\Perf(C_n)$.
		\end{itemize}
	\end{exa}
	\noindent The following is well-known.
	
	\begin{lem}
		Let $X \in \mathcal{T}$ be spherical. Then, $X$ is indecomposable. 
	\end{lem}
	\begin{proof}
		By assumption, $\Hom(X, X)$ is local and hence $X$ is indecomposable.	
	\end{proof}
	
	\subsubsection{Spherical twists}\label{SectionSphericalTwists} \ \medskip
	
	\noindent Suppose that $\mathcal{T}$ admits a DG-enhancement. The assumption is satisfied for all triangulated categories which we consider in this paper as shown by Lunts and Schn\"{u}rer \cite{LuntsSchnuerer}. Under this condition, it was shown by Seidel and Thomas \cite{SeidelThomas} (see also \cite{HocheneggerKalckPloog}) that every spherical object $X \in \mathcal{T}$ gives rise to an auto-equivalence $T_X$ of $\mathcal{T}$, called a \textbf{spherical twist}. By definition, for every $Y \in \mathcal{T}$,
	the object $T_X(Y)$ sits in a distinguished triangle of the form

	\begin{equation}\label{EvaluationTriangle}
	\begin{tikzcd}
	{\Hom^{\bullet}(X,Y) \otimes_k X} \arrow{rr}{\textbf{ev}}  & & Y \arrow{rr}{} & & T_X(Y) \arrow{rr}{} & & \left( \Hom^{\bullet}(X,Y) \otimes_k X\right) [1],
	\end{tikzcd}\end{equation}
	where $\textbf{ev}$ denotes the evaluation map, i.e.\ the counit of the adjunction between the Hom-functor and the tensor product. Moreover, the morphisms $Y \rightarrow T_X(Y)$ define a natural transformation $\operatorname{Id}_{\mathcal{T}} \rightarrow  T_X$. 
	
	\begin{rem}\label{RemarkTwistFunctorSmoothPointTensorProduct} Let $x \in C_n$ be closed and smooth. In this case, the twist functor of $k(x)$ admits a more familiar description. As shown in (3.11) on page 68 of \cite{SeidelThomas}, there exists an isomorphism of functors
	$$ 
	        T_{\Bbbk(x)}({-}) \cong {-} \otimes^{\mathbb{L}} \mathcal{L}(x),
	$$
	 where $\mathcal{L}(x)$ denotes the line bundle associated with the divisor $x$. 
	\end{rem}
	
	\noindent Given a homogeneous basis $f_1, \dots, f_m$ of $\bigoplus_{i \in \mathbb{Z}}\Hom(X[i],Y)[-i]$, the triangle in \eqref{EvaluationTriangle} is isomorphic to a distinguished triangle
	
	\begin{displaymath}
	\begin{tikzcd}
	{\bigoplus_{i=1}^m{X}[n_i]} \arrow{rr}{\bigoplus_{i=1}^m{f_i}}  & & Y \arrow{rr}{} & & T_X(Y) \arrow{rr}{} & & \bigoplus_{i=1}^m{X}[n_i+1].
	\end{tikzcd}
	\end{displaymath}
	
	\noindent This shows that twist functors are compatible with embeddings in the following sense.
	\begin{cor}\label{CorollaryTwistFunctorsUnderEmbeddings} Let $\mathcal{T}'$ be a $\Bbbk$-linear triangulated category which admits a DG-enhancement and let $\mathbb{F}: \mathcal{T} \rightarrow \mathcal{T}'$ be a $\Bbbk$-linear, exact and fully faithful functor. Let $X \in \mathcal{T}$ be spherical and assume that $\mathbb{F}(X)$ is spherical. Then, for all $Y \in \mathcal{T}$,
		\[ \mathbb{F}\circ T_X (Y) \cong T_{\mathbb{F}(X)} \circ \mathbb{F}(Y).\] 
		In particular, if $\mathbb{F}$ is an equivalence and $\mathbb{F}^{-1}$ a quasi-inverse of $\mathbb{F}$, then $\mathbb{F} \circ T_X \circ \mathbb{F}^{-1}(Y) \cong T_{\mathbb{F}(X)}(Y)$.
	\end{cor}
	\noindent The previous corollary allows us to analyze the twist functor of any  spherical object $X \in \Perf(C_n)$ by means of the twist along $\mathbb{F}(X) \in \mathcal{D}^b(\Lambda_n)$.\medskip
	
	\noindent The statement of the following lemma is well-known.
	\begin{lem}\label{LemmaImageSphericalObjectsUnderItsTwist}
	   If $X \in \mathcal{T}$ is $d$-spherical, then $T_X(X) \cong X[1-d]$. If $\Hom^{\bullet}(X, Y)=0$, then $T_X(Y) \cong Y$.
	
	\end{lem}
	\begin{proof}
If $f: X \rightarrow X[d]$ is non-zero, then $T_X(X)$ is isomorphic to the mapping cone of $\operatorname{Id}_X[-d] \oplus f: X[-d] \oplus X \rightarrow X$. There exists a map of distinguished triangles 
	$$\begin{tikzcd}[ampersand replacement=\&, column sep=5em]
            X[-d] \oplus X \arrow{r}{\left(\begin{smallmatrix} f & \operatorname{Id}_X \end{smallmatrix}\right)}  \arrow{dd}[rotate=-90, yshift=+.5em, xshift=-.6em]{\simeq}[swap, xshift=-.3em]{\left(\begin{smallmatrix} \operatorname{Id}_{X[-d]} & 0 \\ f & \operatorname{Id}_X \end{smallmatrix}\right)} \& X \arrow{r} \arrow[equal]{dd} \& T_X(X) \arrow{dd} \arrow{r} \& \left(X[-d] \oplus X\right)[1] \arrow{dd} \\ \\
           X[-d] \oplus X \arrow{r}{\left(\begin{smallmatrix} 0 & \operatorname{Id}_X \end{smallmatrix}\right)} \& X \arrow{r} \& X[1-d] \arrow{r} \& \left(X[-d] \oplus X\right)[1].
         \end{tikzcd}$$
            
      \noindent It follows from the five lemma that all vertical arrows are isomorphisms. This proves the first assertion. If $\Hom^{\bullet}(X,Y)=0$, then $T_X(Y)$ is a mapping cone of the zero map $0 \rightarrow Y$ and hence isomorphic to $Y$.\end{proof}

	\section{The surface model of the categorical resolution}\label{SectionSurfaceModelLambda}
	
	\noindent We recall the relevant results from \cite{OpperPlamondonSchroll} about the surface model of a gentle algebra which we state in a simplified manner for the algebras $\Lambda_n$.\medskip
	
	\noindent The surface model of the algebra $\Lambda_n$ consists of a triple $(\mathbb{T}_n, \marked, \omega)$, where
	\begin{itemize}
		\setlength\itemsep{0.5em}
		
		\item $\mathbb{T}_n$ is a torus with $n$ boundary components and $\marked \subset \partial \mathbb{T}_n$ is a set of $2n$ marked points, two on each component, and
		
		\item $\omega$ is a function which attaches to any oriented loop $\gamma \subseteq \mathbb{T}_n$ an integer $\omega(\gamma)$ which we call the \textit{winding number of $\gamma$}.
		
	\end{itemize}
	
	\noindent  The triple $(\mathbb{T}_n, \marked, \omega)$ captures all essential information about the structure of the triangulated category $\mathcal{D}^b(\Lambda_n)$. Haiden, Katzarkov and Kontsevich showed that the derived category of any (graded) gentle algebra is equivalent to a partially Fukaya category of a compact, oriented surface with marked points which is equipped with a winding number function as above. The topological description of $\mathcal{D}^b(\Lambda_n)$ is an incarnation of this relationship.\medskip
	
	\noindent As we will recall over the course of this section, the geometric model serves as a dictionary to translate algebraic terms such as objects and morphisms to topological notions such as curves and their intersections.\medskip

	\noindent The function $\omega$ is derived from the datum of a \textit{line field} $\eta$ on $\mathbb{T}_n$. For a description of this connection the reader can consult \cite{OpperDerivedEquivalences} or \cite{AmiotPlamondonSchroll}. However, we do not require any familiarity with the concept of a line field throughout the paper and only use this fact in the proof of Lemma \ref{LemmaHopf} and Remark \ref{RemarkDoubleLoopOfAnArcHasVanishingWindingNumber}, where we provide the relevant information.\medskip
	
	\subsection{Tori, marked points and a lamination}\label{SectionToriMarkedPoints}
	
	\noindent For all integer valued points $(i, j) \in \mathbb{Z}^2 \subseteq \mathbb{R}^2$ let  $B_i^j \subseteq \mathbb{R}^2$ denote the open disc with radius $\frac{1}{4}$ and center $(i+\frac{1}{2}, j+\frac{1}{2})$. Denote by $\mathbb{T}_n$ the torus with $n$ removed open discs, i.e.\
	$\mathbb{T}_n$ is the quotient of $\mathbb{R}^2 \setminus B$, where
	
	\[B=\bigsqcup_{(i,j) \in \mathbb{Z}^2}{B_i^j},\]
	
	\noindent with respect to the equivalence relation generated by $(r,s)\sim(r,s+1)$ and $(r,s)\sim (r+n,s)$ for all $r, s \in \mathbb{R}$. In particular, $[0,n) \times [0,1)$ is a fundamental domain of the corresponding quotient map $\rho$. For convenience, we often refer to a point $x \in \mathbb{T}_n$ by a representative in $\mathbb{R}^2$.\medskip

	\noindent The set of marked points $\marked$ is given by
	
	\begin{displaymath}
	\marked=\left\{\left({i+\frac{1}{2},\frac{3}{4}}\right), \left(i+\frac{1}{2}, \frac{1}{4}\right) \, \Big| \, 0 \leq i < n \right\} \subseteq \partial \mathbb{T}_n.
	\end{displaymath}

	\noindent Frequently, we also consider the torus $\mathscr{T}^n$ with $n$ punctures and no boundary which we naturally regard as the quotient of $\mathbb{R}^2$ by the same relations $(r, s+1) \sim (r,s) \sim (r+n, s)$. The images of the centers of the discs $B_i^j$ form a set of $n$ distinct points in the quotient which we regard as interior marked points (``punctures'') and we consider $\mathbb{T}_n$ as being embedded into $\mathscr{T}^n$ in the natural way.\medskip

	\noindent  By design, the set $\marked \subset \partial\mathbb{T}_n$ is in bijection with the set of maximal admissible paths of the pair $(Q(n), R)$ which we defined in Figure \ref{FigureQuiverLambdaN} and \eqref{EquationSetOfRelations}. More precisely, for each $i \in [0,n)$, the bijection identifies the path $d_i a_i$ with the the point $(i+ \frac{1}{2},\frac{1}{4})$ and the path $b_{i} c_i$ with the point $(i-\frac{1}{2}, \frac{3}{4})$.\medskip
	
	\noindent As a result of its construction, $\mathbb{T}_n$ is further equipped with a collection of \textbf{laminates} which are embedded and pairwise disjoint paths $L_{x}$, one for each vertex $x$ of $Q=Q(n)$. The laminates are depicted in Figure \ref{FigureLambdaLamination}. We denote by $L \subseteq \mathbb{T}_n$ the union of all laminates. The complement $\mathbb{T}_n \setminus L$ is a disjoint union of $2n$ open $6$-gons as shown in Figure \ref{FigurePolygons}. The boundary of each polygon contains a unique marked point.
	\begin{figure}[H]
		\begin{displaymath}
		\begin{tikzpicture}[scale=3.25, hobby]
		\draw[dashed] (0,0)--(4,0);
		\draw[dashed] (0,0)--(0,1);
		\draw[dashed] (0,1)--(4,1);
		\draw[dashed] (4,1)--(4,0);

		\foreach \u in {0,1,3}
		{
			\filldraw (\u+0.5,0.625) circle (0.5pt);
			\filldraw (\u+0.5,0.375) circle (0.5pt);
			\draw (\u+0.5,0.5) circle (0.125);
			
			\draw[red, thick] plot  [line width=2,  tension=1] coordinates {   ({\u+0.5+0.125*cos(0)},{0.5+0.125*sin(0)}) ({\u+0.5+0.2*cos(60)},{0.5+0.2*sin(60)}) ({\u+0.5+0.3*cos(90)},{0.5+0.3*sin(90)}) ( (\u+0.5,1) };

			\draw[red, thick] plot  [line width=2,  tension=1] coordinates {   ({\u+0.5+0.125*cos(180)},{0.5+0.125*sin(180)}) ({\u+0.5+0.2*cos(240)},{0.5+0.2*sin(240)}) ({\u+0.5+0.3*cos(270)},{0.5+0.3*sin(270)}) ( (\u+0.5,0) };
			
			\draw[blue, thick] plot  [line width=2,  tension=1] coordinates {   ({\u+0.5+0.125*cos(-30)},{0.5+0.125*sin(-30)}) (\u+0.85,0.625) (\u+1,1) };
			
			\draw[blue, thick] plot  [line width=2,  tension=1] coordinates {   ({\u-1+1+0.5+0.125*cos(150)},{0.5+0.125*sin(150)}) (\u-1+1+0.15,0.325)  (\u-1+1,0) };

			\draw[green, thick] plot  [line width=2,  tension=1] coordinates {   ({\u+0.5+0.125*cos(120)},{0.5+0.125*sin(120)}) ({\u+0.5+0.35*cos(150)},{0.5+0.35*sin(150)}) ({\u+0.5+0.5*cos(180)},{0.5+0.5*sin(180)}) };
			\draw[green, thick] plot  [line width=2,  tension=1] coordinates {   ({\u+0.5+0.125*cos(300)},{0.5+0.125*sin(300)}) ({\u+0.5+0.35*cos(330)},{0.5+0.35*sin(330)}) ({\u+0.5+0.5*cos(360)},{0.5+0.5*sin(360)}) };
			
		}
		\foreach \u in {2}
		{
			\draw[green, thick] plot  [line width=2,  tension=1] coordinates {   ({\u+0.5+0.125*cos(120)},{0.5+0.125*sin(120)}) ({\u+0.5+0.35*cos(150)},{0.5+0.35*sin(150)}) ({\u+0.5+0.5*cos(180)},{0.5+0.5*sin(180)}) };
			\draw[green, thick] plot  [line width=2,  tension=1] coordinates {   ({\u+0.5+0.125*cos(300)},{0.5+0.125*sin(300)}) ({\u+0.5+0.35*cos(330)},{0.5+0.35*sin(330)}) ({\u+0.5+0.5*cos(360)},{0.5+0.5*sin(360)}) };
			
			\draw[blue, thick] plot  [line width=2,  tension=1] coordinates {   ({\u-1+1+0.5+0.125*cos(150)},{0.5+0.125*sin(150)}) (\u-1+1+0.15,0.325) ( (\u-1+1,0) };
			\draw[blue, thick] plot  [line width=2,  tension=1] coordinates {   ({\u+0.5+0.125*cos(-30)},{0.5+0.125*sin(-30)}) (\u+0.85,0.625) ( (\u+1,1) };
			
		}

		\filldraw[white] (2+0.5,0.5) circle (0.275);
\draw[dotted, thick] (2+0.5-0.125,0.5)--(2+0.5+0.125,0.5);

		\end{tikzpicture}
		\end{displaymath}
		\caption{The laminates on $\mathbb{T}_n$. Parallel dashed lines are identified.} \label{FigureLambdaLamination}
	\end{figure}
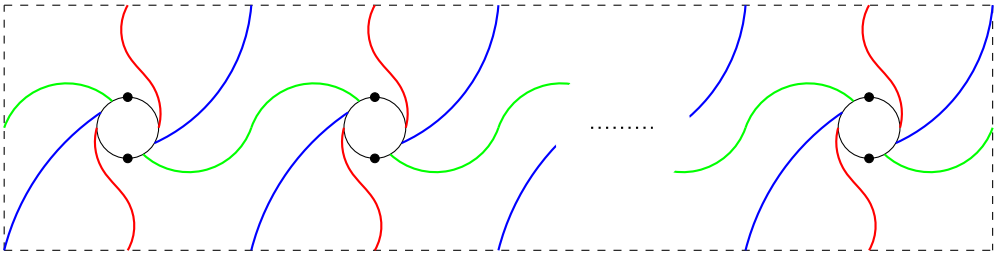
	\noindent If $\Delta$ denotes a $6$-gon in the complement with marked point $\bullet \in \partial \Delta$ which corresponds to a maximal admissible path $\beta \alpha$, then $L_{s(\alpha)}, L_{t(\alpha)}, L_{t(\beta)}$ are the laminates on its boundary in the clockwise orientation of $\Delta$.
	
	\begin{figure}[H]
		\begin{displaymath}
		\begin{array}{cc}
		
		{
			\begin{tikzpicture}[scale=3, hobby]
			
			\draw[dashed] (0,0)--(2,0);
			\draw[dashed] (0,0)--(0,1);
			\draw[dashed] (0,1)--(2,1);
			\draw[dashed] (2,1)--(2,0);

			\foreach \u in {0, 1}
			{
				
				\draw (\u+0.5,0.5) circle (0.125);
			}
			
			\filldraw (0+0.5,0.375) circle (0.5pt);
			
			\foreach \u in {0}
			{
				\draw[green, thick] plot  [line width=2,  tension=1] coordinates {   ({\u+0.5+0.125*cos(300)},{0.5+0.125*sin(300)}) ({\u+0.5+0.35*cos(330)},{0.5+0.35*sin(330)}) ({\u+0.5+0.5*cos(360)},{0.5+0.5*sin(360)}) };
				\draw[red, thick] plot  [line width=2,  tension=1] coordinates {   ({\u+0.5+0.125*cos(0)},{0.5+0.125*sin(0)}) ({\u+0.5+0.2*cos(60)},{0.5+0.2*sin(60)}) ({\u+0.5+0.3*cos(90)},{0.5+0.3*sin(90)}) ( (\u+0.5,1) };
				\draw[red, thick] plot  [line width=2,  tension=1] coordinates {   ({\u+0.5+0.125*cos(180)},{0.5+0.125*sin(180)}) ({\u+0.5+0.2*cos(240)},{0.5+0.2*sin(240)}) ({\u+0.5+0.3*cos(270)},{0.5+0.3*sin(270)}) ( (\u+0.5,0) };

				\draw[blue, thick] plot  [line width=2,  tension=1] coordinates {   ({\u+0.5+0.125*cos(-30)},{0.5+0.125*sin(-30)}) (\u+0.85,0.625) ( (\u+1,1) };
			}

			\foreach \u in {1}
			{
				
				\draw[green, thick] plot  [line width=2,  tension=1] coordinates {   ({\u+0.5+0.125*cos(120)},{0.5+0.125*sin(120)}) ({\u+0.5+0.35*cos(150)},{0.5+0.35*sin(150)}) ({\u+0.5+0.5*cos(180)},{0.5+0.5*sin(180)}) };
				
				\draw[blue, thick] plot  [line width=2,  tension=1] coordinates {   ({\u-1+1+0.5+0.125*cos(150)},{0.5+0.125*sin(150)}) (\u-1+1+0.15,0.325) ( (\u-1+1,0) };
			}
			
			\draw[black] ({0+0.15+0.45*cos(70)},{0.1}) node{$L_{t(d_{i})}$};
			\draw[black] ({0+0.5+0.35*cos(330)},{0.5-0.1+0.35*sin(330)}) node{$L_{s(a_i)}$};
			\draw[black] (1.2,0.1) node{$L_{t(a_{i})}$};

			
			\draw[->] ({1.5+0.25*cos(140)}, {0.5+0.25*sin(140)}) arc (140:187:0.25) node[pos=0.75, left]{$a_i$};
			\draw[->] ({0.5+0.145+0.75*cos(75)}, {0.5-0.58+0.75*sin(75)}) arc (75:91:0.75) node[pos=0.5, above]{$d_i$};
			
			\end{tikzpicture}
		}
		& 
		{
			\begin{tikzpicture}[scale=3, hobby]
			
			\draw[dashed] (0,0)--(2,0);
			\draw[dashed] (0,0)--(0,1);
			\draw[dashed] (0,1)--(2,1);
			\draw[dashed] (2,1)--(2,0);

			\foreach \u in {0, 1}
			{

				\draw (\u+0.5,0.5) circle (0.125);
			}
			\filldraw (1+0.5,0.625) circle (0.5pt);

			\foreach \u in {0}
			{
				\draw[green, thick] plot  [line width=2,  tension=1] coordinates {   ({\u+0.5+0.125*cos(300)},{0.5+0.125*sin(300)}) ({\u+0.5+0.35*cos(330)},{0.5+0.35*sin(330)}) ({\u+0.5+0.5*cos(360)},{0.5+0.5*sin(360)}) };

				\draw[blue, thick] plot  [line width=2,  tension=1] coordinates {   ({\u+0.5+0.125*cos(-30)},{0.5+0.125*sin(-30)}) (\u+0.85,0.625) ( (\u+1,1) };
			}

			\foreach \u in {1}
			{
				\draw[red, thick] plot  [line width=2,  tension=1] coordinates {   ({\u+0.5+0.125*cos(0)},{0.5+0.125*sin(0)}) ({\u+0.5+0.2*cos(60)},{0.5+0.2*sin(60)}) ({\u+0.5+0.3*cos(90)},{0.5+0.3*sin(90)}) ( (\u+0.5,1) };
				\draw[red, thick] plot  [line width=2,  tension=1] coordinates {   ({\u+0.5+0.125*cos(180)},{0.5+0.125*sin(180)}) ({\u+0.5+0.2*cos(240)},{0.5+0.2*sin(240)}) ({\u+0.5+0.3*cos(270)},{0.5+0.3*sin(270)}) ( (\u+0.5,0) };

				\draw[green, thick] plot  [line width=2,  tension=1] coordinates {   ({\u+0.5+0.125*cos(120)},{0.5+0.125*sin(120)}) ({\u+0.5+0.35*cos(150)},{0.5+0.35*sin(150)}) ({\u+0.5+0.5*cos(180)},{0.5+0.5*sin(180)}) };
				
				\draw[blue, thick] plot  [line width=2,  tension=1] coordinates {   ({\u-1+1+0.5+0.125*cos(150)},{0.5+0.125*sin(150)}) (\u-1+1+0.15,0.325) ( (\u-1+1,0) };
			}

\draw[black] ({2-(0.15+0.45*cos(70))},{1-0.1}) node{$L_{t(b_{i})}$};
\draw[black] ({2-(0.5+0.35*cos(330))},{1-(0.5-0.1+0.35*sin(330))}) node{$L_{s(c_i)}$};
\draw[black] ({2-1.2},{1-0.1}) node{$L_{t(c_{i})}$};
			
			
			\draw[->] ({0.5+0.25*cos(-40)}, {0.5+0.25*sin(-40)}) arc (-40:1:0.25) node[pos=0.7, right]{$c_i$};

			\draw[->] ({1.5-0.145+0.75*cos(258)}, {1.08+0.75*sin(258)}) arc (258:271:0.75) node[pos=0.65, below]{$b_i$};

			\end{tikzpicture}
		}
		\end{array}
		\end{displaymath}
		\caption{The two types of $6$-gons occurring in $\mathbb{T}_n \setminus L$. Recall that $L_{t(d_i)}=L_{t(b_{i+1})}$,  $L_{s(a_i)}=L_{s(c_i)}$ and $L_{t(a_i)}=L_{t(c_i)}=L_{s(d_i)}=L_{s(b_{i})}$.
		}\label{FigurePolygons}
	\end{figure}
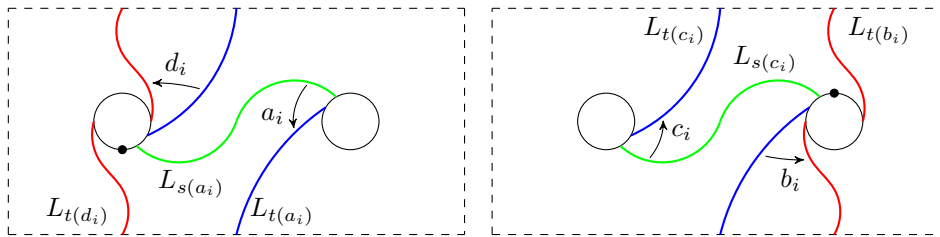
	
\noindent The laminates cut $\partial \mathbb{T}_n$ into segments. Each such boundary segment $\delta$  sits in a polygon which contains a  marked point $\bullet$ and hence determines a unique admissible path: if $p$ is the maximal admissible path corresponding to $\bullet$ and $\delta$ is bounded by the laminates $L_{x_s}$ and $L_{x_t}$ (following the counter-clockwise orientation on the boundary), then we associate with $\delta$ the subpath $q$ of $p$ which starts on $x_s$ and ends on $x_t$. In particular, $q=p$ if and only if $\delta$ contains $\bullet$. For the remaining cases see Figure \ref{FigurePolygons}. \label{InducedSubpath}
	
	\subsection{Curves and winding numbers}\label{SectionComplexOfLoops}
	\noindent Certain homotopy classes of curves in $\mathbb{T}_n$ give rise to indecomposable objects in $\mathcal{D}^b(\Lambda_n)$. We consider two types of curves, namely \emph{loops} and \emph{arcs}. The former are closed whereas the latter start and end at marked points. Before we explain how to construct a complex from a curve, we introduce some necessary background material. As indicated above, not all loops give rise to an object and the necessary condition can be expressed in terms of its winding number which we define in Section \ref{SectionQuiverOfALoop}.

	\subsubsection{Loops and arcs} \ \medskip

		\noindent Suppose $\Sigma$ is a compact, oriented surface, with or without boundary, with a finite  subset $\marked \subseteq \Sigma$ of marked points. The interior marked points are called \textbf{punctures} and their set shall be denoted by $\mathscr{P}$. We will mainly consider two examples in this paper: $\Sigma= \mathbb{T}_n$ or $\Sigma=\mathscr{T}^n$.
			
		 A \textbf{curve} in $\Sigma$ is an immersion $\gamma:\Omega \rightarrow \Sigma$, where $\Omega=[0,1]$ or $\Omega=S^1$ and $\partial \Omega= \gamma^{-1}(\marked)$. If $\Omega= [0,1]$, the curve $\gamma$ is called an \textbf{arc} and if $\Omega = S^1$, it is called a \textbf{loop}.
		 
		  A \textbf{homotopy of curves} is a homotopy $H: [0,1] \times \Omega \rightarrow \Sigma$ of the underlying maps which is constant on $[0,1] \times \partial \Omega$ and such that $H^{-1}(\mathscr{P}) \subseteq [0,1] \times \partial \Omega$.    By definition, we assume that a loop is not homotopic to a constant map and that an arc cannot be homotoped into a contractible neighborhood $U \subset \Sigma$ which contains a single puncture. In particular, there exists a canonical bijection between homotopy classes of loops on $\mathbb{T}_n$ and loops on $\mathscr{T}^n$.\medskip

\noindent Throughout the paper, we will stick to a few conventions:

\begin{convention} \
	\begin{itemize}
		\item Unless stated otherwise, we assume loops to be primitive: a loop $\gamma: S^1 \rightarrow \Sigma$ is said to be \textbf{primitive} if $\gamma$ is not homotopic to a loop which factors through a non-trivial covering map of $S^1$.
		\item 	Up to homotopy, we may and will assume that the number of intersections of a curve with all the laminates is minimal in its  homotopy class and that all such intersections are double points which lie in the interior. 
		\item Depending on the circumstances, we regard curves as oriented or unoriented objects by choosing (or not choosing) and orientation on its domain. The correct interpretation of the words ``arc'' and  ``loop'' should be apparent from the context. 
		\end{itemize}
\end{convention}

		\subsubsection{Gradings on curves} \ \medskip
		
		\noindent Let $\gamma$ be a curve on $\mathbb{T}_n$ and denote by $\gamma \cap L$ the set of intersections of $\gamma$ with the laminates.  Regarding $\gamma \cap L$ as a subset of the domain of $\gamma$ in the natural way, we can speak of \textit{neighboring} intersections.
		 
		 \begin{definition}
		 A \textbf{grading} of $\gamma$ is a function $g: \gamma \cap L \rightarrow \mathbb{Z}$ which satisfies the following property for all pairs $(p,q)$ of neighboring intersections  $p,q \in \gamma \cap L$:\smallskip
		 
	\begin{changemargin}{0.3cm}{0.3cm} 
\noindent	\textit{If $\delta$ denotes the segment of $\gamma$ between $p$ and $q$, oriented such that $\delta$ starts at $p$ and ends at $q$, and  $\Delta \subseteq \mathbb{T}_n \setminus L$ denotes the open polygon which contains $\delta \setminus \{p, q\}$, then }
		
		$$
	g(q) = \begin{cases}
              g(p) + 1, & \textit{if the marked point in $\Delta$ is on the left of $\delta$}; \\
              g(p) - 1, & \textit{otherwise.}
             \end{cases}
             $$

\end{changemargin}
		 \end{definition}

\noindent The notion of grading carries over to homotopy classes of curves in the obvious way. A \textbf{graded curve} is a pair $(\gamma, g)$, where $\gamma$ is a curve and $g$ is a grading of $\gamma$. If $g$ is a grading of $\gamma$ and $n \in \mathbb{Z}$, then so is its \textbf{shift} $g[n]$, defined by $g[n](p)=g(p)-n$ for all $p \in \gamma \cap L$. While all arcs can be graded, there exist loops which do not admit a grading and hence we say that a loop (or generally any curve) is \textbf{gradable} if it admits a grading.\medskip

\noindent If a curve $\gamma$ is gradable, then any grading $g$ is determined completely by a single value $g(p)$ for any $p \in \gamma \cap L$ and all gradings of $\gamma$ are shifts of $g$.

	\subsubsection{The quiver of a curve and the winding number of a loop} \label{SectionQuiverOfALoop} \ \medskip

	\noindent 	Let $\gamma \subseteq \mathbb{T}_n$ be an oriented curve. Denote by $q_0, \dots, q_m$ the ordered sequence of elements in  $\gamma \cap L$ and for each $i \in [0,m]$, denote by $L_{x_i}$ the laminate which contains $q_i$. Then, we write $P_i=\Lambda_n x_i$ for the indecomposable projective $\Lambda_n$-module associated with $x_i$ and denote by $\gamma_i$ the segment of $\gamma$ between $q_i$ and $q_{i+1}$. The orientation of $\gamma$ induces a canonical orientation on $\gamma_i$.
	
	This data gives rise to a quiver $Q(\gamma)$ which is of type $\widetilde{A}$ if $\gamma$ is a loop and which is of type $A$ if $\gamma$ is an arc. We will see that for every grading $g$ of $\gamma$, $Q(\gamma)$ encodes complexes in the same way as the quivers in Figure \ref{FigureQuivers} and Figure \ref{FigurePicardGroupComplex} encode the complexes $\mathcal{O}(\lambda)$ and $\Bbbk(i, \lambda)$.\medskip
	
	\noindent The quiver $Q(\gamma)$ has vertices $\{P_0, \dots, P_m\}$ as defined above and arrows $\sigma_0, \dots, \sigma_m$ whose start and end points are determined by the segments $\gamma_i$ in the following way. For each $i \in [0,m)$, $\sigma_i$ connects $P_i$ with $P_{i+1}$ -- more precisely, if $\bullet$ denotes the unique marked point in the polygon $\Delta_i \supseteq \gamma_i$, we have 

		\[
\sigma_i= \begin{cases}{\begin{tikzcd}P_i \arrow{r}{} & P_{i+1},\end{tikzcd}} & \textrm{if $\bullet$ lies to the left of $\gamma_i$ in $\Delta_i$;} \\ {\begin{tikzcd}P_i & \arrow{l}{} P_{i+1},\end{tikzcd}} & \textrm{if $\bullet$ lies to the right of $\gamma_i$ in $\Delta_i$.} \end{cases}\]	
	
	\begin{definition}
	Assume that $\gamma$ is a loop.	The \textbf{winding number $\omega(\gamma)$ of $\gamma$} is defined as the difference $b-a$, where
		$a$ denotes the number of arrows $\sigma_i$ from $P_i$ to $P_{i+1}$ and $b$ denotes the number of arrows $\sigma_j$ from $P_{j+1}$ to $P_j$.
	\end{definition} 
	
	\noindent A change of orientation of a loop changes the sign of its winding number. On the other hand, one observes that $Q(\gamma)$ does not depend on the chosen orientation as an abstract quiver. Gradable loops can be characterized as follows.
\begin{lem}\label{LemmaConditionGradable}
	A loop $\gamma$ is gradable if and only if $\omega(\gamma)=0$ if and only if $Q(\gamma)$ has the same number of clockwise and counter-clockwise arrows.  
	\end{lem}
\noindent Note that whether a loop is gradable or not does not depend on the choice of an orientation.
	\begin{exa}\label{WindingNumbersBoundary}
Every boundary component  $B \subseteq \partial\mathbb{T}_n$ determines an embedding $S^1 \rightarrow \mathbb{T}_n$ which we regard as a simple loop $\gamma$ after a small deformation. Then $\gamma$ inherits a canonical orientation from $B$ (the induced orientation on $\partial \mathbb{T}_n$). The quiver of $\gamma$ is depicted in Figure \ref{FigureWindingNumberBoundary} from which we determine its winding number as $2-4=-2$. We refer to this winding number as the \textbf{winding number of $B$}.\end{exa}

\begin{figure}
	\begin{tikzpicture}[scale=4, hobby]
	\draw[dashed] (0,0)--(1,0);
	\draw[dashed] (0,0)--(0,1);
	\draw[dashed] (0,1)--(1,1);
	\draw[dashed] (1,1)--(1,0);

	\foreach \u in {0}
	{
		\filldraw (\u+0.5,0.625) circle (0.5pt);
		\filldraw (\u+0.5,0.375) circle (0.5pt);
		\draw (\u+0.5,0.5) circle (0.125);
		
		\draw[red, thick] plot  [line width=2,  tension=1] coordinates {   ({\u+0.5+0.125*cos(0)},{0.5+0.125*sin(0)}) ({\u+0.5+0.2*cos(60)},{0.5+0.2*sin(60)}) ({\u+0.5+0.3*cos(90)},{0.5+0.3*sin(90)}) ( (\u+0.5,1) };
		\draw[red, thick] plot  [line width=2,  tension=1] coordinates {   ({\u+0.5+0.125*cos(180)},{0.5+0.125*sin(180)}) ({\u+0.5+0.2*cos(240)},{0.5+0.2*sin(240)}) ({\u+0.5+0.3*cos(270)},{0.5+0.3*sin(270)}) ( (\u+0.5,0) };
		
		\draw[green, thick] plot  [line width=2,  tension=1] coordinates {   ({\u+0.5+0.125*cos(120)},{0.5+0.125*sin(120)}) ({\u+0.5+0.35*cos(150)},{0.5+0.35*sin(150)}) ({\u+0.5+0.5*cos(180)},{0.5+0.5*sin(180)}) };
		\draw[green, thick] plot  [line width=2,  tension=1] coordinates {   ({\u+0.5+0.125*cos(300)},{0.5+0.125*sin(300)}) ({\u+0.5+0.35*cos(330)},{0.5+0.35*sin(330)}) ({\u+0.5+0.5*cos(360)},{0.5+0.5*sin(360)}) };
		
		\draw[blue, thick] plot  [line width=2,  tension=1] coordinates {   ({\u+0.5+0.125*cos(-30)},{0.5+0.125*sin(-30)}) (\u+0.85,0.625)  (\u+1,1) };
		\draw[blue, thick] plot  [line width=2,  tension=1] coordinates {     (\u-1+1,0) (\u-1+1+0.15,0.325) ({\u-1+1+0.5+0.125*cos(150)},{0.5+0.125*sin(150)})  };
		
	}
	
	\filldraw[black] ({0.5+0.375*cos(91.7)},{0.5+0.375*sin(91.7)}) circle (0.6pt);
	\filldraw[black] ({0.5+0.375*cos(271.7)},{0.5+0.375*sin(271.7)}) circle (0.6pt);
	
	\filldraw[black] ({0.5+0.375*cos(20)},{0.5+0.375*sin(20)}) circle (0.6pt);
	\filldraw[black] ({0.5+0.375*cos(205.3)},{0.5+0.375*sin(205.3)}) circle (0.6pt);
	
	\filldraw[black] ({0.5+0.375*cos(154)},{0.5+0.375*sin(154)}) circle (0.6pt);
	\filldraw[black] ({0.5+0.375*cos(334)},{0.5+0.375*sin(334)}) circle (0.6pt);

	
	\draw[->, thick] ({0.5+0.375*cos(97.3)},{0.5+0.375*sin(97.3)}) arc (97.3:148.5:0.375); 
	\draw[<-, thick] ({0.5+0.375*cos(159.5)},{0.5+0.375*sin(159.5)}) arc (159.5:199.8:0.375); 
	\draw[<-, thick] ({0.5+0.375*cos(210.8)},{0.5+0.375*sin(210.8)}) arc (210.8:266.2:0.375);
	\draw[->, thick] ({0.5+0.375*cos(277.2)},{0.5+0.375*sin(277.2)}) arc (277.2:328.5:0.375);
	\draw[<-, thick] ({0.5+0.375*cos(339.5)},{0.5+0.375*sin(339.5)}) arc (339.5:374.5:0.375);
	\draw[<-, thick] ({0.5+0.375*cos(339.5)},{0.5+0.375*sin(339.5)}) arc (339.5:374.5:0.375);
	\draw[<-, thick] ({0.5+0.375*cos(25.5)},{0.5+0.375*sin(25.5)}) arc (25.5:86.2:0.375);
	
	\end{tikzpicture}
	\caption{The quiver of the clockwise boundary loop with $4$ clockwise and $2$  counter-clockwise arrows.} \label{FigureWindingNumberBoundary}
\end{figure}
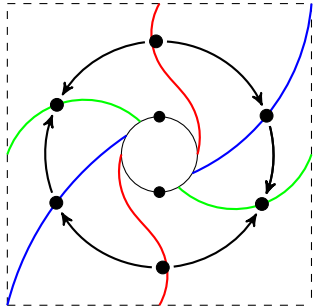
	
	\subsection{The complex of a graded curve}\label{SectionComplexOfACurve}
	\ \medskip

	\noindent In order to define a complex of a curve, we recall the definition of a local system.
	
	\begin{definition}
	Let $\gamma$ be a curve with domain $\Omega$. A \textbf{local system} $\mathcal{V}$ on $\gamma$ consists of a collection of vector spaces $(\mathcal{V}_z)_{z \in \Omega}$ and isomorphisms $\mathcal{V}_u: \mathcal{V}_{u(0)} \rightarrow \mathcal{V}_{u(1)}$ for every path $u: [0,1] \rightarrow \Omega$ which only depends on the homotopy class of $u$. Moreover, if $v$ is a path with $v(0)=u(1)$, then $\mathcal{V}_{v \ast u}= \mathcal{V}_v \circ \mathcal{V}_u$, where $v \ast u$ denotes the concatenation of $v$ with $u$.
	\end{definition}

\noindent The well-defined number $\dim \mathcal{V} \coloneqq \dim \mathcal{V}_z$ is the \textbf{dimension of $\mathcal{V}$}. After the choice of a base point on $\Omega$, one sees that isomorphism classes of indecomposable local systems (coproducts being defined pointwise) on a loop are in bijection with powers of irreducible polynomials over $\Bbbk$ and hence with the polynomials $(X-\lambda)^m$, where $\lambda \in \Bbbk^{\times}$ and $m =\dim \mathcal{V}$. On the other hand, every indecomposable local system on an arc is isomorphic to the constant $1$-dimensional local system, i.e.\ $\mathcal{V}_z= \Bbbk$ for all $z \in [0,1]$ and every isomorphism $\mathcal{V}_{u(0)} \rightarrow \mathcal{V}_{u(1)}$ above is the identity.\medskip

\noindent Let $\mathcal{V}$ be a local system on a graded curve $(\gamma, g)$. We attach a complex $\P_{(\gamma, g)}(\mathcal{V})$ to the quiver $Q(\gamma)$ by, roughly speaking,  ``folding'' $Q(\gamma)$ to a complex in the same way as the quivers in Figure \ref{FigurePicardGroupComplex}. 

We use the notation of Section \ref{SectionQuiverOfALoop}. Set $\mathcal{V}_i=\mathcal{V}_{q_i}$. Then, as a graded $\Lambda_n$-module, $$\P_{(\gamma, g)}(\mathcal{V})= \bigoplus_i \left(P_i \otimes \mathcal{V}_i\right)[-g(q_i)].$$

\noindent The differential is described as follows. For each $i$, let $t_i$ denote the map between $P_i$ and $P_{i+1}$ induced by the unique admissible path $\alpha_i$ determined by the boundary segment in $\partial \mathbb{T}_n \cap \partial \Delta_i$ which is bounded by $L_{x_i}$ and $L_{x_{i+1}}$, c.f. the last paragraph of Section \ref{SectionToriMarkedPoints}. The segment $\gamma_i$ of $\gamma$ which is bounded by $L_{x_i}$ and $L_{x_{i+1}}$ also determines an isomorphism $u_i$ between $\mathcal{V}_i$ and $\mathcal{V}_{i+1}$. The differential is the sum of  all maps $t_i \otimes u_i$, shifted appropriately. Lemma \ref{LemmaConditionGradable} ensures that the procedure above defines a differential. Note that $t_i: P_{t(\alpha_i)} \rightarrow P_{s(\alpha_i)}$ and  $(x_i, x_{i+1})=(t(\alpha_i), s(\alpha_i))$ if and only if the marked point in $\Delta_i$ lies to the left of $\gamma_i$. Hence the domain (resp.~codomain) of $t_i$ agrees with the start (resp.~end) of $\sigma_i$ in $Q(\gamma)$.\medskip

\noindent The above construction respects direct sums in that $\P_{(\gamma, g)}(\mathcal{U} \oplus \mathcal{V}) \cong \P_{(\gamma, g)}(\mathcal{U}) \oplus \P_{(\gamma, g)}(\mathcal{U})$. In fact, $\P=\P_{(\gamma, g)}(\mathcal{V})$ is indecomposable if and only if $\mathcal{V}$ is indecomposable and its isomorphism class is invariant under isomorphisms of local systems and change of orientation on $\gamma$. If $\gamma$ is a loop and $\mathcal{V}$ is indecomposable, then $\P$ is an instance of a \textit{band complex} in the sense of \cite{BekkertMerklen} and \cite{BurbanDrozd2004}. In particular, $\tau \P \cong \P$ as first proved by Bobinski \cite{Bobinski}. Moreover, $\dim \mathcal{V}$ is the level of $\P$  in the corresponding homogeneous tube.

If $\gamma$ is an arc and $\mathcal{V}$ is indecomposable, then $\P$ is a so-called \textit{string complex}. As we may assume that $\mathcal{V}$ is constant, we often suppress the additional datum of a local system in this case.\medskip

\noindent Every homotopy $H:[0,1] \times \Omega$ between curves $\gamma_0=H|_{\{0\} \times \Omega}$ and $\gamma_1=H|_{\{1\} \times \Omega}$ which are in minimal position with $L$ induces a bijection $\gamma_0 \cap L \cong \gamma_1 \cap L$ and hence a bijection between the sets  gradings on $\gamma_0$ and $\gamma_1$. Because the domains of $\gamma_0$ and $\gamma_1$ agree, we further have a canonical bijection between local systems on $\gamma_0$ and $\gamma_1$. The following statement is found in \cite[Theorem 2.12]{OpperPlamondonSchroll}. The bijection between curves and indecomposable objects was first proved in \cite{HaidenKatzarkovKontsevich}.

\begin{prp}\label{PropositionBijectionObjectsCurves}
Let $\cX$ denote the set of triples $(\gamma,g, \mathcal{V})$, where $(\gamma, g)$ is a graded curve and $\mathcal{V}$ is an indecomposable local system on $\gamma$. Let further $\sim$ denote the equivalence relation on $\cX$  such that $(\gamma, g, \mathcal{V}) \sim (\gamma', g', \mathcal{V}')$ if and only if there exists a homotopy between $\gamma$ and $\gamma'$ under which $g$ corresponds to $g'$ and $\mathcal{V}$ corresponds to $\mathcal{V}'$. The assignment
	\[\begin{tikzcd}
	(\gamma, g, \mathcal{V}) \arrow[mapsto]{r} & \P_{(\gamma, g)}(\mathcal{V}),
	\end{tikzcd}
	\]
	\noindent induces a bijection from the set of $\sim$-equivalence classes of $\cX$ to the set of isomorphism classes of indecomposable objects in $\mathcal{D}^b(\Lambda_n)$.
	
\end{prp}
\noindent In summary, Proposition \ref{PropositionBijectionObjectsCurves} states that every indecomposable object $X \in \mathcal{D}^b(\Lambda_n)$ is represented by a curve endowed with a grading and an  indecomposable local system. To simplify notation, we often refer to such a curve and the local system by $\gamma_X$ and $\mathcal{V}_X$ respectively so that $X \cong \P_{(\gamma_X, g)}(\mathcal{V}_X)$ for some unique grading $g$ on $\gamma_X$. In this case, we say that $X$ is \textbf{represented} by $\gamma_X$. Finally,  if $Y \in \Perf(C_n)$ is indecomposable, we write $\gamma_Y$ and $\mathcal{V}_Y$ as short for $\gamma_{\mathbb{F}(Y)}$ and $\mathcal{V}_{\mathbb{F}(Y)}$ and say that $Y$ is represented by  $\gamma_Y$. As a consequence of Theorem \ref{TheoremBurbanDrozdImage}, we prove in Section \ref{SectionFractional} that loops on $\mathbb{T}_n$ (or equivalently, on $\mathscr{T}^n$) represent precisely the indecomposable perfect complexes of $C_n$. \medskip

	\begin{exa}\label{ExampleLoopsOnTorus}
		Let $\gamma_{\Pic}: S^1 \rightarrow \mathbb{T}_n$ be the loop defined by  $\gamma_{\Pic}(e^{2\pi i t}):=(n \cdot t, \frac{1}{8})$ for all $t \in [0,1]$. Similarly, for $j \in [0,n)$, let $\gamma_{\Bbbk(x)}^j:S^1 \rightarrow \mathbb{T}_n$ denote the loop defined by $\gamma_{\Bbbk(x)}^j(e^{2\pi i t}):=(j+1, t)$, $t \in [0,1]$.
		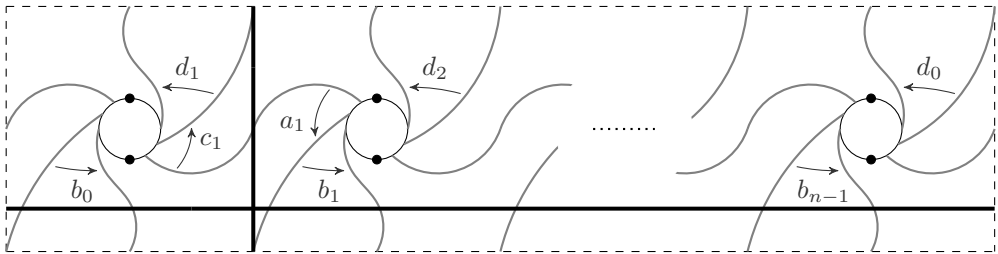
\begin{figure}[H]
			\begin{displaymath}
			\begin{tikzpicture}[scale=3.25, hobby]
			\draw[dashed] (0,0)--(4,0);
			\draw[dashed] (0,0)--(0,1);
			\draw[dashed] (0,1)--(4,1);
			\draw[dashed] (4,1)--(4,0);

			\foreach \u in {0,1,3}
			{
				\filldraw (\u+0.5,0.625) circle (0.5pt);
				\filldraw (\u+0.5,0.375) circle (0.5pt);
				\draw (\u+0.5,0.5) circle (0.125);
				
				\draw[gray, thick] plot  [line width=2,  tension=1] coordinates {   ({\u+0.5+0.125*cos(0)},{0.5+0.125*sin(0)}) ({\u+0.5+0.2*cos(60)},{0.5+0.2*sin(60)}) ({\u+0.5+0.3*cos(90)},{0.5+0.3*sin(90)}) ( (\u+0.5,1) };
				\draw[gray, thick] plot  [line width=2,  tension=1] coordinates {   ({\u+0.5+0.125*cos(180)},{0.5+0.125*sin(180)}) ({\u+0.5+0.2*cos(240)},{0.5+0.2*sin(240)}) ({\u+0.5+0.3*cos(270)},{0.5+0.3*sin(270)}) ( (\u+0.5,0) };
				
				\draw[gray, thick] plot  [line width=2,  tension=1] coordinates {   ({\u+0.5+0.125*cos(-30)},{0.5+0.125*sin(-30)}) (\u+0.85,0.625) ( (\u+1,1) };
				
				\draw[gray, thick] plot  [line width=2,  tension=1] coordinates {   ({\u-1+1+0.5+0.125*cos(150)},{0.5+0.125*sin(150)}) (\u-1+1+0.15,0.325) ( (\u-1+1,0) };

				\draw[gray, thick] plot  [line width=2,  tension=1] coordinates {   ({\u+0.5+0.125*cos(120)},{0.5+0.125*sin(120)}) ({\u+0.5+0.35*cos(150)},{0.5+0.35*sin(150)}) ({\u+0.5+0.5*cos(180)},{0.5+0.5*sin(180)}) };
				\draw[gray, thick] plot  [line width=2,  tension=1] coordinates {   ({\u+0.5+0.125*cos(300)},{0.5+0.125*sin(300)}) ({\u+0.5+0.35*cos(330)},{0.5+0.35*sin(330)}) ({\u+0.5+0.5*cos(360)},{0.5+0.5*sin(360)}) };
				
			}
			\foreach \u in {2}
			{
				\draw[gray, thick] plot  [line width=2,  tension=1] coordinates {   ({\u+0.5+0.125*cos(120)},{0.5+0.125*sin(120)}) ({\u+0.5+0.35*cos(150)},{0.5+0.35*sin(150)}) ({\u+0.5+0.5*cos(180)},{0.5+0.5*sin(180)}) };
				\draw[gray, thick] plot  [line width=2,  tension=1] coordinates {   ({\u+0.5+0.125*cos(300)},{0.5+0.125*sin(300)}) ({\u+0.5+0.35*cos(330)},{0.5+0.35*sin(330)}) ({\u+0.5+0.5*cos(360)},{0.5+0.5*sin(360)}) };
				
				\draw[gray, thick] plot  [line width=2,  tension=1] coordinates {   ({\u-1+1+0.5+0.125*cos(150)},{0.5+0.125*sin(150)}) (\u-1+1+0.15,0.325) ( (\u-1+1,0) };
				\draw[gray, thick] plot  [line width=2,  tension=1] coordinates {   ({\u+0.5+0.125*cos(-30)},{0.5+0.125*sin(-30)}) (\u+0.85,0.625) ( (\u+1,1) };
				
			}
			
			
			\filldraw[white] (2+0.5,0.5) circle (0.275);

			\draw[dotted, thick] (2+0.5-0.125,0.5)--(2+0.5+0.125,0.5);

			\draw[line width=1.5, black] (0,0.175)--(0.75,0.175);
			
			\draw[line width=1.5, black] (0.75, 0.175)--(4, 0.175);
			\draw[line width=1.5, black] (1,0)--(1, 0.75);
			\draw[line width=1.5, black] (1,0.75)--(1,1);
			
	\draw[->, color=black!80] ({0.5-0.145+0.75*cos(258)}, {1.08+0.75*sin(258)}) arc (258:271:0.75) node[pos=0.65, below]{$b_2$};
	\draw[->, color=black!80] ({1.5-0.145+0.75*cos(258)}, {1.08+0.75*sin(258)}) arc (258:271:0.75) node[pos=0.65, below]{$b_1$};
	\draw[->, color=black!80] ({3.5-0.145+0.75*cos(258)}, {1.08+0.75*sin(258)}) arc (258:271:0.75) node[pos=0.65, below]{$b_3$};
	
	\draw[->, color=black!80] ({0.5+0.145+0.75*cos(75)}, {0.5-0.58+0.75*sin(75)}) arc (75:91:0.75) node[pos=0.5, above]{$d_1$};
	\draw[->, color=black!80] ({1.5+0.145+0.75*cos(75)}, {0.5-0.58+0.75*sin(75)}) arc (75:91:0.75) node[pos=0.5, above]{$d_0$};
	\draw[->, color=black!80] ({3.5+0.145+0.75*cos(75)}, {0.5-0.58+0.75*sin(75)}) arc (75:91:0.75) node[pos=0.5, above]{$d_2$};
	
	\draw[->, color=black!80] ({0.5+0.25*cos(-40)}, {0.5+0.25*sin(-40)}) arc (-40:1:0.25) node[pos=0.7, right]{$c_1$};
	\draw[->, color=black!80] ({1.5+0.25*cos(140)}, {0.5+0.25*sin(140)}) arc (140:187:0.25) node[pos=0.75, left]{$a_1$};

			\end{tikzpicture}
			\end{displaymath}
			\caption{The loops $\gamma_{\Pic}$ (horizontal) and $\gamma_{\Bbbk(x)}^1$ (vertical).} \label{FigureLoopsOfPicardSkyscraper}
		\end{figure}
		
		\noindent Replacing the vertices of $Q(\gamma_{\Pic})$ by bullet points and assigning the underlying path of $t_i$ as the label of the arrow $\sigma_i$, $Q(\gamma_{\Pic})$ becomes the quiver on the left hand side of Figure \ref{FigurePicardGroupComplex} on page \pageref{FigurePicardGroupComplex}, whereas $Q(\gamma_{\Bbbk(x)}^i)$ is the quiver on the right hand side. Thus, suppressing the grading we have $\P_{\gamma_{\Pic}}(\mathcal{V}) \cong \mathcal{O}(\lambda)$ and $\P_{\gamma_{\Bbbk(x)}^i}(\mathcal{W}) \cong \Bbbk(i,\lambda)$ for suitable $1$-dimensional local systems $\mathcal{V}$ and $\mathcal{W}$.
		
	\end{exa}

\subsubsection{Fractional Calabi-Yau objects: loops and boundary arcs}\label{SectionFractional}\ \medskip
	
\noindent 	One of the main differences between band complexes and string complexes is that the latter are generally not $\tau$-invariant. A criterion which allows us to characterize the $\tau$-invariant indecomposable objects in the derived category of any gentle algebra is provided in \cite[Proposition 2.16]{OpperDerivedEquivalences}. It says that a string complex $X$ satisfies $\tau^m X \cong X[d]$ for some $m > 0$ and some $d \in \mathbb{Z}$ if and only if $\gamma_X$ is homotopic to a boundary arc on a component with exactly $m'$ marked points and winding number $-d'$ such that $(m,d)$ is an integer multiple of $(m',d')$.

	\begin{prp}\label{PropositionClassificationCalabiYauObjects}Let $m \geq 1$, $d \in \mathbb{Z}$ and let $X \in \mathcal{D}^b(\Lambda_n)$. Then, the following are true.
	\begin{enumerate}
	    \item If $\tau^m X \cong X[d]$, then $X$ is $\tau$-invariant or $\tau^2 X \cong X[2]$;
	    \item  $X$ is $\tau$-invariant if and only if each indecomposable direct summand of $X$ is represented by a loop.
	    \item We have $\tau^2 X \cong X[2]$ if and only if each indecomposable direct summand of $X$ is represented by a boundary arc.
	\end{enumerate}
  
	\end{prp}
	\begin{proof}  From  Example \ref{WindingNumbersBoundary} and the discussion preceding the proposition, we know that an indecomposable object $Y \in \mathcal{D}^b(\Lambda_n)$ satisfies  $\tau^p Y \cong Y[q]$ for a pair $(p,q) \in \mathbb{Z}^2$ if and only if there exists $(m,d) \in \{(1,0), (2,2)\}$ such that $\tau^m Y \cong Y[d]$.  If $(m,d)=(1,0)$, $Y$ is represented by a loop and if $(m,d)=(2,2)$, then $Y$ is represented by a boundary arc. Suppose that $X$ satisfies $\tau^m X \cong X[d]$ and $X \cong \bigoplus_{i=1}^r{X_i}$ with each $X_i$ indecomposable. Then, for each $j \in [1,r]$, there exists a non-empty subset $J \subset \{1, \dots, r\}$ such that $j \in J$ and a cyclic permutation $\sigma: J \rightarrow J$ such that $\tau^m X_i \cong X_{\sigma(i)}[d]$ for all $i \in J$. Since $\sigma^{|J|}=\text{Id}_J$, we conclude that $\tau^{m |J|} X_j \cong X_j[d |J|]$ and that $(m|J|, d|J|)$ is an integer multiple of $(1,0)$ or $(2,2)$. Since $|J| \neq 0$, one easily derives that either $\tau X_i \cong X_i$ for all $i \in [1,r]$ or $\tau^2 X_i \cong X_i[2]$ for all $i \in [1,r]$.
	\end{proof}

\begin{cor}
	Let $X \in \mathcal{D}^b(\Lambda_n)$ be indecomposable. Then,  $X$ lies in the essential image of the functor $\mathbb{F}: \Perf(C_n) \longrightarrow \mathcal{D}^b(\Lambda_n)$  if and only if $\gamma_X$ is a loop.
\end{cor}

\subsection{Morphisms and Intersections}\ \medskip
	
	\noindent We recall the relationship between morphisms and intersections of curves. A finite set $\{\gamma_1, \dots, \gamma_m\}$ of curves is in \textbf{minimal position} if no three curves from the set intersect in a single point in the interior (only ``double-points'') and for all (not necessarily distinct) $i, j \in [1,m]$, the number of (self-)intersections of $\gamma_i$ and $\gamma_j$ is minimal within their respective homotopy classes.
	\begin{convention}
		For the rest of this paper, we will always assume every instance of a finite set of curves to be in minimal position.
	\end{convention}  
\noindent Every finite set of curves can be deformed into minimal position by homotopy, and hence this assumption does not impose any restrictions on the homotopy classes of the curves. \medskip
	
\noindent 	Let $(\gamma_1, \gamma_2)$ be a pair of distinct curves and for $i \in \{1,2\}$, let $\Omega_i$ denote the domain of $\gamma_i$. The set $\gamma_1 \overrightarrow{\cap} \gamma_2$ of \textbf{oriented intersections} consists of all pairs $(s_1,s_2) \in \Omega_1 \times \Omega_2$ such that $p\coloneqq\gamma_1(s_1)=\gamma_2(s_2)$, and such that if $p$ is a marked point, then  locally around $p$, $\gamma_1$ ``comes before'' $\gamma_2$ in the counter-clockwise orientation as shown in Figure \ref{FigureDirectedBoundaryIntersection}. A \textbf{self-intersection} of a curve $\gamma$ with domain $\Omega$ is a pair $(s_1, s_2) \in \Omega^2$ with $s_1 \neq s_2$ and $\gamma(s_1)=\gamma(s_2)$.
	\begin{figure}
		\begin{displaymath}
		\begin{tikzpicture}[scale=0.8, hobby]
		\draw[thick, 
		decoration={markings, mark=at position 0.4 with {\arrow{<}}},
		postaction={decorate}
		](0,0) circle (50pt);
		\filldraw (-50pt, 0) circle (2pt);
		\draw (-40pt,0) node{$p$};
		\draw (0,1.3) node {$\partial \mathbb{T}_n$};
		\draw [line width=0.5, color=black] plot  [ tension=1] coordinates {  (-50pt,0) (-3, 1) (-5, 1.5)  };
		\draw [line width=0.5, color=black] plot  [ tension=1] coordinates {  (-50pt,0) (-3, -0.5) (-5, -1.5)  };
		\draw (-3,1.3) node{$\gamma_1$};
		\draw (-3,-0.8) node{$\gamma_2$};
		
		\draw[dashed, ->] ({3.2*cos(165)},{3.2*sin(165)}) arc (164:188:3.2);
		\end{tikzpicture}
		\end{displaymath}
		\caption{An oriented boundary intersection $p$ from $\gamma_1$ to $\gamma_2$.} \label{FigureDirectedBoundaryIntersection}
	\end{figure}
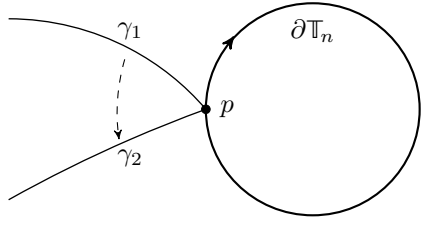

		\noindent Every interior intersection of $\gamma_1$ and $\gamma_2$ contributes elements in both $\gamma_1 \overrightarrow{\cap} \gamma_2$ and $\gamma_2 \overrightarrow{\cap} \gamma_1$, whereas every boundary intersection determines an element in only one of the two sets. Frequently, we do not distinguish between an oriented intersection and its image in the surface.\medskip
		
	\noindent The following proposition is a special case of \cite[Theorem 3.3]{OpperPlamondonSchroll}.
	
	\begin{prp}\label{PropositionMorpismIntersection}
	Let $X,Y \in \mathcal{D}^b(\Lambda_n)$ be indecomposable. Whenever $X$ or $Y$ is $\tau$-invariant, we assume that its associated local system has dimension $1$. 
	If $X$ is not isomorphic to a shift of $Y$ or $\gamma_X$ is an arc, then a homogeneous basis of $\Hom^{\bullet}(X,Y)$ is in bijection with $\gamma_X \overrightarrow{\cap} \gamma_Y$.
	\end{prp}
\noindent The remaining case is slightly different. When $X$ is $\tau$-invariant with $\dim \mathcal{V}_X=1$ and $X \cong Y[m]$ for some $m \in \mathbb{Z}$, then \cite[Theorem 3.3]{OpperPlamondonSchroll} states that there exists a homogeneous basis $\{\phi, \omega\} \cup \mathcal{I}$ of $\Hom^{\bullet}(X,Y)$ and a bijection $\mathcal{I} \xrightarrow{\sim} \gamma_X \overrightarrow{\cap} \gamma_Y$, where $\phi \in \Hom^m(X,Y)$ is an isomorphism and $\omega \in \Hom^{m+1}(X,Y)$ occurs as the connecting morphism in an Auslander-Reiten triangle
	\[
	\begin{tikzcd}
	  Y[m] \arrow{r} & V  \arrow{r} & X \arrow{r}{\omega} & Y[m+1].	\end{tikzcd}
	\]
	
	\noindent We note that the number of self-intersections of a loop $\gamma$ in minimal position coincides with the cardinality of $\gamma \overrightarrow{\cap} \gamma'$, where $\gamma' \simeq \gamma$ such that $\{\gamma, \gamma'\}$ is in minimal position. We summarize the situation for $\tau$-invariant objects as follows. 
	\begin{cor}\label{CorollaryMorphisms}
		Let $X, Y \in \mathcal{D}^b(\Lambda_n)$ be $\tau$-invariant indecomposable objects such that $\dim \mathcal{V}_X=1=\dim \mathcal{V}_Y$ and $\gamma_X \not \simeq \gamma_Y$. With the notation $\operatorname{hom}^{\bullet}(U,V)\coloneqq \dim \Hom^{\bullet}(U,V)$ we have
		\begin{enumerate}
			\setlength\itemsep{0.5em}
			
			\item $\gamma_X$ and $\gamma_Y$ have precisely $\operatorname{hom}^{\bullet}(X,Y)$ intersections;
			
			\item $\gamma_X$ has precisely $\frac{1}{2} \cdot \operatorname{hom}^{\bullet}(X,X)- 1$ self-intersections, \label{Point2List}
			
		\end{enumerate}
	\end{cor}

	\begin{notation}
		If $X,Y \in \mathcal{D}^b(\Lambda_n)$ are indecomposable and $p \in \gamma_X \overrightarrow{\cap} \gamma_Y$, let $f_p \in \Hom^{\bullet}(X,Y)$ denote the basis element associated with $p$. 
	\end{notation}	

\noindent \textbf{A remark on degrees}: If $(\gamma, g)$ and $(\gamma', g')$ are graded and $\dim \mathcal{V}=1=\dim \mathcal{V}'$, then every $p \in \gamma \overrightarrow{\cap} \gamma'$ defines an integer $\deg(p)$, namely the degree of $f_p: \P_{(\gamma, g)}(\mathcal{V}) \rightarrow \P_{(\gamma', g')}(\mathcal{V}')$.  Let $\Delta$ denote the $6$-gon which contains $p$, where as usual we assume that $\gamma, \gamma'$ and all laminates are  in minimal position. Let $\bullet \in \partial \Delta$ denote the unique marked point. 

Starting in $\bullet$ and following $\partial \Delta$ counter-clockwise, let us denote by $q$ the first intersection of $\gamma$ with a laminate and following $\partial \Delta$ further in the same direction, denote by $q'$ the first intersection of $\gamma'$ with a laminate after $q$. The following is a consequence of the constructions in \cite[Section 3.3]{OpperPlamondonSchroll}.
\begin{lem}\label{LemmaDegreeOfIntersections}
With the notation above, $\deg(p)= g'(q')-g(q)$.
\end{lem}

\subsection{Mapping Cones and resolutions of crossings} \ \medskip

\noindent Let $\Sigma$ be a marked surface. Suppose $\gamma, \gamma' \subseteq \Sigma$ are curves in minimal position and suppose $p \in \gamma \overrightarrow{\cap} \gamma'$ is an oriented intersection which is not a puncture. We obtain a new curve by resolving $p$ as follows. If $p \in \partial \Sigma$, denote by $\gamma_p$ the concatenation of $\gamma$ followed by $\gamma'$ at the end points which correspond to $p$. If $p$ lies in the interior, we resolve $p$ by cutting $\gamma$ and $\gamma'$ at $p$ and gluing them back together as shown in Figure \ref{FigureResolutionOfCrossings}. In this case, the resulting path $\gamma_p$ is not necessarily a curve but  a multi-curve, i.e.\ a finite set of possibly non-primitive curves. We refer to the individual curves of a multi-curve as its \textbf{components}. To simplify the notation in the next theorem, we say that an object $Z \in \mathcal{D}^b(\Lambda_n)$ is represented by a non-primitive loop $\gamma^t$, where $t > 1$ and $\gamma$ is primitive, if $Z \cong \P_{(\gamma, g)}(\mathcal{U})$ for some local system $\mathcal{U}$ corresponding to a polynomial $X^t-\lambda$, $\lambda \in \Bbbk^{\times}$. Note that for $p\coloneqq \characteristic \Bbbk$ and $k \geq 0$ such that $t= p^k \cdot  m$ with $\gcd(m, p)=1$, $\mathcal{U}$ splits into $m$ indecomposable local systems of dimension $p^k$. We further say that a multi-curve $\{\gamma_1, \dots, \gamma_m\}$ represents an object $X \in \mathcal{D}^b(\Lambda_n)$ if there exists a decomposition $X \cong \bigoplus_{i=1}^m X_i$ such that $\gamma_i$ represents $X_i$.

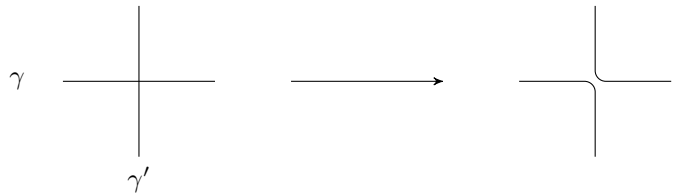
\begin{figure}
	\begin{displaymath}\begin{tikzpicture}
			\draw[white] (-1.5,0)--(7.5,0);
			
			\draw (0, -1)--(0,1);
			\draw (-1,0)--(1,0);
			\draw[->] (2,0)--(4,0);
			
			\draw[rounded corners] (6,-1)--(6,0)--(5,0);
			\draw[rounded corners] (6,1)--(6,0)--(7,0);			
			
			\draw (0,-1.3) node{$\gamma'$};
			\draw (-1.6, 0) node{$\gamma$};
			
		\end{tikzpicture}
	\end{displaymath}
	\caption{Resolution of an interior intersection from $\gamma$ to $\gamma'$.} \label{FigureResolutionOfCrossings}
\end{figure}
\noindent The following theorem states that the resolved curves represents the mapping cone of a morphism and is taken from \cite{OpperPlamondonSchroll} with a slight adjustment of the statement due to an inaccuracy in \cite{CanakciPauksztelloSchrollCorrigendum} which we comment on below.
\begin{thm}[c.f.\ {\cite[Theorem 4.1]{OpperPlamondonSchroll}}] \label{TheoremResolutionCrossings}
	Let $X, Y \in \mathcal{D}^b(\Lambda_n)$ be indecomposable, let $p \in \gamma_X \overrightarrow{\cap} \gamma_Y$ and let $\gamma_p$ denote the resolved multi-curve. Then, $\gamma_p$ represents the mapping cone of $f_p \in \Hom^{\bullet}(X,Y)$ in the sense above.
\end{thm}
 
\noindent Phrased in  our notation it was claimed in \cite[Lemma 2.5]{CanakciPauksztelloSchrollCorrigendum} that an object $\P_{(\gamma, g)}(\mathcal{U})$ with $\mathcal{U}$ corresponding to a polynomial $X^{\dim U}-\lambda$ splits into $\dim \mathcal{U}$ indecomposable direct summands of the form $\P_{(\gamma, g)}(\mathcal{U}_i)$ with $\dim \mathcal{U}_i=1$. However, as explained above, this is only true if the characteristic of $\Bbbk$ does not divide $\dim \mathcal{U}$.

\subsection{Compositions of morphisms via triangles, bigons and tridents}
\ \medskip

\noindent In this section, we recall the geometric description of compositions of morphisms $\mathcal{D}^b(\Lambda_n)$. Suppose that $\gamma_1, \gamma_2, \gamma_3 \subseteq \mathbb{T}_n$ are curves  which have oriented intersections $p \in \gamma_1 \overrightarrow{\cap} \gamma_2$ and $q \in \gamma_2 \overrightarrow{\cap} \gamma_3$. Let us denote by $\widetilde{\mathbb{T}}_n$ a fixed universal cover of $\mathbb{T}_n$.  By a lift of a loop $\gamma: S^1 \rightarrow \mathbb{T}_n$, we mean a lift $\mathbb{R} \rightarrow \widetilde{\mathbb{T}}_n$ along the universal cover $\mathbb{R} \rightarrow S^1$.\medskip

\begin{definition}
\noindent  Let $\Delta \subseteq \mathbb{R}^2$ be a triangle equipped with the induced orientation and denote by $S_1, S_2, S_3$ the three segments of $\partial \Delta$ in clockwise order. Moreover, we denote by $D^2 \subseteq \mathbb{R}^2$ the closed unit disc. Let $r \in \gamma_1 \overrightarrow{\cap} \gamma_3$. Then, the triple $(p, q, r)$ is called
\begin{enumerate}

	\item an \textbf{intersection triangle}, if there exists an orientation-preserving embedding $\varphi: \Delta \rightarrow \widetilde{\mathbb{T}}_n$ such that $\varphi|_{S_i}$ is a homeomorphism onto a segment of a lift of $\gamma_i$ and the corners of $\Delta$ are mapped bijectively onto  $\{p, q, r\}$;
	
	\item a \textbf{double-bigon}, if all curves are arcs and there exists an orientation-preserving embedding $\varphi: D^2 \rightarrow \widetilde{\mathbb{T}}_n$ which maps $S^1 \cap \left(\mathbb{R} \times \mathbb{R}_{\geq 0}\right)$, $[-1,1] \times \{0\}$ and $S^1 \cap \left(\mathbb{R} \times \mathbb{R}_{\leq 0}\right)$  bijectively onto lifts of $\gamma_1, \gamma_2$ and $\gamma_3$, respectively; in particular, $\varphi$ maps $p, q$ and $r$ to the same point in $\{\pm 1\} \subseteq D^2$.
		\item a \textbf{trident} if $p, q, r$ correspond to the same point on the boundary and locally the curves are arranged as in Figure \ref{FigureIntersectionTriangle}.
\end{enumerate}
\end{definition}

\begin{definition}\label{DefinitionSetComposition}
	In the notation above, define $C(p,q)$ as the set of all $r \in \gamma_1 \overrightarrow{\cap} \gamma_3$ such that $(p,q,r)$ is an intersection triangle, a double-bigon or a trident.
\end{definition}

	\begin{figure}
		\begin{displaymath}\arraycolsep=10pt
		\begin{array}{ccc}
		{		
			\begin{tikzpicture}
			\draw[white] (0,2)--(0,-2);
			\foreach \i in {1,2,3}
			{
				\draw  ({2*cos(90+\i * 120)},{2*sin(90+\i *120)})--({2*cos(210+\i * 120)},{2*sin(210+\i *120)});
				\filldraw ({2*cos(90+\i * 120)},{2*sin(90+\i *120)}) circle (1.5pt);
		
				\draw ({1.3*cos(150-\i * 120)},{1.3*sin(150-\i *120)}) node{$\widetilde{\gamma}_{\i}$};
			}
			
			\draw ({2.35*cos(90-1 * 120)},{2.35*sin(90-1 *120)}) node{$\widetilde{p}$};
				\draw ({2.35*cos(90-2 * 120)},{2.35*sin(90-2 *120)}) node{$\widetilde{q}$};
					\draw ({2.35*cos(90-3 * 120)},{2.35*sin(90-3 *120)}) node{$\widetilde{r}$};

			\draw[->] ({2*cos(90-1 * 120)+1*cos(123)},{2*sin(90-1 * 120)+1*sin(123)}) arc (123:177:1) node[pos=0.5, left]{$f_p$};
			\draw[->, dashed] ({2*cos(90-2 * 120)+1*cos(3)},{2*sin(90-2 * 120)+1*sin(3)}) arc (3:57:1) node[pos=0.5, right]{$f_q$};
					\draw[<-, dashed] ({2*cos(90-3 * 120)+1*cos(243)},{2*sin(90-3 * 120)+1*sin(243)}) arc (243:297:1) node[pos=0.5, below]{$f_r$};
			
						\end{tikzpicture}
		} & 
		{				\begin{tikzpicture}
			\draw[white] (0,2)--(0,-2.5);

			\filldraw (-1,0) circle (2pt);
			
			\foreach \i in {1,2,3}
			{
				\pgfmathsetmacro\v{\i-2}
				\draw (-1,0)--({3*cos(180+\v*30)},{3*sin(180+\v*30)});
				\draw ({3.35*cos(180+\v*30)},{3.35*sin(180+\v*30)}) node{$\gamma_{\i}$};
			} 
			\draw[->] (-0.5,0) circle (0.5) node{$\partial \mathbb{T}_n$};

			\draw[->, dashed] ({-1+1.75*cos(140)},{1.75*sin(140)}) arc(140:177:1.75) node[pos=0.4, left]{$f_p$};
				\draw[->, dashed] ({-1+1.75*cos(183)},{1.75*sin(183)}) arc(183:220:1.75) node[pos=0.6, left]{$f_q$};
					\draw[->, dashed] ({-1+1*cos(140)},{1*sin(140)}) arc(140:220:1) node[pos=0.25, left]{$f_r$};
			\end{tikzpicture}
		}   &
		
		{\begin{tikzpicture}[hobby]
			\draw[white] (-3.5,0) circle(2pt);
			\draw[white] (0,2)--(0,-2.5);
			
			\filldraw (-1,0) circle (2pt);
			\filldraw (-3,0) circle (2pt);
			\foreach \i in {1,2,3}
			{
				\draw [color=black] plot  [tension=1] coordinates {   (-1,0) (-2,2-\i) (-3,0)};	
			}

			\end{tikzpicture}}
		\end{array}
		\end{displaymath}
		\caption{From left to right: an intersection triangle, a trident and a double-bigon.} \label{FigureIntersectionTriangle}
	\end{figure}
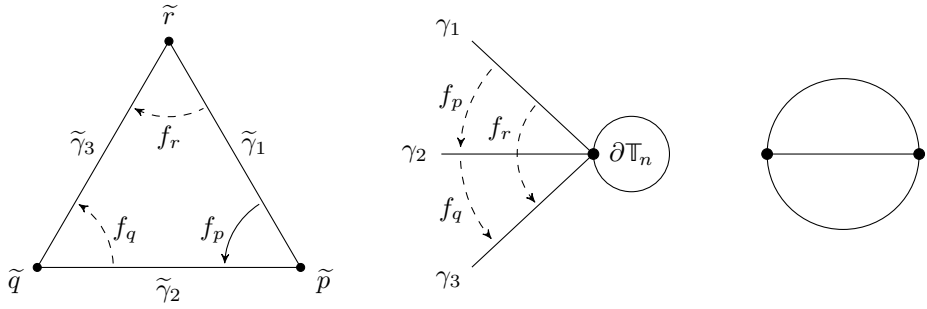

	\begin{definition}
		Let $\gamma, \gamma' \subseteq \mathbb{T}_n$ be curves and let $f: X_{\gamma} \rightarrow X_{\gamma'}$ be a morphism. The \textbf{support of $f$} is the subset $\supp f \subseteq \gamma \overrightarrow{\cap} \gamma'$ consisting of all $p \in \gamma \overrightarrow{\cap } \gamma'$ such that $f_p$ appears with non-zero scalar in a decomposition of $f$ with respect to the basis of $\Hom^{\bullet}(X,Y)$ associated with $\gamma \overrightarrow{\cap} \gamma'$.
	\end{definition}

\noindent  Compositions of morphisms can now be described in the following way:
\begin{prp}\cite[Theorem 2.8.]{OpperDerivedEquivalences}
For curves $\gamma_1, \gamma_2, \gamma_3 \subseteq \mathbb{T}_n$, let $p \in \gamma_1 \overrightarrow{\cap} \gamma_2$ and $q \in \gamma_2 \overrightarrow{\cap} \gamma_3$ be intersections and let $f_p: X_{\gamma_1} \rightarrow X_{\gamma_2}$, $f_q: X_{\gamma_2} \rightarrow X_{\gamma_3}$ denote the corresponding morphisms. If $\gamma_1  \not \simeq \gamma_2$ or $\gamma_1$ is an arc, then 

$$\supp \left(f_q \circ f_p\right)= C(p, q).$$
\end{prp}

\noindent For later reference we want to emphasize the following observation:

\begin{lem}\label{LemmaInteriorMorphismsClosedUnderComposition}
	Let $f, g$ be  morphisms between indecomposable objects in $\mathcal{D}^b(\Lambda_n)$ such that $g \circ f$ is defined. If $f$ or $g$  is supported only at interior intersections, then so is  $g \circ f$. 
	\end{lem}
	\begin{proof}
		The assumptions imply that no trident or double-bigon appears in the set $C(p,q)$. Note that the point $\widetilde{r}$ in Figure \ref{FigureIntersectionTriangle} cannot be a boundary intersection due to the definition of oriented intersections at the boundary.
	\end{proof}

\section{Spherical objects as simple loops on tori}\label{SectionSphericalObjectsTori}
	
	\noindent In this section, we show that isomorphism classes of spherical objects in $\mathcal{D}^b(\Lambda_n)$ are -- up to shift --  in bijection with a certain set of simple loops on $\mathbb{T}_n$. We call a curve on a marked surface \textbf{simple} if all its self-intersections lie on the boundary. In other words, a simple loop has no self-intersections and a simple arc can only intersect itself at its end points.
	
	\begin{lem}\label{LemmaImageOfSphericalIsSpherical}
		Let $Y \in \mathcal{D}^b(\Lambda_n)$. Then, $Y$ is spherical if and only if $Y \cong \mathbb{F}(X)$ for some spherical object $X \in \Perf(C_n)$. In particular, every spherical object in $\mathcal{D}^b(\Lambda_n)$ is $1$-spherical.
	\end{lem}
	\begin{proof}
		This follows from Theorem \ref{TheoremBurbanDrozdImage} and Proposition \ref{PropositionClassificationCalabiYauObjects}.
	\end{proof}

	\begin{lem}
		Let $Y \in \mathcal{D}^b(\Lambda_n)$. Then $Y$ is spherical if and only if $Y \cong \P_{(\gamma, g)}(\mathcal{V})$, where $\gamma$ is simple and $\dim \mathcal{V}=1$.
	\end{lem}
	\begin{proof} By \cite[Proposition 5.16]{ArnesenLakingPauksztello} we have  $\dim \Hom^{\bullet}(X,X) \geq 3 $ for every $\tau$-invariant indecomposable object $X \in \mathcal{D}^b(\Lambda_n)$ such that $\dim \mathcal{V}_X \geq 2$. Thus, such objects are never spherical and the assertion is a consequence of Corollary \ref{CorollaryMorphisms}.
	\end{proof}
	
	\noindent It still remains to be discussed which simple loops are gradable and hence correspond to a family of spherical objects.\medskip

	\noindent A simple loop $\gamma$ on a surface $\Sigma$ is said to be \textbf{non-separating} if its complement $\Sigma \setminus \gamma$ is connected. Otherwise, $\gamma$ is called \textbf{separating}. 
	\begin{lem}\label{LemmaHopf}
		Let $\gamma \subseteq \mathbb{T}_n$ be a separating simple loop. Then $\omega(\gamma) \neq 0$. In particular, $\gamma$ does not represent an object of $\mathcal{D}^b(\Lambda_n)$.
	\end{lem}
	\begin{proof}This is a consequence of the fact  that $\omega$ is the winding number function of a line field, e.g.\ see \cite[Lemma 2.4]{OpperDerivedEquivalences}. Since $\gamma$ is separating and by additivity of the Euler characteristic, $\gamma$ bounds a unique subsurface $\Sigma$ of genus $0$. We denote by $B_1, \dots, B_b$ the boundary components of $\Sigma$ and assume that $\gamma$ inherits its orientation from $\Sigma$. Since all boundary winding numbers are equal to $-2$ (see Example \ref{WindingNumbersBoundary}), it follows from the Poincar\'e-Hopf index theorem \cite{Hopf} that
		\[\omega(\gamma)  - 2 (b-1) = \sum_{i=1}^b{\omega(B_i)}=2 \chi(\Sigma)=4-2b \]
		and hence $\omega(\gamma)=2$. 
	\end{proof}

	\begin{cor}\label{CorollarySphericlaCurvesAreNonSeparating}
		Let $X \in \Perf(C_n)$ be indecomposable. Then, $X$ is spherical if and only if $\gamma_X$ is a simple, non-separating loop and $\dim \mathcal{V}_X =1$.
	\end{cor}
	
	\noindent We prove in the next section that every non-separating simple loop on $\mathbb{T}_n$ is gradable and hence represents an object of $\mathcal{D}^b(\Lambda_n)$.
	
	\section{The mapping class group of a torus with boundary}\label{SectionMappingClassGroup}
	\noindent This section discusses the mapping class group of $\mathbb{T}_n$ and its connection to auto-equivalences of $\mathcal{D}^b(\Lambda_n)$. We recommend \cite{FarbMargalit} for additional information about mapping class groups.
	
	\subsection{Mapping class groups and their connection to auto-equivalences}
	
	\noindent By an \textbf{isotopy}, we mean a smooth path of diffeomorphisms which is constant on marked points. To spell this out, an isotopy is a smooth map $I:[0,1] \times \mathbb{T}_n \rightarrow \mathbb{T}_n$ such that for all $t \in [0,1]$, the map $I(t,-): \mathbb{T}_n \rightarrow \mathbb{T}_n$ is a diffeomorphism and such that $I(t, x)=x$ for all $t \in [0,1]$ and all marked points $x \in \mathbb{T}_n$. 
	
	\begin{definition}\label{DefinitionMappingClassGroup}
		The \textbf{mapping class group} $\MCG(\mathbb{T}_n)$ is the group consisting of all isotopy classes of orientation-preserving diffeomorphisms $H:\mathbb{T}_n \rightarrow \mathbb{T}_n$ which preserves the set of marked points. The \textbf{pure mapping class group} $\PMCG(\mathbb{T}_n)$ is the subgroup of all those diffeomorphisms which restrict to the identity on $\partial \mathbb{T}_n$.	
	\end{definition}

\noindent 	The mapping class group and the pure mapping class group fit into an exact sequence

		\begin{equation}\label{ShortExactSequenceMCG}\begin{tikzcd}\mathbf{1} \arrow{r} & \PMCG(\mathbb{T}_n) \arrow{r} & \MCG(\mathbb{T}_n) \arrow{r}{|_{\marked}} & \mathfrak{S}_n \ltimes \mathbb{Z}_2^n \arrow{r} & \mathbf{1},\end{tikzcd}\end{equation}
		\noindent where $|_{\marked}$ restricts a diffeomorphism to the set of marked points yielding an element in $\mathfrak{S}_n \ltimes \mathbb{Z}_2^n$.  A transposition in $\mathfrak{S}_n$ is realized by a so-called \textit{half-twist}  which permutes a pair of boundary components, see Figure \ref{FigureHalfTwists}.
		\begin{figure}[H]
			\centering
			\begin{tikzpicture}[scale=0.3]
			\draw (0,0) circle (6);
			\draw (-2, 0) circle(1);
			\draw (2, 0) circle(1);
			
			\draw[thick, dashed] (-6,0)--(-3,0);
			\draw[thick, dashed] (6,0)--(3,0);
			\draw[->] (8,0)--(14,0);
			
			\draw ({0+22},0) circle (6);
			\draw ({-2+22}, 0) circle(1);
			\draw ({2+22}, 0) circle(1);

			\draw[scale=1,domain=0:1,smooth, dashed,variable=\t, thick] plot ({22+(-6+3*\t)* cos(deg((6-6+3*\t)/3 * pi))},{(-6+3*\t)* sin(deg((6-6+3*\t)/3 * pi))});
			\draw[scale=1,domain=0:1,smooth, dashed,variable=\t, thick] plot ({22+(6-3*\t)* cos(deg((6-6+3*\t)/3 * pi))},{(6-3*\t)* sin(deg((6-6+3*\t)/3 * pi))});        
			
			\draw (-2,0) node{$A$};
			\draw (2,0) node{$B$};
			
			\draw[dashed] (-1,0)--(1,0);	
			\draw[dashed] (21,0)--(23,0);
				
			\draw (20,0) node{$B$};
			\draw (24,0) node{$A$};

			\draw[blue, very thick] (0,6)--(0,-6);
			\begin{scope}[rotate around={90:(22,0)}]
			\draw[blue, very thick] ({22+(-6+3*0)* cos(deg((6-6+3*0)/3 * pi))},{(-6+3*0)* sin(deg((6-6+3*0)/3 * pi))}) to[curve through ={({22+(-6+3*0.25)* cos(deg((6-6+3*0.25)/3 * pi))},{(-6+3*0.25)* sin(deg((6-6+3*0.25)/3 * pi))}) . . ({22+(-6+3*0.5)* cos(deg((6-6+3*0.5)/3 * pi))},{(-6+3*0.5)* sin(deg((6-6+3*0.5)/3 * pi))}) . . ({22+(-6+3*0.9)* cos(deg((6-6+3*0.9)/3 * pi))},{(-6+3*0.9)* sin(deg((6-6+3*0.9)/3 * pi))}) ..  ({22+(6-3*(1-0.1))* cos(deg((6-6+3*(1-0.1))/3 * pi))},{(6-3*(1-0.1))* sin(deg((6-6+3*(1-0.1))/3 * pi))}) .. ({22+(6-3*(1-0.5))* cos(deg((6-6+3*(1-0.5))/3 * pi))},{(6-3*(1-0.5))* sin(deg((6-6+3*(1-0.5))/3 * pi))}) .. ({22+(6-3*(1-0.75))* cos(deg((6-6+3*(1-0.75))/3 * pi))},{(6-3*(1-0.75))* sin(deg((6-6+3*(1-0.75))/3 * pi))})
			}] ({22+(6-3*(1-1))* cos(deg((6-6+3*(1-1))/3 * pi))},{(6-3*(1-1))* sin(deg((6-6+3*(1-1))/3 * pi))});
			\end{scope}
			
			\end{tikzpicture}
			
			\caption{The action of a half-twist on curves. It acts as the identity outside of a neighborhood of $A$ and $B$.}
			\label{FigureHalfTwists}
		\end{figure}
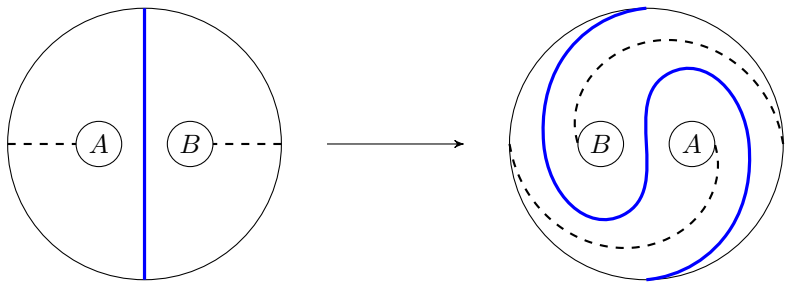
\noindent 		Another example of a mapping class is the fractional twist around a boundary component $B$. It rotates the surface in a collar neighborhood of $B$ and acts as the identity outside of this neighborhood, see Figure \ref{FigureFractionalTwist}.
		
		\begin{figure}[H]
\begin{displaymath}
\begin{tikzpicture}[scale=0.6]
\draw[decoration={markings, mark=at position 0.4 with {\arrow{<}}},
postaction={decorate}
] (0,0) circle (1);
\foreach \x in {0,2}
{
	\filldraw ({1*cos(\x*180*0.5)},{1*sin(\x*180*0.5)}) circle (2pt);
	\draw[dashed] ({3*cos(\x*180*0.5)},{3*sin(\x*180*0.5)})--({1*cos(\x*180*0.5)},{1*sin(\x*180*0.5)});
}
\draw (0,0) circle (3);
\draw[->] (4, 0)--(7,0);

\draw (11,0) circle (3);
\draw[decoration={markings, mark=at position 0.4 with {\arrow{<}}},
postaction={decorate}
] (11,0) circle (1);
\begin{scope}[hobby]
\foreach \x in {0,2}
{
	\filldraw ({1*cos(\x*180*0.5)+11},{1*sin(\x*180*0.5)}) circle (2pt);
	\draw[thick, dashed] plot coordinates {({3*cos(\x*180*0.5)+11},{3*sin(\x*180*0.5)})  ({7/3)*cos(\x*180*0.5+60)+11},{(7/3)*sin(\x*180*0.5+60)}) ({(5/3)*cos(\x*180*0.5+120)+11},{(5/3)*sin(\x*180*0.5+120)}) ({1*cos(\x*180*0.5+180)+11},{1*sin(\x*180*0.5+180)})};
}
\end{scope}
\end{tikzpicture}
\end{displaymath}
\caption{The action of a counter-clockwise fractional twist on arcs.} \label{FigureFractionalTwist}
\end{figure}
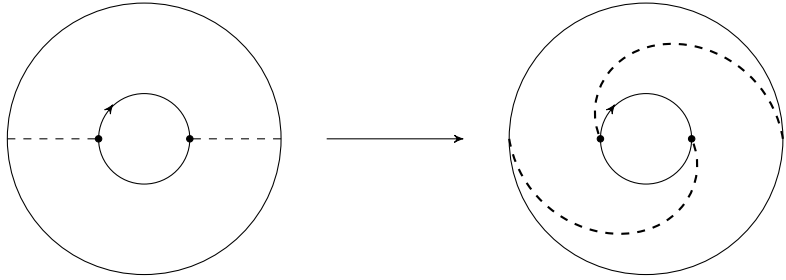

	\begin{rem}\label{RemarkSmoothTopologicalMappingClassGroup}
		Replacing diffeomorphisms by homeomorphisms and smooth maps by continuous maps in Definition \ref{DefinitionMappingClassGroup} leads to isomorphic groups. Below we will switch freely  between ``homeomorphisms'' and ``diffeomorphisms'' and will not distinguish between continuous and smooth isotopies. 
	\end{rem}

\noindent It turns out that there is a close relationship between auto-equivalences of $\mathcal{D}^b(\Lambda_n)$ and mapping classes of its geometric model $\mathbb{T}_n$. The following is a special case of \cite[Theorem C]{OpperDerivedEquivalences}.
	\begin{prp}\label{PropositionExactSequenceAutoequivalences}
	There exists a short exact sequence
		\begin{equation}\label{AutoEquivalenceSequence}\begin{tikzcd}\mathbf{1} \arrow{r} & \left(\Bbbk^{\times}\right)^{n+1} \times \mathbb{Z} \arrow{r} & \Aut\left(\mathcal{D}^b(\Lambda_n)\right) \arrow{r}{\Psi} &  \MCG(\mathbb{T}_n, \omega) \arrow{r} & \mathbf{1},\end{tikzcd}\end{equation}
\noindent where $\MCG(\mathbb{T}_n, \omega) \subseteq \MCG(\mathbb{T}_n)$ denotes the subgroup of all mapping classes which preserve the winding numbers of all loops.		
	\end{prp}
	\noindent The definition of $\Psi$ is recalled in Section \ref{SectionDiffeomorphismsPuncturedTori}. The $\mathbb{Z}$-component in the kernel of the sequence \eqref{AutoEquivalenceSequence}  acts as the shift functor on $\mathcal{D}^b(\Lambda_n)$, the other component by multiplying the arrows $c_0, \dots, c_{n-1}$ and $d_0$ in the quiver of $\Lambda_n$ with non-zero scalars. For any auto-equivalence $T \in \Aut(\mathcal{D}^b(\Lambda_n))$, the mapping class $\Psi(T)$ is uniquely determined by the following property:
	\begin{prp}[{\cite[Theorem A]{OpperDerivedEquivalences}}]\label{PropositionTheoremA} For all indecomposable objects $X \in \mathcal{D}^b(\Lambda_n)$,
$$
\Psi(T)\left(\gamma_X\right) \simeq \gamma_{T(X)}.
$$
\end{prp}		
		
\noindent In other words, the action of $\Psi(T)$ on curves on $\mathbb{T}_n$ mirrors the action of $T$ on indecomposable objects of $\mathcal{D}^b(\Lambda_n)$.
	
	\begin{lem}\label{LemmaTranspositionsAreInStabilizer}
	Half-twists and fractional twists are elements of $\MCG(\mathbb{T}_n, \omega)$.
	\end{lem}
	\begin{proof}
	Fractional twists act trivially on the homotopy class of any loop and hence preserve all winding numbers. A half-twist which permutes boundary components $A$ and $B$ stabilizes an embedded path $\delta$ which connects $A$ and $B$ and inverts its orientation. For example, in Figure \ref{FigureHalfTwists}, $\delta$ is the horizontal line between $A$ and $B$ which crosses the blue vertical line transversely. A curve $\gamma$ is affected by the half-twist if and only if $\gamma$ and $\delta$ intersect. Locally around an intersection, a segment of $\gamma$ (the blue vertical curve in Figure \ref{FigureHalfTwists}) is replaced by the blue thick curve on the right hand side of Figure \ref{FigureHalfTwists}. As a result, a half-twist preserves all winding numbers if and only if the winding number $u$ of the loop in Figure \ref{FigureSpecialCurveHalfTwist} vanishes. This can be verified by hand or, more conceptually, by the argument explained in Remark \ref{RemarkDoubleLoopOfAnArcHasVanishingWindingNumber}.	\end{proof}
	
	\begin{figure}[H]
			\centering
			\begin{tikzpicture}[scale=0.3]
			\draw ({0},0) circle (6);
			\draw ({-2}, 0) circle(1);
			\draw ({2}, 0) circle(1);

			\draw (-2,0) node{$A$};
			\draw (2,0) node{$B$};


			\begin{scope}[closed hobby]
			\draw[very thick] plot coordinates {(2, 2) (4,0) (2,-2) (0,0) (-2,2) (-4,0) (-2,-2) (0,0)  };
			\end{scope}
			
			\end{tikzpicture}
			\caption{}
			\label{FigureSpecialCurveHalfTwist}
		\end{figure}

		\begin{rem}\label{RemarkDoubleLoopOfAnArcHasVanishingWindingNumber}
Let  $\gamma:[0,1] \rightarrow \mathbb{T}_n$ be an oriented (but not necessarily embedded) arc which cannot be homotoped onto $\partial \mathbb{T}_n$. One can define a gradable loop $\gamma_{\operatorname{loop}}$ as in Figure \ref{FigureSpecialCurveHalfTwist} by concatenating $\gamma$ with the clockwise simple closed boundary arc based at $\gamma(1)$, followed by the inverse $\gamma^{-1}$ and the counter-clockwise simple closed boundary arc based at $\gamma(0)$.  Using the homological definition of $\omega$ one can give a simple argument for the vanishing of $\omega(\gamma_{\operatorname{loop}})$: there exists a homomorphism $\overline{\omega}: H_1(\mathbb{P}(T\mathbb{T}_n), \mathbb{Z}) \rightarrow \mathbb{Z}$, where $\mathbb{P}(T\mathbb{T}_n)$ denotes the projectivized tangent bundle. More precisely, $\overline{\omega}$ is the resulting homomorphism of the intersection pairing with the submanifold of $\mathbb{P}(T\Sigma)$ determined by a line field. Further details on the definition of $\overline{w}$ (in the case of compact surfaces) can be found in \cite{LekiliPolishchukGentle}. The derivative $\dot \gamma$ of an immersed loop $\gamma$ in minimal position defines a simple loop in $\mathbb{P}(T\mathbb{T}_n)$ and hence defines a class $[\dot{\gamma}] \in H_1(\mathbb{P}(T\mathbb{T}_n), \mathbb{Z})$. Then $\omega(\gamma)=\overline{\omega}([\dot{\gamma}])$ and the claim follows from the observation that $[\dot{\gamma_{\operatorname{loop}}}]=[\dot \delta_0]-[\dot \delta_1]$, where $\delta_0$ and $\delta_1$ denote the clockwise simple loops around the boundary components which contain $\gamma(0)$ and $\gamma(1)$ respectively.
		\end{rem}

	\subsection{Dehn twists, Humphries generators and their spherical twists}\label{SectionDehnTwists} An important example of a mapping class and the geometric counterpart of spherical twists is the \textit{Dehn twist} about a  given simple loop $\gamma$.\medskip
	
	\noindent Let $W$ be a  tubular neighbourhood of $\gamma$, i.e.\ a neighbourhood $W$ with an orientation-preserving diffeomorphism $\phi:S^1 \times [-1,1] \rightarrow W$ such that $\phi|_{S^1 \times \{0\}}=\gamma$. Then, the \textbf{Dehn twist} about $\gamma$ is the mapping class of the diffeomorphism $D_{\gamma}: \mathbb{T}_n \rightarrow \mathbb{T}_n$ defined  on $\mathbb{T}_n \setminus W$ by $D_{\gamma}|_{\mathbb{T}_n \setminus W}:=\text{Id}_{\mathbb{T}_n \setminus W}$  and on $W$ by 
	
	\[ D_{\gamma}(\phi(z, t))\coloneqq \phi\left(z \cdot e^{2\pi i (t+1)}, t\right).\]
	
	\noindent While the diffeomorphism $D_{\gamma}$ depends on $W$, its mapping class is well-defined. Since the map $\phi$ above is assumed to be orientation preserving, the mapping class of $D_{\gamma}$ does not depend on the orientation of $\gamma$.\medskip
	
	\noindent In analogy to spherical twists (see Corollary \ref{CorollaryTwistFunctorsUnderEmbeddings}) the set of Dehn twists is closed under conjugation.
	\begin{lem}\label{LemmaConjugationOfDehnTwists}
		Let $\gamma$ be a simple loop on an oriented surface $\mathbb{T}_n$ and let $F: \mathbb{T}_n \rightarrow \mathbb{T}_n$ be a diffeomorphism. Then, $D_{F \circ \gamma}$ and $F \circ D_{\gamma} \circ F^{-1}$ define the same mapping classes.
	\end{lem}
	\noindent By classical results of Dehn and Humphries, the group $\PMCG(\mathbb{T}_n)$ is generated by a finite set of Dehn twists. We refer to the Dehn twists about the loops $\gamma_{\Pic}$ and  $\gamma_{\Bbbk(x)}^0, \dots, \gamma_{\Bbbk(x)}^{n-1}$ as the \textbf{Humphries generators}.\medskip
	
	\noindent  The following is a special case of a theorem due to Humphries. 
	\begin{thm}[Humphries]\label{TheoremHumphriesGenerators}
		$\PMCG(\mathbb{T}_n)$ is generated by the Humphries generators, i.e.\ the mapping classes of the Dehn twists about the loops $\gamma_{\Pic}, \gamma_{\Bbbk(x)}^0, \dots, \gamma_{\Bbbk(x)}^{n-1}$ (see Example \ref{ExampleLoopsOnTorus}) as well as the Dehn twists about the simple boundary loops.
	\end{thm}
	\begin{proof}We include a proof in lack of a suitable reference. We regard the surface $\mathscr{T}^n$ as being obtained from $\mathbb{T}_n$ by gluing a once-punctured disc to every boundary component. If we denote by $B_0, \dots, B_{n-1}$ the components of $\partial \mathbb{T}_n$, then \cite[Proposition 3.19]{FarbMargalit} asserts the existence of a short exact sequence
		
		\begin{equation}\label{ShortExactSequenceMCGInclusion}\begin{tikzcd}\mathbf{1} \arrow{r} & \mathbb{Z}^n \arrow{r}{\iota} & \PMCG(\mathbb{T}_n) \arrow{r}{\pi} & \PMCG(\mathscr{T}^n) \arrow{r} & \mathbf{1}.\end{tikzcd}\end{equation}
		\noindent Here, $\PMCG(\mathscr{T}^n)$ denotes the pure mapping class group of $\mathscr{T}^n$, i.e.\ the group of self-diffeomorphisms which fix every puncture modulo isotopies which are constant on all punctures.
		
		The inclusion $\iota$ sends a tuple $(m_i) \in \mathbb{Z}^n$ to the composition of Dehn twists $\prod_{i=0}^{n-1}D_{B_i}^{m_i}$, where $D_{B_i}$ denotes the Dehn twist around the boundary loop of $B_i$. The projection $\pi$ sends a diffeomorphism to its radial extension: every self-diffeomorphism $\varphi$ of the annulus $S^1$ extends to a self-diffeomorphism $\overline{\varphi}$ of $D^2 \setminus \{(0,0)\}$ defined by
		\[\overline{\varphi}(x)\coloneqq \lVert x\rVert  \cdot \varphi\left(\frac{x}{\lVert x \rVert}\right).\]

		\noindent By \cite[Corollary 4.15]{FarbMargalit} and the discussion leading to it, $\PMCG(\mathscr{T}^n)$ is generated by the Humphries generators.\end{proof}
		
\noindent We note that every Dehn twist about a boundary loop acts trivially on any homotopy class of loops.\footnote{By deforming any given loop $\gamma$ we may assume that the tubular neighbourhood in Section \ref{SectionDehnTwists} of the boundary loop and $\gamma$ are disjoint.} In particular, we find the following.
\begin{prp}\label{PropositionBoundaryTwistsStabiliser}
Every Dehn twist about a boundary loop is an element of $\MCG(\mathbb{T}_n,\omega)$.
\end{prp}

	\noindent The relevance of the Humphries generators for us and the relationship between spherical twists and Dehn twists becomes apparent in the following application of \cite[Theorem C]{OpperDerivedEquivalences} and Proposition \ref{PropositionTheoremA}.
	\begin{thm}\label{TheoremEquationForPsi}
	
Let $X \in \mathcal{D}^b(\Lambda_n)$ be spherical. Then $\Psi(T_X)= D_{\gamma_X}$ and for every indecomposable object $Y \in \mathcal{D}^b(\Lambda_n)$, 
$$\gamma_{T_X(Y)} \simeq D_{\gamma_X}(\gamma_Y).$$
\end{thm}

 \noindent The following proposition is a direct consequence of Theorem \ref{TheoremEquationForPsi}, Theorem \ref{TheoremImages} and Corollary  \ref{CorollaryTwistFunctorsUnderEmbeddings}.
	
		\begin{prp}\label{PropositionSphericalTwistsBecomeDehnTwists}
		Let $x \in \mathbb{P}^1_j \subseteq C_n$ be closed and smooth and let $\mathcal{L}\in \Pic^{\mathbb{0}}(C_n)$, i.e.\ $\mathcal{L}$ is a line bundle of multi-degree $(0, \dots, 0)$. Let  $Y \in \Perf(C_n)$ be indecomposable. Then, $D_{\gamma_{\Bbbk(x)}^j}(\gamma_Y)$ is a representative of $T_{\Bbbk(x)}(Y)$ and $D_{\gamma_{\Pic}}(\gamma_Y)$ is a representative of $T_{\mathcal{L}}(Y)$.
	\end{prp}

\begin{cor}\label{CorollaryStabilizer}
$\MCG(\mathbb{T}_n)=\MCG(\mathbb{T}_n, \omega)$.
\end{cor}
\begin{proof}Let $H \in \MCG(\mathbb{T}_n)$. By \eqref{ShortExactSequenceMCG}, there exists $F \in \MCG(\mathbb{T}_n)$, which is a composition of half-twists and fractional twists, such that $H \circ F \in \PMCG(\mathbb{T}_n)$. Since $\Img \Psi \subseteq \MCG(\mathbb{T}_n, \omega)$, it follows from Theorem \ref{TheoremImages} and Theorem \ref{TheoremEquationForPsi} that the Dehn twists around the Humphries generators are elements of $\MCG(\mathbb{T}_n,\omega)$. It therefore follows from Theorem \ref{TheoremHumphriesGenerators} and Proposition \ref{PropositionBoundaryTwistsStabiliser} that $\PMCG(\mathbb{T}_n) \subseteq \MCG(\mathbb{T}_n, \omega)$. Finally, Lemma \ref{LemmaTranspositionsAreInStabilizer} implies $F \in \MCG(\mathbb{T}_n, \omega)$ and hence $H \in \MCG(\mathbb{T}_n, \omega)$.
\end{proof}

	\noindent  With Theorem \ref{TheoremHumphriesGenerators} and Proposition  \ref{PropositionSphericalTwistsBecomeDehnTwists} in mind, we show in Proposition \ref{PropositionTransitivityNonSeparatingLoops} that $\MCG(\mathbb{T}_n)$ acts transitively on non-separating simple loops. The proof is essentially the same as on page 37 in \cite{FarbMargalit}. However, we include it as we need a slightly stronger statement than what is proved there.
	
	\begin{prp}\label{PropositionTransitivityNonSeparatingLoops} The pure mapping class group $\PMCG(\mathbb{T}_n)$ acts transitively on the set of non-separating simple loops. 
	\end{prp}
	\begin{proof}
		Suppose  that $\gamma$ and $\gamma'$ are non-separating simple loops on $\mathbb{T}_n$. Then, the surfaces $\Sigma$ and $\Sigma'$, obtained from $\mathbb{T}_n$ by cutting at $\gamma$ and $\gamma'$ respectively, are both connected surfaces of genus $0$ with exactly $n+2$ boundary components. Note that the boundary components $B_1, \dots, B_n$ of $\mathbb{T}_n$ are canonically identified with boundary components of $\Sigma$. We denote by $D_1, D_2$ (resp. $D_1', D_2'$)  the remaining two components of $\partial \Sigma$ (resp. $\partial \Sigma'$).
		
		Since $\Sigma$ and $\Sigma'$ have the same genus and the same number of boundary components, there exists a diffeomorphism $\phi: \Sigma \rightarrow \Sigma'$. After composing $\phi$ with any orientation-reversing self-diffeomorphism of $\Sigma$ or $\Sigma'$, we may assume that $\phi$ is orientation-preserving. By a series of half-twists, one can permute any two boundary components on $\Sigma$ by an orientation-preserving diffeomorphism which restricts to the identity on all other boundary components. In this way, we obtain an orientation-preserving diffeomorphism $\phi: \Sigma \rightarrow \Sigma'$, which restricts to orientation-preserving diffeomorphisms $B_i \rightarrow B_i$ for every $i \in [1,n]$. By Lemma \ref{LemmaExtendBoundarydiffeomorphisms} below, there exists an orientation-preserving diffeomorphism $\psi: \Sigma \rightarrow \Sigma$ such that $\phi \circ \psi$ restricts to the identity $B_i \rightarrow B_i$ for all $i \in [1,n]$. Identifying $D_1$ with $D_2$ and $D_1'$ with $D_2'$ respectively, one sees that $\phi \circ \psi$ gives rise to a homeomorphism $H: \mathbb{T}_n \rightarrow \mathbb{T}_n$ such that $H(\gamma)=\gamma'$ and  which restricts to the identity on $B_i$ for all $i \in [1,n]$. \end{proof}
		
	\noindent The following well-known lemma was used in the previous proof.
	
	\begin{lem}\label{LemmaExtendBoundarydiffeomorphisms}
		Let $\Sigma$ be a compact, oriented surface, $B$ a boundary component of $\Sigma$ and let $f: B \rightarrow B$ be an orientation-preserving diffeomorphism, where $B$ is equipped with the induced orientation. Then, for every closed neighborhood $W$ of $B$, $f$ extends to a diffeomorphism $F_W:\Sigma \rightarrow \Sigma$ which restricts to the identity on the complement of $W$.
	\end{lem}
	\begin{proof}The mapping class group of the circle is trivial: given an orientation-preserving self-diffeomorphism $\phi: S^1 \rightarrow S^1$,  consider a lift $\tilde{\phi}:\mathbb{R} \rightarrow \mathbb{R}$. In particular, $\tilde{\phi}(x+1)=\tilde{\phi}(x)+1$ for all $x \in \mathbb{R}$. Then, $\tilde{\phi}$ is orientation-preserving, i.e. strictly increasing, and the convex homotopy $H:[0,1] \times \mathbb{R} \rightarrow \mathbb{R}$, $(t,x) \mapsto (1-t)\tilde{\phi}(x) + t x$, is an isotopy which induces an isotopy from $\phi$ to the identity. Thus, there exists an isotopy $\psi: [0,1] \times B \rightarrow B$ from $f$ to $\text{Id}_B$. For $U \subseteq W$ a collar neighbourhood of $B$ with diffeomorphism $\phi:U \rightarrow  [0,1] \times B$, define $F_W$ as the extension of the identity of $\Sigma \setminus U$ by the map $\psi\circ \phi:U \rightarrow U$.
	\end{proof}

	\section{The proofs of Theorem \ref{IntroTheoremClassificationSphericalObjects} and Theorem \ref{IntroTheoremTransitivitySphericalOjects}}\label{SectionAllTheProofs}
	
\noindent  We give a proof of Theorem \ref{IntroTheoremClassificationSphericalObjects} and Theorem \ref{IntroTheoremTransitivitySphericalOjects} from the introduction. For convenience of the reader we state them again.
 \begin{thm}\label{TheoremClassificationSphericalObjects}
 	
 	There exists a bijection between isomorphism classes of spherical objects on $C_n$, up to shift, and pairs $([\gamma], \lambda)$, where $\lambda \in \Bbbk^{\times}$ and $[\gamma]$ is the homotopy class of an unoriented, non-separating simple loop $\gamma$ on the $n$-punctured torus $\mathscr{T}^n$.	
 \end{thm} 
\noindent We will see in Section \ref{SectionSurfaceModelCycle}, $\mathscr{T}^n$ is a geometric model for $\mathcal{D}^b(C_n)$. We will derive from $\mathbb{T}_n$ which is the geometric model for $\mathcal{D}^b(\Lambda_n)$. From this point of view, the appearance of the punctured torus in Theorem \ref{TheoremClassificationSphericalObjects} is not surprising.
 	\begin{thm}\label{TheoremTransitivitySphericalOjects}The group of auto-equivalences of $\mathcal{D}^b(C_n)$ acts transitively on the set of isomorphism classes of spherical objects in $\Perf(C_n)$.
 \end{thm}

\begin{proof}[Proofs of Theorem \ref{TheoremClassificationSphericalObjects} and Theorem \ref{TheoremTransitivitySphericalOjects}]
 \noindent We recall that homotopy classes of loops on $\mathbb{T}_n$ are naturally in bijection with homotopy classes of loops on $\mathscr{T}^n$ and so it is sufficient to give a bijection between isomorphism classes of spherical objects in $\Perf(C_n)$ and non-separating simple loops on $\mathbb{T}_n$ instead. According to Corollary \ref{CorollarySphericlaCurvesAreNonSeparating} an indecomposable object $X \in \mathcal{D}^b(C_n)$ is spherical if and only if $\gamma_X \subseteq \mathbb{T}_n$ is a simple non-separating loop and the associated local system $\mathcal{V}_X$ has dimension $1$. The parameter $\lambda \in \Bbbk^{\times}$ in Theorem \ref{TheoremClassificationSphericalObjects} specifies the isomorphism class of $\mathcal{V}_X$ after choosing a base point for $\gamma_X$. To finish the proof of Theorem \ref{TheoremClassificationSphericalObjects} it is left to prove that every simple non-separating loop is gradable and hence represents a spherical object. In order to prove Theorem \ref{TheoremTransitivitySphericalOjects} it is sufficient to show that for every spherical object $X \in \Perf(C_n)$ there exists an auto-equivalence $T$ of $\mathcal{D}^b(C_n)$ such that $T(X) \cong \Bbbk(z)[m]$, where $m \in \mathbb{Z}$ and $z \in C_n$ is a smooth point. Since $\Aut(C_n)$ acts transitively on its set of smooth points, this will imply Theorem \ref{TheoremTransitivitySphericalOjects}.
	
Let $X \in \Perf(C_n)$ be spherical. By Corollary \ref{CorollarySphericlaCurvesAreNonSeparating},  $\gamma=\gamma_X$ is a simple non-separating loop on $\mathbb{T}_n$ and by Proposition \ref{PropositionTransitivityNonSeparatingLoops}, there exists a mapping class $H \in \PMCG(\mathbb{T}_n)$ such that $H(\gamma) \simeq \gamma^0_{\Bbbk(x)}$. Since Dehn twists about boundary loops act trivially on homotopy classes of loops and by Theorem \ref{TheoremHumphriesGenerators}, we may assume that $H$ is a composition of Humphries generators and their inverses. By Proposition \ref{PropositionSphericalTwistsBecomeDehnTwists} and Corollary \ref{CorollaryTwistFunctorsUnderEmbeddings}, there exists $T \in \Aut(\mathcal{D}^b(C_n))$, which is a composition of spherical twists (and their inverses) by skyscraper sheaves of smooth points and $\mathcal{O}_{C_n}$, with the property that $\mathbb{F}(T(X))$ is represented by $H(\gamma) \simeq \gamma^0_{\Bbbk(x)}$. Thus, $\mathbb{F}(T(X))\cong \Bbbk(0, \lambda)[m]$ for some $m \in \mathbb{Z}$ and $\lambda \in \Bbbk^{\times}$. Since $\mathbb{F}$ is an embedding, $T(X)$ is isomorphic to a shift of $\Bbbk(z)$ for a smooth point $z \in \mathbb{P}^1_0 \subseteq C_n$. This finishes the proof of Theorem \ref{TheoremTransitivitySphericalOjects}.

 The proof above and Proposition \ref{PropositionTransitivityNonSeparatingLoops} also show that every non-separating simple loop on $\mathbb{T}_n$ represents a spherical object of $\Perf(C_n)$. Namely, $\gamma$ above is represented by $T^{-1}(\Bbbk(z))$. Together with Corollary \ref{CorollarySphericlaCurvesAreNonSeparating}, this completes the proof of Theorem \ref{TheoremClassificationSphericalObjects}.\end{proof}

\section{A surface model and auto-equivalences of the derived category of a cycle}\label{SectionSurfaceModelCycle}
\noindent We describe the  group of auto-equivalences of $\mathcal{D}^b(C_n)$. The result is derived by regarding $\mathcal{D}^b(C_n)$ as the Verdier quotient of $\mathcal{D}^b(\Lambda_n)$ at objects which are represented by boundary arcs. The underlying idea to contract boundary arcs by passing to a localization was used in \cite{HaidenKatzarkovKontsevich} to classify indecomposable objects in the partially wrapped Fukaya categories of a graded marked surface $\Sigma$ in terms of curves. For our purposes, we use the same approach to provide a topological description of morphisms and their compositions in $\mathcal{D}^b(C_n)$ as in the case of the category $\mathcal{D}^b(\Lambda_n)$. We note that the approach by localization was used by Lekili and Polishchuk \cite{LekiliPolishchukAuslanderOrders} to prove equivalences between derived categories of certain stacky nodal curves (which includes the case $C_n$) and wrapped Fukaya categories of punctured surfaces in the sense of \cite{HaidenKatzarkovKontsevich}. The topological model for $\mathcal{D}^b(C_n)$  reflects this relationship.

\subsection{The category of boundary arcs, the category of loops and perpendicular categories} 

\begin{definition}
Let $\Dinv \subseteq \mathcal{D}^b(\Lambda_n)$ denote the additive closure of indecomposable objects which are represented by loops and denote by $\Dpart$ the additive closure of all indecomposable objects which are represented by boundary arcs. 
\end{definition} 
\noindent The following is a consequence of  Proposition \ref{PropositionClassificationCalabiYauObjects} and characterizes the categories $\Dpart$ and $\Dinv$.
\begin{prp}\label{PropositionCharaterizationDpartDinv}
 $\Dpart$ coincides with the full subcategory of objects $X \in  \mathcal{D}^b(\Lambda_n)$ such that $\tau^2 X \cong X[2]$ and $\Dinv$ is the full subcategory containing all $\tau$-invariant objects. In particular, $\Dinv$ is the essential image of $\mathbb{F}$.
\end{prp}

\noindent We recall the notion of perpendicular categories.

\begin{definition}
Given a full subcategory $\mathcal{U} \subseteq \mathcal{D}^b(\Lambda_n)$, its \textbf{right perpendicular category} $\mathcal{U}^\perp$ is the full subcategory consisting of all objects $X \in \mathcal{D}^b(\Lambda_n)$ such that $\Hom^{\bullet}(Y,X)=0$ for all $Y \in \mathcal{U}$. Similar, the \textbf{left perpendicular category} $^\perp \mathcal{U}$ is defined by interchanging the roles of $X$ and $Y$.
\end{definition}

\noindent An application of the five lemma shows that both perpendicular categories are triangulated. Moreover, $X \oplus Y \in \mathcal{U}^{\perp}$ if and only if $X, Y \in \mathcal{U}^{\perp}$, i.e.\ perpendicular categories are \textit{thick} subcategories.

\begin{lem}\label{LemmaInvPartOrthogonalCategories}
We have
$$
\begin{array}{ccc} {^\perp\Dinv}=\Dpart=\Dinv^{\perp} & \textrm{ and } & {^\perp}\Dpart=\Dinv=\Dpart^{\perp}. \end{array}
$$
\noindent In particular, the categories $\Dinv$ and $\Dpart$ are triangulated.
\end{lem}
\begin{proof}
If $X \in \Dinv$ and $Y \in \Dpart$, then any pair of direct summands of $X$ and $Y$ are represented by disjoint curves on $\mathbb{T}_n$. Thus, there are no non-zero morphisms between $X$ and $Y$ in any direction. In particular, ${^\perp}\Dinv \supseteq \Dpart \subseteq \Dinv^{\perp}$. Any indecomposable object of $\mathcal{D}^b(\Lambda_n)$ is either contained in $\Dinv$ or is represented by an arc and hence has a non-zero morphism to some object in $\Dpart$. This proves equality of $\Dpart$ and the perpendicular categories. The same type of arguments imply the second equality.
\end{proof}

\begin{lem}\label{LemmaEquivalencesRestrict}
Let $T \in \Aut(\mathcal{D}^b(\Lambda_n))$. Then, $T$ restricts to auto-equivalences of $\Dinv$ and $\Dpart$, respectively.
\end{lem}
\begin{proof}
Since $\tau$ lies in the center of $\Aut(\mathcal{D}^b(\Lambda_n))$, Proposition \ref{PropositionCharaterizationDpartDinv} implies that both subcategories are stable under triangle equivalences. \end{proof}

\subsection{A geometric model of the derived category of $C_n$ via Verdier quotients}

\subsubsection{Digression: Verdier quotients by thick subcategories}\ \medskip

\noindent Let  $\mathcal{T}$ be a triangulated category and let $\mathcal{U}$ be a thick subcategory. The \textbf{Verdier quotient} of $\mathcal{T}$ by $\mathcal{U}$ is a triangulated category $\quotient{\mathcal{T}}{\mathcal{U}}$ and an exact functor $\pi=\pi_{\mathcal{U}}: \mathcal{T} \rightarrow \quotient{\mathcal{T}}{\mathcal{U}}$ such that $\mathcal{U}$ is in the kernel of $\pi$ and $\pi$ is universal with this property. The objects of $\quotient{\mathcal{T}}{\mathcal{U}}$ are the same as the objects of $\mathcal{T}$ and $\pi(X) \coloneqq X$ for all $X \in \mathcal{T}$. A morphism from an object $X$ to an object $Y$ is an equivalence class of \textbf{roofs} $(\alpha, f)$
$$
    \begin{tikzcd}& W \arrow{ddr}{f} \arrow[swap]{ddl}{\alpha} \\ \\ X  & & Y, 
    \end{tikzcd}$$

\noindent where $W \in \mathcal{T}$ and $\alpha$ and $f$ are morphisms in $\mathcal{T}$ such that the mapping cone of $\alpha$ is an object of $\mathcal{U}$. \medskip

\noindent Two roofs $X \xleftarrow{\alpha} W \xrightarrow{f} Y$ and $X \xleftarrow{\beta} W' \xrightarrow{g} Y$ are \textbf{equivalent}, if there exists a  roof  $X \xleftarrow{u} U \xrightarrow{v} Y$ and maps $w: U \rightarrow W$ and $w': U \rightarrow W'$ which make the following diagram commutative:
\[\begin{tikzcd}
& W \arrow{dr}{f} \arrow[swap]{dl}{\alpha}
\\ X & U \arrow[swap, dashed]{l}{u} \arrow[dashed]{r}{v} \arrow[near start, dashed]{d}{w'} \arrow[near start, dashed]{u}{w} & Y
\\ & W' \arrow{ul}{\beta} \arrow[swap]{ur}{g}\end{tikzcd}\]

\noindent If $X \xleftarrow{\alpha} W_1 \xrightarrow{f} Y$ and $Y \xleftarrow{\beta} W_2 \xrightarrow{g} Z$ are roofs, we can form the homotopy pullback $W'=W_1 \times_Y W_2$ along $f$ and $\beta$. If $u: W' \longrightarrow W_1$ and $v: W' \longrightarrow W_2$ denote the structure maps of the pullback, the composition $(\beta, g) \circ (\alpha, f)$ is the equivalence class of the roof 

$$
\begin{tikzcd} & & W' \arrow{ddr}{v} \arrow[swap]{ddl}{u} \\ \\ & W_1 \arrow[opacity=60]{ddr}{f} \arrow[swap]{ddl}{\alpha} & & W_2 \arrow{ddr}{g} \arrow[swap]{ddl}{\beta} \\ \\ X  & & Y && Z.
\end{tikzcd}$$

\noindent Finally, a triangle in $\quotient{\mathcal{T}}{\mathcal{U}}$ is distinguished if and only if it is isomorphic to the image  of a distinguished triangle in $\mathcal{T}$ under $\pi$.\medskip

\noindent We will frequently use the following standard facts:\\

\begin{itemize}

    \item   A morphism $f: X \rightarrow Y$ in $\mathcal{T}$ becomes invertible in $\quotient{\mathcal{T}}{\mathcal{U}}$ if and only if the mapping cone of $f$ lies in $\mathcal{U}$. In particular, $\pi(U) \cong 0$ for all $U \in \mathcal{U}$.

    \item Two morphisms $f, g :X \rightarrow Y$ are identified in $\quotient{\mathcal{T}}{\mathcal{U}}$ if and only $f-g$ factors through $\mathcal{U}$. In particular, the canonical map $\Hom(X,Y) \rightarrow \Hom(\pi(X), \pi(X))$ is an isomorphism if $X \in {^{\perp}\mathcal{U}}$ or $Y \in { \mathcal{U}^{\perp}}$.
    
      \item If $\pi(X) \cong \pi(Y)$, then there exists a morphism $\alpha \in \Hom(X,Y) \cup \Hom(Y,X)$ whose mapping cone is an object of $\mathcal{U}$.
    
    \item If $X \in \mathcal{T}$ and $U \in \mathcal{U}$, then the canonical maps $X \longrightarrow X \oplus U$ and $X \oplus U \longrightarrow X$ induce mutually inverse isomorphisms in $\quotient{\mathcal{T}}{\mathcal{U}}$.
\end{itemize}

\noindent Proofs of the first two statements above can be found in Neeman's book \cite{Neeman}: Lemma 2.1.35, Lemma 2.1.26. The last two follow easily from the composition law and the fact that every roof $X \xleftarrow{f} W \xrightarrow{g} Y$ is equivalent to the composition $(f, \operatorname{Id}_W) \circ (\operatorname{Id}_W, g)$.

\medskip

\subsubsection{The Verdier quotient by the boundary category} \ \medskip

\noindent The following proposition serves as the basis for the geometric model. It was first proved  in \cite{LekiliPolishchukMirrorSymmetry}. 

\begin{prp}\label{PropositionVerdierQuotientIsDb}
There exists a triangle  equivalence
$$
    \begin{tikzcd}\mathscr{G}: \quotient{\mathcal{D}^b(\Lambda_n)}{\Dpart} \arrow{rr}{\simeq} && \mathcal{D}^b(C_n). \end{tikzcd}
$$

\end{prp}
\begin{proof}
Since $\Dpart$ is triangulated, the quotient is well-defined. As shown in \cite[Theorem 2, Theorem 6]{BurbanDrozdTilting}, there exists an exact functor $\mathbb{G}:\Coh \mathbb{X}_n \rightarrow \Coh C_n$, where $\mathbb{X}_n$ denotes the non-commutative curve mentioned at the beginning of Section \ref{SectionCategoricalResolutionsCycles}. The kernel of $\mathbb{G}$ is a semi-simple abelian category $T$.  Moreover, it is proved in \cite{BurbanDrozdTilting} that $\mathbb{G}$ induces an equivalence between the Serre quotient $\quotient{\Coh \mathbb{X}_n}{T}$ and $\Coh C_n$. By a result of Miyachi \cite[Theorem 3.2.]{MiyachiLocalization}, the derived functor of $\mathbb{G}$ induces an equivalence $$
    \begin{tikzcd} \quotient{\mathcal{D}^b(\mathbb{X}_n)}{\mathcal{D}^b_{T}(\mathbb{X}_n)} \arrow{rr}{\simeq} && \mathcal{D}^b(C_n), \end{tikzcd}
$$
\noindent where $\mathcal{D}^b_{T}(\mathbb{X}_n)$ denotes the full subcategory of all objects whose cohomologies lie in $T$. By induction and the use of truncation functors of the standard $t$-structure in  $\mathcal{D}^b(\mathbb{X}_n)$, we see that $\mathcal{D}^b_{T}(\mathbb{X}_n)$ coincides with the triangulated hull $\mathcal{T}$ of $T$ inside $\mathcal{D}^b(\mathbb{X}_n) \simeq \mathcal{D}^b(\Lambda_n)$. Proposition 12 in \cite{BurbanDrozdTilting} and its proof generalize to arbitrary cycles in the obvious way and we observe that that $T$ consists precisely of appropriate shifts of all the complexes

\begin{equation}\label{arraylabel}
\begin{array}{ccc} {\begin{tikzcd}[ampersand replacement=\&] P_{t(b_i)} \arrow{r}{b_i} \& P_{s(b_i)}  \arrow{r}{a_i} \& P_{s(a_i)}  \end{tikzcd}} & \textrm{ and } & {\begin{tikzcd}[ampersand replacement=\&] P_{t(d_i)} \arrow{r}{d_i} \& P_{s(d_i)}  \arrow{r}{c_i} \& P_{s(c_i)}.  \end{tikzcd}}  \end{array}
\end{equation}
\noindent where $i \in \mathbb{Z}_n$. For each $i \in \mathbb{Z}_n$, these complexes represent the two segments on the boundary component $B_i$ cut out by the marked points. By \cite[Theorem 4.1]{OpperPlamondonSchroll}, the concatenation of two arcs at a marked point $p$ represents the mapping cone of the morphism associated with $p$. It follows that the complexes in \eqref{arraylabel} generate $\Dpart_i$. In particular, $\mathcal{T}=\Dpart$ which establishes the claimed equivalence.
\end{proof}

\begin{rem}
The restriction of a Verdier quotient functor to the perpendicular category of its kernel is fully faithful. By Lemma \ref{LemmaInvPartOrthogonalCategories} and Proposition \ref{PropositionVerdierQuotientIsDb}, we obtain an embedding $\mathfrak{G}:\Dinv \hookrightarrow \mathcal{D}^b(C_n)$. In fact, it follows from \cite[Theorem 6]{BurbanDrozdTilting} that $\mathfrak{G}\circ \mathbb{F}$ is isomorphic to the identity functor of $\Perf(C_n)$.
\end{rem}

\noindent In the subsequent sections, we denote by $\pi: \mathcal{D}^b(\Lambda_n) \rightarrow \quotient{\mathcal{D}^b(\Lambda_n)}{\Dpart}$ the localization functor.

\subsubsection{Indecomposable objects of $\mathcal{D}^b(C_n)$ as curves on the $n$-punctured torus}\ \medskip

\noindent We give an interpretation of isomorphism classes on indecomposable objects in $\mathcal{D}^b(C_n)$ as curves on the $n$-punctured torus. Our classification follows already from the results in \cite{HaidenKatzarkovKontsevich}. However, since some of the intermediate results will be useful in later parts of this paper, we include a slightly more detailed proof of the classification result which avoids the language of $A_{\infty}$-categories. In what follows, we denote a mapping cone of a morphism $f$ in a triangulated category by $C_f$.\medskip

\noindent For our description of indecomposable objects in $\mathcal{D}^b(C_n)$, we need a mild generalization of \cite[Theorem 4.1.]{OpperPlamondonSchroll} which states that if $p$ is an intersection of  two arcs at a marked point and $f=f_p$, then  $\gamma_{C_f}$  is obtained by concatenation of the original arcs at $p$.
\begin{lem}\label{LemmaMappingConesAtTwoEnds}
Let $Z_1, Z_2 \in \Dpart$, let $X \in \mathcal{D}^b(\Lambda_n)$ be indecomposable and let $f: Z_1 \oplus Z_2 \rightarrow X$ be a morphism whose components $f_i:Z_i \rightarrow X$ correspond to boundary intersections  $p_i \in \gamma_{Z_i} \overrightarrow{\cap} \gamma_X$ at distinct end points of $\gamma_X$. Then, $C_f$ is represented by the concatenation of $\gamma_X$ with $\gamma_{Z_1}$ and $\gamma_{Z_2}$ at $p_1$ and $p_2$.
\end{lem}
\begin{proof}
Following the explicit description in Section $3$ of \cite{OpperPlamondonSchroll}, we see that $f_i$ is of one of the following shapes:

\begin{equation}\label{GraphMapDiagram}
     \begin{tikzcd}[row sep=5pt]
     Z_i \arrow{dd}{f_i}  & \arrow[dash,dashed]{r} & \bullet \arrow[dash]{r} \arrow[dashed, swap]{dd}{\alpha} \arrow[phantom]{r}[yshift=-4ex]{(\ddagger)} & \bullet \arrow[dash]{r}  \arrow[equal]{dd} & \cdots \arrow[dash]{r} & \circ \arrow[equal]{dd} & & \arrow[dashed, dash]{r} & \circ \arrow{dd}{\beta} \\ &&&&&& \text{or} \\
     X  & \arrow[dash,dashed]{r} & \bullet \arrow[dash]{r} & \bullet \arrow[dash]{r} & \cdots \arrow[dash]{r} & \circ & & \arrow[dashed, dash]{r} & \circ & \text{}
     \end{tikzcd}
\end{equation}
\noindent The rows represent the quivers of $Z_i$ and $X$ respectively with arrows being replaced by unoriented edges for the sake of generality, so that each horizontal line represents one of the maps  $t_i \otimes u_i$ in Section \ref{SectionComplexOfACurve} which constitute the differential. The vertex $\circ$ represents the projective module at the end of a string complex. Downward arrows indicate maps induced by non-trivial admissible paths and double lines represent isomorphisms. Dashed lines and dashed arrows may not be present. 

Chain maps as in the left diagram are called \textit{graph maps} and maps as on the right hand side are called \textit{singleton single maps} in \cite{ArnesenLakingPauksztello}. For example, the identity map of a string or band complex is an instance of a graph map.  If it exists, let $\alpha$ denote the unique downward arrow in the diagram associated with $f_i$. 

If $f_i$ is a singleton single map, then its mapping cone (as a chain complex) is a string complex and its quiver is obtained by joining the quiver of $Z_i$ (the upper row of the diagram) with the quiver of $X$ by the arrow $\beta$.

If $f_i$ is a graph map and $\alpha$ exists, then in the left most square in \eqref{GraphMapDiagram} the upper and the lower horizontal lines are oriented the same way in the original quivers of $\gamma_{Z_i}$ and $\gamma_X$. Let $v$ denote one of the diagonal arrows in ($\ddagger$), that is,

\begin{equation}\label{DiagramGraphMap}
 \begin{tikzcd}[row sep=5pt]
    \bullet  \arrow[dashed, swap]{dd}{\alpha} & \bullet \arrow{l} \arrow[equal]{dd} & & \bullet \arrow{r}  \arrow[dashed, swap]{dd}{\alpha} & \bullet   \arrow[equal]{dd} \\ & & \text{or}   \\
     \bullet  & \bullet \arrow{l} \arrow{uul}{v} & &   \bullet \arrow{r} \arrow{uur}{v}& \bullet
     \end{tikzcd}\end{equation}
If $v$ is a map which renders the corresponding diagram in \eqref{DiagramGraphMap} commutative, then $v$ corresponds to an endomorphism $\overline{v}:C_{f_i} \rightarrow C_{f_i}$ of graded $\Lambda_n$-modules (not chain complexes) of degree $0$ such that $\overline{v}^2=0$. Hence $1-\overline{v}$ is an invertible graded map and one can verify that after conjugation with $1-\overline{v}$, $C_{f_i}$ becomes the direct sum of chain complexes whose representation by quivers is
\begin{displaymath}
     \begin{tikzcd}[row sep=5pt]
     \arrow[dash,dashed]{r} & \bullet \arrow[]{dd}{\alpha} & & \bullet \arrow[dash]{r}  \arrow[equal]{dd} & \cdots \arrow[dash]{r} & \circ \arrow[equal]{dd} \\ & & \oplus \\
    \arrow[dash,dashed]{r} & \bullet &  & \bullet \arrow[dash]{r} & \cdots \arrow[dash]{r} & \circ
     \end{tikzcd}
     \end{displaymath}
\noindent The second direct summand is contractible and the first is a string complex which is represented by the concatenation of $\gamma_X$ and $\gamma_{Z_i}$ at $p_i$.

 The proof is analogous, if $\alpha$ is not present in which case the horizontal arrows in the upper and lower row in \eqref{DiagramGraphMap} will have mutually opposite orientations. 

After this preparation, the actual proof of the assertion is quite short: If $p_1$ and $p_2$ correspond to distinct ends of $\gamma_X$ (even though these might correspond to the \textit{same} marked point in $\mathbb{T}_n$, e.g.\ if $\gamma_X$ is closed), then the above transformations of the complex can be performed independently without interference as graph maps at different ends of a string complex can not overlap other than at the arrow $\alpha$. In particular, $C_f$ is represented by the concatenation of $\gamma_X$ with $\gamma_{Z_1}$ and $\gamma_{Z_2}$ at $p_1$ and $p_2$.
\end{proof}

\begin{rem}\label{RemarkConnectingMorphismRepresentable}
In terms of diagrams as in \eqref{GraphMapDiagram}, the connecting morphism $C_f \rightarrow \left(Z_1 \oplus Z_2\right)[1]$ is given by the projection onto the upper row and   the map $X \rightarrow C_f$ is the inclusion of the lower row. If $C_f \not \simeq 0$, both correspond to graph maps or singleton single maps under the transformations in the proof of Lemma \ref{LemmaMappingConesAtTwoEnds}  and  are represented by the ``obvious'' intersections as illustrated in Figure \ref{FigureObviousIntersections}.
\end{rem}
\begin{figure}[H]
	\centering
	\begin{tikzpicture}[hobby, scale=1]
		\begin{scope}
			\draw (-1,2.3) node{$X$};
			\draw (-0.25,0.75) node{$C_{f_1 \oplus f_2}$};

			\draw[decoration={markings, mark=at position 0.4 with {\arrow{<}}},
			postaction={decorate}
			]  (3,1) ellipse (1 and 1);
			\draw[decoration={markings, mark=at position 0.4 with {\arrow{<}}},
			postaction={decorate}
			]  (-2,1) ellipse (1 and 1);
			\draw  plot[ tension=.7] coordinates {(3,0) (2,0) (-1,2) (-2,2)};
			\draw  plot[tension=.7] coordinates {(2,1) (-1,1)};
			
			
			\draw[orange,ultra thick] ({-2+1*cos(0)},{1+1*sin(0)}) arc(0:90:1) node[pos=0.35, left, black]{$Z_1$};
			\draw[orange,ultra thick] ({3+1*cos(180)},{1+1*sin(180)}) arc(180:270:1) node[pos=0.35, right, black]{$Z_2$};
			
				\filldraw (2,1) circle (2pt);
			\filldraw (-1,1) circle (2pt);
			\filldraw (3,0) circle (2pt);
			\filldraw (-2,2) circle (2pt);
			\filldraw (0.5,1) circle (2pt);
			
			
	\draw[<-] ({-2+1.25*cos(-5)},{2+1.25*sin(-5)}) arc(-5:-38:1.125) node[pos=0.75, right]{\footnotesize $f_1$};
		\draw[<-] ({3+1.25*cos(175)},{0+1.25*sin(175)}) arc(175:142:1.125) node[pos=0.75, left]{\footnotesize $f_2$};
			
		\end{scope}	
	\end{tikzpicture}

\caption{The curves of the objects $X, Z_1, Z_2$ and the mapping cone $C_{f_1 \oplus f_2}$. The intersections $\bullet$ represent the maps $f_1, f_2$ as well as $X \rightarrow C_{f_1 \oplus f_2}$ and $C_{f_1 \oplus f_2} \rightarrow \left(Z_1 \oplus Z_2\right)[1]$.} \label{FigureObviousIntersections}
\end{figure}
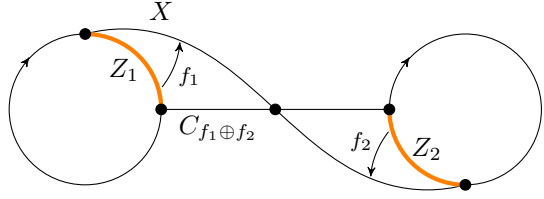

\noindent In order to understand when two indecomposable objects are isomorphic, we need to understand distinguished triangles whose mapping cone is an object of $\Dpart$, or equivalently, mapping cones of maps $Z \longrightarrow X$, where $Z \in \Dpart$ and $X \in \mathcal{D}^b(\Lambda_n)$ is indecomposable. The following lemma serves as a preparation. 

\begin{lem}\label{LemmaMappingConesFactorizingMorphisms}Let $\mathcal{T}$ be a triangulated category and for $i=1,2$, let $\beta_i: Z_i \rightarrow X$ be a morphism in $\mathcal{T}$. If there exists $\rho: Z_1 \rightarrow Z_2$ such that $\beta_1=\beta_2 \circ \rho$, then $C_{\beta_1 \oplus \beta_2} \cong Z_1[1] \oplus C_{\beta_2}$.
\end{lem}
\begin{proof}Consider the commutative diagram
\[\begin{tikzcd}[ampersand replacement=\&]
    Z_1 \oplus Z_2 \arrow{rr}{\left(\begin{smallmatrix}\beta_1 & \beta_2 \end{smallmatrix}\right)} \arrow[swap]{d}{\alpha=\left(\begin{smallmatrix}\operatorname{Id} & 0 \\ \rho & \operatorname{Id}\end{smallmatrix}\right)} \& \& X \arrow{rr} \arrow[equal]{d} \&\& C_{\beta_1 \oplus \beta_2} \arrow{d}{} \arrow{rr}{} \&\& \left(Z_1 \oplus Z_2\right)[1] \arrow{d}{\alpha[1]}\\ 
    Z_1 \oplus Z_2 \arrow{rr}{\left(\begin{smallmatrix}0 & \beta_2\end{smallmatrix}\right)} \&\& X \arrow{rr}{\left(\begin{smallmatrix}0  & c \end{smallmatrix}\right)^T} \& \&  Z_1[1] \oplus C_{\beta_2}\arrow[swap]{rr}{\left(\begin{smallmatrix}\operatorname{Id} & 0 \\ 0 & d \end{smallmatrix}\right)} \& \&\left(Z_1 \oplus Z_2\right)[1].
\end{tikzcd}\]
\noindent The vertical maps in the  diagram are invertible and its rows are distinguished triangles. The lower row is the direct sum of a distinguished triangle involving $\beta_2, c$ and $d$ and a split triangle.
\end{proof}

\begin{lem}\label{LemmaOctahedral1}
	\noindent Let $0 \neq f:Z \longrightarrow X$ be a morphism in $\mathcal{D}^b(\Lambda_n)$, where $Z \in \Dpart$ and $X$ is indecomposable. Then, $C_f \cong Z' \oplus W$, where $Z' \in \Dpart$ and $\gamma_{W}$ is obtained from $\gamma_X$ by clockwise rotations at its end points. Moreover, the connecting morphism $C_f \rightarrow Z[1]$ is a sum of a morphism $Z' \rightarrow Z[1]$ and a morphism $W \rightarrow Z[1]$ which is supported at a single intersection.
\end{lem}
\begin{proof}
Choose a splitting $Z = \bigoplus_{a}{Z^a}$ into indecomposable summands. By splitting of summands using Lemma \ref{LemmaMappingConesFactorizingMorphisms} we may assume that each component $Z^a \rightarrow X$ of $f$ is non-zero. Let  $f^a=\sum{f_i^a}$ denote the corresponding decomposition of $f|_{Z^a}$ such that for all $a$ and all $i \neq j$, $f_i^a$ and $f_j^a$ are non-zero multiples of morphisms associated to distinct intersections $p_i^a, p_j^a \in \gamma_{Z^a} \overrightarrow{\cap} \gamma_X$. For each pair $(a, i)$, there exists an object $Z_i^a \in \Dpart$ and boundary intersections $q_i^a \in \gamma_{Z^a} \overrightarrow{\cap} \gamma_{Z_i^a}$ and $t_i^a \in \gamma_{Z_i^a} \overrightarrow{\cap} \gamma_X$  such that $p_i^a \in C(q_i^a, t_i^a)$, see Definition \ref{DefinitionSetComposition}. In particular, if $\alpha_i^a: Z^a \rightarrow Z_i^a$ and $\beta_i^a: Z_i^a \rightarrow X$ denote the corresponding morphisms, then $f_i^a$ is a non-zero multiple of $\beta_i^a \circ \alpha_i^a$. The above conditions determine the homotopy class of $\gamma_{Z_i^a}$ as well as $q_i^a$ and $t_i^a$ uniquely. In particular, we observe that for all pairs $(a,i)$ and $(a,j)$, $\gamma_{Z_i^a}$ and $\gamma_{Z_j^a}$ differ by full rotations around the boundary and therefore have identical end points.

Our goal is to apply Lemma \ref{LemmaMappingConesFactorizingMorphisms} to compute $C_f$. The set  $\mathcal{Z}=\{Z_i^a \}$ is ordered by the condition that $Z_i^a \prec Z_j^b$ if and only if there exists a morphism $\rho: Z_i^a \rightarrow Z_j^b$ associated with a boundary intersection such that $\beta_i^a$ is a non-zero multiple of  $\beta_j^b \circ \rho$. We note that the relation $\prec$ is transitive and reflexive but in general not anti-symmetric as distinct $Z_i^a$ and $Z_j^b$ may be isomorphic. Moreover, $Z_i^a$ and $Z_j^b$ are not necessarily $\prec$-comparable but they are if $t_i^a$ and $t_j^b$ correspond to the same end point of $\gamma_X$. Hence, $\mathcal{Z}$ naturally splits into two (possibly empty) subsets $\mathcal{Z}=\mathcal{Z}^+ \sqcup \mathcal{Z}^-$, one for each end point of $\gamma_X$, such that any two elements from the same subset are $\prec$--comparable. We denote by $Z^{\pm}$ any (fixed) $\prec$--maximal element in $\mathcal{Z}^{\pm}$ and by $z^{\pm}: {Z}^{\pm} \rightarrow X$ the corresponding map from the set $\{\beta_i^a\}$.

Set $\alpha \coloneqq \bigoplus_{(a,i)}{\alpha_i^a}$ and $\beta^{\pm} \coloneqq \bigoplus_{Z_i^a \in \mathcal{Z}^{\pm}}{\beta_i^a}$. Then, $\alpha, \beta^{+}$ and $\beta^-$ are supported only at boundary intersections. We also write $\beta= \beta^+ + \beta^-$ so that $ \beta \circ \alpha=f$ after rescaling each $\alpha_i^a$ such that $\beta_i^a \circ \alpha_i^a=f_i^a$. Then, $\beta^{\pm}$ factors through $z^{\pm}$ and hence $\beta$ factors through $z^+ \oplus z^-$. Finally, the octahedral axiom shows the existence of a distinguished triangle

\[\begin{tikzcd}
C_{\alpha} \arrow{r} & C_f \arrow{r} & C_{\beta} \arrow{r}{w} & C_{\alpha}[1]\end{tikzcd}.\]

\noindent Note that the map $w$ can be chosen to be a composition of maps $u:C_{\beta} \rightarrow \big(\bigoplus Z_i^a \big)[1]$ and  $v[1]:\big(\bigoplus Z_i^a \big)[1]\rightarrow C_{\alpha}[1]$ which appear in distinguished triangles

\begin{displaymath}
	\begin{tikzcd}
		 \bigoplus Z_i^a \arrow{r}{\beta} & X \arrow{r} & C_{\beta} \arrow{r}{u} & \big(\bigoplus Z_i^a \big)[1], \\
	   Z \arrow{r}{\alpha} & \bigoplus Z_i^a \arrow{r}{v} & C_{\alpha} \arrow{r} & Z[1].	  
	\end{tikzcd}
\end{displaymath}
It follows from Remark \ref{RemarkConnectingMorphismRepresentable} that $w=0$. Finally, Lemma \ref{LemmaMappingConesFactorizingMorphisms} and $C_{\alpha} \in \Dpart$ imply that $C_f \cong C_{\beta} \oplus C_{\alpha} \cong C_{z^+ \oplus z^-} \oplus C_{\alpha}$. Now, the assertion follows from Lemma \ref{LemmaMappingConesAtTwoEnds} and Remark \ref{RemarkConnectingMorphismRepresentable}.\end{proof}

\begin{rem}\label{RemarkDualVersions}
	Dual versions of Lemma \ref{LemmaMappingConesFactorizingMorphisms} for morphisms $X \rightarrow Z_i$ and Lemma \ref{LemmaOctahedral1} for morphisms $X \rightarrow Z$ are proved in the same way.
\end{rem}
The inclusion $\mathbb{T}_n \subseteq \mathscr{T}^n$ induces a surjective map  from the of homotopy classes of curves on $\mathbb{T}_n$ to the set of homotopy classes of curves on $\mathscr{T}^n$.  More precisely, we view $\mathscr{T}^n$ as being glued from $\mathbb{T}_n$ and a $1$-punctured disc (the puncture being the center) for each boundary component of $\mathbb{T}_n$. Given a curve $\gamma$ on $\mathbb{T}_n$, we apply the radial contraction of each disc to its puncture in the center. This yields a curve $\overline{\gamma}$ on $\mathscr{T}^n$ with the same domain as $\gamma$. If $\gamma$ if an arc with end points on boundary components $B_1$ and $B_2$, then $\overline{\gamma}$ is an arc whose end points are the punctures associated with $B_1$ and $B_2$. If $\gamma$ is a loop, then so is $\overline{\gamma}$. We further observe that if $\gamma_1, \gamma_2 \subseteq \mathbb{T}_n$ are arcs, then $\overline{\gamma_1} \simeq \overline{\gamma_2}$ if and only if $\gamma_1$ and $\gamma_2$ are obtained from each other through an application of fractional twists.

\begin{prp}\label{PropositionBijectionArcsDbCoh}
There exists a bijection between isomorphism classes of indecomposable objects of $\mathcal{D}^b(C_n)$ up to shift and homotopy classes of curves on $\mathscr{T}^n$ equipped with indecomposable local systems. More precisely, if $X \in \mathcal{D}^b(\Lambda_n)$ is indecomposable and $(\gamma, \mathcal{V})$ is its associated curve $\gamma$ with local system $\mathcal{V}$, then $(\overline{\gamma}, \mathcal{V})$ is the pair assigned to $\mathscr{G}\circ \pi(X) \in \mathcal{D}^b(C_n)$.
\end{prp}
\begin{proof}
Let $X \in \mathcal{D}^b(\Lambda_n)$ be such that no direct summand of $X$ lies in $\Dpart$. If $\pi(X)$ is indecomposable and $X \cong X_1 \oplus X_2$, then by additivity of $\pi$, $\pi(X) \cong \pi(X_1) \oplus \pi(X_2)$. Consequently, if $\pi(X_1) \not \cong 0$, then $X_2 \in \Dpart$ which shows that $X$ is indecomposable whenever $\pi(X)$ is. Next, suppose that $X$ is indecomposable and $\pi(X) \cong \pi(Y)$, where $Y \in \mathcal{D}^b(\Lambda_n)$. Then, $Y$ is isomorphic to the mapping cone of a map $g: Z \rightarrow X$, where $Z \in \Dpart$. By Lemma \ref{LemmaOctahedral1}, $C_g \cong Z' \oplus W$, where $Z' \in \Dpart$ and $W$ is indecomposable such that $\gamma_{W}$ is obtained from $\gamma_X$ by rotation at end points. By restricting to the case $Y=Y_1 \oplus Y_2$ and using the Krull-Remak-Schmidt property of $\mathcal{D}^b(\Lambda_n)$, we see that $\pi(X) \cong \pi(W)$ is indecomposable. On the other hand, if $Y \in \mathcal{D}^b(\Lambda_n)$ is indecomposable such that $\gamma_X$ and $\gamma_{Y}$ differ by rotations at end points, then $\pi(X) \cong \pi(Y)$ up to shift by Lemma \ref{LemmaMappingConesAtTwoEnds}.
\end{proof}

\begin{notation} By abuse of notation, we will occasionally use $\gamma_X$ for both the curve $\gamma=\gamma_X \subseteq \mathbb{T}_n$ of an indecomposable object $X \in \mathcal{D}^b(\Lambda_n)$ and the curve $\gamma_{\mathscr{G} \circ \pi(X)} \subseteq \mathscr{T}^n$ which represents $\mathscr{G} \circ \pi(X) \in \mathcal{D}^b(C_n)$. As explained above, the latter is given by $\overline{\gamma}$.
	\end{notation}

\subsubsection{Morphisms in $\mathcal{D}^b(C_n)$ as intersections of curves}\label{SectionMorphismsIntersections} \ \medskip

\noindent We extend the graphical description of morphisms and compositions to the category $\mathcal{D}^b(C_n)$. The computations of morphisms in the Verdier quotient of $\mathcal{D}^b(\Lambda_n)$ ``at the boundary'' which we obtain in this subsection \ref{TheoremGeometricDescriptionMorphismPuncturedCase} are by no means specific to the algebra $\Lambda_n$ and can be generalized to arbitrary gentle algebras and similar Verdier quotients, see Remark \ref{RemarkGeneralizations}.\medskip

\noindent To begin with, we recall the definition of exceptional cycles in a triangulated category due to Broomhead, Pauksztello and Ploog.

\begin{definition}[\cite{BroomheadPauksztelloPloog}] \label{DefinitionExceptionalCycles}
Let $r \geq 2$. A collection $\{U_i\}_{i \in \mathbb{Z}_r}$ of objects in a triangulated category $\mathcal{T}$ with Serre functor $\mathcal{S}$ is an \textbf{exceptional $r$-cycle} if 

\begin{enumerate}
    \item every object $U_i$ is exceptional, i.e.\ $\dim \Hom^{\bullet}(U_i,U_i)=1$; 
    \item for each $i \in \mathbb{Z}_r$, the exists $m_i \in \mathbb{Z}$ such that $\mathcal{S} U_i \cong U_{i+1}[m_i]$;
    \item $\Hom^{\bullet}(U_i,U_j)=0$, unless $j=i$ or $j=i+1$.
\end{enumerate}
\end{definition}
\noindent As in the case of spherical objects, every exceptional cycle $\mathcal{U}=\{U_i\}$ induces an auto-equivalence $T_{\mathcal{U}}$ (\cite{BroomheadPauksztelloPloog}) described by the evaluation triangle \eqref{EvaluationTriangle}, where $X = \oplus_{i \in \mathbb{Z}_r}{U_i}$.\medskip

\noindent For every component $B \subseteq \partial \mathbb{T}_n$ the segments on $B$ between the two marked points correspond to an exceptional $2$-cycle $U^1_B, U^2_B$  which follows from \cite[Section 2.5]{OpperDerivedEquivalences}. We set $U_B \coloneqq U^1_B \oplus U^2_B$ and denote by $T_{B}$ the associated twist functor. From Corollary \ref{CorollaryTwistFunctorsUnderEmbeddings} (which is also true for $T_B)$, we conclude that $T_B$ and $T_{B'}$ commute for any pair of components $B, B'$.

\begin{cor}\label{CorollaryThetaWellDefined}
The functor $\vartheta\coloneqq \prod_{B \subseteq \partial \mathbb{T}_n} T_B$, the composition of the twist functors $T_B$, is well-defined up to isomorphism of functors.
\end{cor}

\noindent It follows from \cite[Corollary 5.2.]{OpperPlamondonSchroll} and Lemma \ref{LemmaOctahedral1} that $\vartheta$ coincides with $\tau^{-1}$ on the level of objects so that $\Psi(\vartheta)=\Psi(\tau^{-1}) \in \MCG(\mathbb{T}_n)$. In case of discrete derived categories such a statement was already proved in \cite{BroomheadPauksztelloPloog}. In fact, even the following is true.
\begin{prp}\label{PropositionIsomorphismTauTheta}
There exists an isomorphism of functors $\vartheta \simeq \tau^{-1}$.
\end{prp}
\begin{proof}
The mapping class $\Psi(\vartheta \tau)$ is trivial. By Proposition \ref{PropositionExactSequenceAutoequivalences}, it is therefore a composition of a shift and an equivalence which is induced by an automorphism of $\Lambda_n$ which multiplies each arrow $\alpha \in \{d_0\}\cup\{c_i\}_{i \in \mathbb{Z}_n}$ in the quiver of Figure \ref{FigureQuiverLambdaN} with a scalar $\lambda_{\alpha} \in \Bbbk^{\times}$. We have $\vartheta\tau(Y) \cong Y$ for all $Y \in \Dinv={^{\perp}\Dpart}$. In particular, this applies to  $\mathcal{O}(\lambda)$ and all objects $\Bbbk(i, \mu)$ which shows that the shift is trivial and $\lambda_{\alpha}=1$ for all $\alpha \in \{d_0\}\cup\{c_i\}_{i \in \mathbb{Z}_n}$.
\end{proof}

\noindent For each boundary component $B \in \partial \mathbb{T}_n$, the definition of $T_{B}$ gives rise to a natural transformation $\eta^B: \operatorname{Id}_{\mathcal{D}^b(\Lambda_n)} \rightarrow T_B$ given on $X$ by the map $X \rightarrow T_B(X)$ in the evaluation triangle. Again, the composition of all $\eta^B$ is independent from the order of composition and defines a natural transformation $\eta: \operatorname{Id}_{\mathcal{D}^b(\Lambda_n)} \rightarrow \vartheta$. That is, if $B_1, \dots, B_n$ are the components of $\partial \mathbb{T}_n$ and $\eta(i)\coloneqq \eta^{B_i}$, then $\eta_X$ is the composition $X \xrightarrow{\eta(1)_X} T_{B_1}(X) \xrightarrow{\eta(2)_{T_{B_1}(X)}} T_{B_2} \circ T_{B_1}(X) \xrightarrow{} \cdots \rightarrow \vartheta(X)$.\medskip

\begin{rem}
The natural transformation $\eta$ gives a different explanation for the natural transformation $\operatorname{Id}_{\mathcal{D}^b(\Coh \mathbb{X}_n)} \rightarrow \tau^{-1}$ from \cite[Corollary 1]{BurbanDrozdTilting} in the derived category of the non-commutative curve $\mathbb{X}_n$.
\end{rem}

\noindent For all $X, Y \in \mathcal{D}^b(\Lambda_n)$, $\eta$ induces an infinite tower of  graded $\Bbbk$-linear maps 
\[\begin{tikzcd}
\Hom^{\bullet}(X,Y) \arrow{r} & \Hom^{\bullet}(X, \vartheta(Y)) \arrow{r} & \Hom^{\bullet}(X, \vartheta^2(Y)) \arrow{r} & \cdots 
\end{tikzcd}\]
which allows us to define its colimit. Moreover, for each number $m \geq 0$, $\eta$ give rise to a map 
\begin{equation}\label{structuremaps}\begin{tikzcd}[row sep=20pt] \Hom^{\bullet}(X, \vartheta^m(Y)) \arrow{rr} &&  \Hom^{\bullet}(\pi(X),\pi(Y)) \\ f \arrow[mapsto]{rr} && (\alpha, \vartheta^{-m}(f)), \end{tikzcd} \end{equation}
\noindent where $\alpha: \vartheta^{-m}(X) \rightarrow X$ is the map induced by the $m$-fold power $\eta^m$ of $\eta$. More precisely, $\eta^m: \operatorname{Id} \rightarrow \vartheta^m$ is given on $X$ by the composition $X \xrightarrow{\eta_X} \vartheta(X) \xrightarrow{\eta_{\vartheta(X)}} \vartheta^2(X) \rightarrow \cdots \rightarrow \vartheta^m(X)$.

\begin{lem}\label{LemmaWellDefinedColimit}
Let $X, Y \in \mathcal{D}^b(\Lambda_n)$. Then,
 the diagram
\[\begin{tikzcd}
\Hom^{\bullet}(X,Y) \arrow{rr}{\left(\eta^m\right)_Y \circ {-}} \arrow{d}[swap]{\operatorname{can}} & & \Hom^{\bullet}(X, \vartheta^m(Y)) \arrow{dll} \\ \Hom^{\bullet}(\pi(X), \pi(Y))
\end{tikzcd}\]
\noindent commutes. 
\end{lem}
\begin{proof}
Let $f:X \rightarrow Y$. Denote by $\alpha_Y: Y \rightarrow \vartheta^m(Y)$ and $\alpha_X: \vartheta^{-m}(X) \rightarrow X$ the maps induced by $\eta^m$ and set $g \coloneqq \alpha_Y \circ f$. Consider the diagram
\[\begin{tikzcd}
& \vartheta^{-m}(X) \arrow{dl} \arrow[equal]{d} \arrow{drr}{\vartheta^{-m}(g)}
\\ X & \vartheta^{-m}(X) \arrow{l} \arrow[near start]{rr}{f \circ \alpha_X} \arrow{d} & & Y
\\ & X \arrow[equal]{ul} \arrow[swap]{urr}{f}
\end{tikzcd}\]
\noindent where each unlabeled arrow represents the map $\alpha_X$. As $\eta^m$ is a natural transformation it follows $f \circ \alpha_X = \vartheta^{-m}(g)$. Thus, the above diagram commutes and the roofs $\left(\alpha_X, \theta^{-m}\left((\eta^m)_Y \circ f\right)\right)$ and $(\operatorname{Id}_X, f)=\pi(f)$ are equivalent. 
\end{proof}

\begin{lem}\label{LemmaInjectivity}
Let $X, Y \in \mathcal{D}^b(\Lambda_n)$ such that no direct summand of $Y$ lies in $\Dpart$. Then, there exists $l \geq 0$ such that for all $m \geq l$, the map
$\Hom^{\bullet}(X,\vartheta^m(Y)) \rightarrow \Hom^{\bullet}(\pi(X),\pi(Y))$ is injective. In particular, for all $m \geq l$, the map $\Hom^{\bullet}(X,\vartheta^m(Y)) \rightarrow \Hom^{\bullet}(X, \vartheta^{m+1}(Y))$ is injective.
\end{lem}
\begin{proof}
Since no direct summand of $Y$ lies in $\Dpart$, the graphical description of composition shows that a morphism $f:X \rightarrow Y$ factors through $\Dpart$ if only if all $p \in \supp f$ are interior intersections and $p$ bounds an immersed intersection triangle whose other two corners are end points of $\gamma_X$ and $\gamma_Y$, and whose third side is a boundary arc. As $\vartheta$ is given by clockwise rotation of the boundary, we see  that $\gamma_X$ and $\gamma_{\vartheta^m(Y)}$ cannot bound any such triangle for $m$ large enough. For all such $m$, the first map is injective.  Injectivity of the second  map follows from Lemma \ref{LemmaWellDefinedColimit}.
\end{proof}

\noindent The following proposition describes the morphisms in $\mathcal{D}^b(C_n)$.
\begin{prp}\label{PropositionColimit}
Let $X, Y \in \mathcal{D}^b(\Lambda_n)$ such that $Y$ has no direct summands in $\Dpart$. Then, the maps $\Hom^{\bullet}(X, \vartheta^i(Y)) \rightarrow \Hom^{\bullet}(\pi(X), \pi(Y))$  induce an isomorphism

\[\begin{tikzcd} \colim_{i \in \mathbb{N}} \,\Hom^{\bullet}(X,\vartheta^i(Y)) \arrow{rr}{\simeq} & &
\Hom^{\bullet}(\pi(X), \pi(Y)).\end{tikzcd}\]
\end{prp}
\begin{proof}
By Lemma \ref{LemmaWellDefinedColimit} and Lemma \ref{LemmaInjectivity}, the map exists and is injective. We may assume that $X$ is indecomposable. Let $g: \pi(X) \rightarrow \pi(Y)$ be a morphism and let $(\alpha, f)$ be a representing roof of $g$, where $\alpha: W \rightarrow  X$ and $f: W \rightarrow Y$. We may further assume that $W$ is indecomposable because $\pi(X)$ is indecomposable. Indeed, it follows from  Lemma \ref{LemmaMappingConesFactorizingMorphisms} and the proof of Lemma \ref{LemmaOctahedral1} that $W \cong X' \oplus Z'$, where $X'$ is indecomposable and $Z' \in \Dpart$. We then can replace $(\alpha, f)$ by the roof $(\alpha \circ \iota, f \circ \iota)$  where $\iota: X' \rightarrow W$ denotes the natural inclusion. 

By definition, the morphism $\alpha: W \rightarrow X$ sits in a distinguished triangle
\begin{displaymath}
\begin{tikzcd}
Z \arrow{r}{u} & W \arrow{r}{\alpha} & X \arrow{r} & Z[1],
\end{tikzcd}
\end{displaymath}
\noindent where $Z \in \Dpart$. Since $X$ is indecomposable, Lemma \ref{LemmaMappingConesFactorizingMorphisms} and the proof of Lemma \ref{LemmaOctahedral1} show that $Z$ must be indecomposable and $\gamma_X$ is obtained from $\gamma_W$ by a number of positive rotations at its end points. Let us denote by $r$ and $s$ the necessary number of such rotations at each end point. We denote their maximum by $m$ and claim that there exists a map $w:\vartheta^{-m}(X) \rightarrow W$ such that $\alpha \circ w: \vartheta^{-m}(X) \rightarrow X$ is induced by $\eta^m$. In this case the roof $(\alpha \circ w, f \circ w)$ is equivalent to $(\alpha, f)$ as can be seen from the commutative diagram

\[\begin{tikzcd}
&& W \arrow{dll}[swap]{\alpha} \arrow{drr}{f}
\\ X && \vartheta^{-m}(X) \arrow{ll} \arrow{u}{w} \arrow{rr}[pos=0.35]{f \circ w} \arrow[equal]{d} && Y
\\ && \vartheta^{-m}(X) \arrow{ull} \arrow[swap, pos=0.35]{urr}{f\circ w}
\end{tikzcd}\]
\noindent where all unlabeled arrows are induced by $\eta^m$. Then, $\vartheta^m(f \circ w) \in \Hom(X, \vartheta^m(Y))$ is a preimage of $g$.
Once again, the existence of $w$ can be read off from the surface: $\alpha$ corresponds to a single intersection (Remark \ref{RemarkConnectingMorphismRepresentable}) which forms an intersection triangle with the intersection corresponding to $\eta_{\vartheta^{-m}(X)}$ and a boundary intersection $p \in \gamma_{\vartheta^{-m}(X)} \overrightarrow{\cap} \gamma_W$. We define $w$ as a suitable multiple of $f_p$. \end{proof}

\noindent  Proposition \ref{PropositionColimit} can be rephrased in the following way.

\begin{thm}\label{TheoremGeometricDescriptionMorphismPuncturedCase}
Let $X,Y \in \mathcal{D}^b(C_n) \setminus \Perf(C_n)$ be indecomposable. Then, there exists a bijection between a homogeneous basis of $\Hom^{\bullet}(X,Y)$ and the set consisting of

\begin{enumerate}
    \item all interior intersections of $\gamma_X$ and $\gamma_Y$, and,
    \item paths in the loop quiver 
        
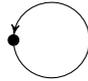
\[\begin{tikzpicture} \filldraw (-0.5,0) circle (2pt);
        \draw[<-] ({0.5*cos(171))},{0.5*sin(171)}) arc (171:{-180+9}:0.5);
        \end{tikzpicture}\]
        \noindent for each  $p \in \gamma_X \overrightarrow{\cap} \gamma_Y$ which is a puncture.
\end{enumerate}

\end{thm}
\begin{proof}
	Let $X', Y' \in \mathcal{D}^b(\Lambda_n)$ such that  $\pi(X') \cong X$ and $\pi(Y') \cong Y$. We may assume that the cardinality of $X' \overrightarrow{\cap} Y'$ coincides with the number of intersections of $X$ and $Y$. In particular, the boundary intersections of $\gamma_{X'}$ and $\gamma_{Y'}$ are in bijection with intersections of $\gamma_X$ and $\gamma_Y$ at punctures. Moreover, every interior intersection $p \in \gamma_X \cap \gamma_Y$ corresponds to a unique interior intersection $p' \in \gamma_{X'} \cap \gamma_{Y'}$ and $\pi(f_{p'})$ is by definition an element in the claimed basis of $\Hom^{\bullet}(X,Y)$. Let $p \in \gamma_X \overrightarrow{\cap} \gamma_Y$ be a puncture and let denote $q$ the corresponding boundary intersection of $\gamma_{X'}$ and $\gamma_{Y'}$. Then, for every $m \in 2 \mathbb{Z}$, $q \in \gamma_{\vartheta^{-m}(X')} \overrightarrow{\cap} \gamma_{Y'}$ and the paths in the loop quiver of $p$ correspond to roofs 
	$X' \xleftarrow{\eta^m} \vartheta^{-m}(X') \xrightarrow{f_q} Y'[\deg q]$, i.e.\ the image of $\vartheta^m(f_q)$ under the map $\Hom^{\bullet}(X', \vartheta^m(Y')) \rightarrow \Hom^{\bullet}(X,Y)$.
\end{proof}

\begin{rem}\label{RemarkMorphismsPuncturesAsIntersections}
The morphisms in Theorem \ref{TheoremGeometricDescriptionMorphismPuncturedCase} (2) can be visualized as intersections as follows. Let $\widetilde{\mathscr{T}}$ denote the universal cover of $\mathscr{T}^n$, where all punctures are thought as being removed from the surface. Then, locally around a puncture $p$, $\mathscr{T}^n$ is homeomorphic to $\mathbb{C}^{\times}$. Using polar coordinates and the identifications $(-\pi, \pi) \cong \mathbb{R} \cong (0,1)$, $\widetilde{\mathscr{T}}$ can be identified locally with the open unit disc $\mathbb{D}_{\operatorname{sl}}$ with a slit $[-1,0] \times \{0\}$ removed. Arcs with end point at $p$ are lifted to non-compact arcs with limit $0 \in \mathbb{D}_{\operatorname{sl}}$. The various morphisms arising from an intersection of $\gamma_X$ and $\gamma_Y$ at $p$ then correspond to the intersections at $0$ between a fixed lift of $\gamma_X$ and the different lifts of $\gamma_Y$.
\end{rem}

\begin{rem}\label{RemarkCompositions}
The isomorphism in Proposition \ref{PropositionColimit} is compatible with compositions in the following sense. For $X,Y, Z \in \mathcal{D}^b(\Lambda_n)$, consider the tower consisting of the objects 
\[\begin{tikzcd}
\Hom^{\bullet}\left(X, \vartheta^i(Y)\right) \times \Hom\left(\vartheta^i(Y), \vartheta^i(Z)\right), 
\end{tikzcd}\]
\noindent and the product of the maps $\Hom^{\bullet}(X, \vartheta^i(Y)) \rightarrow \Hom^{\bullet}(X, \vartheta^{i+1}(Y))$ with the canonical maps $\Hom(\vartheta^i(Y), \vartheta^i(Z)) \rightarrow \Hom(\vartheta^{i+1}(Y), \vartheta^{i+1}(Z))$  induced by $\vartheta$. Each object in the tower has a composition map to $\Hom^{\bullet}(X,\vartheta^m(Z))$. As $\vartheta$ is an equivalence, the colimit of this tower is naturally isomorphic to $\colim_{i \in \mathbb{N}} \Hom^{\bullet}(X, \vartheta^i(Y)) \times \Hom^{\bullet}(Y,Z)$ and has an induced composition map to $\Hom^{\bullet}(X,Z)$. Under the isomorphism in Proposition \ref{PropositionColimit}, this coincides with the usual composition map. This observation allows us to interpret compositions in $\mathcal{D}^b(C_n)$ in a similar way as in the category $\mathcal{D}^b(\Lambda_n)$ via tridents, triangles and double-bigons in the universal cover $\widetilde{\mathscr{T}}$ of $\mathscr{T}^n$ whose corners are allowed to converge to a puncture, see Remark \ref{RemarkMorphismsPuncturesAsIntersections}.
\end{rem}
\begin{exa}\label{ExampleEndomorphismRingPuncture} Suppose $\delta_0, \delta_1, \delta_2 \subseteq \mathscr{T}^n$ are arcs which meet at a puncture as in Figure \ref{FigureEndomorphismRingArcs}. The intersection generates a subalgebra of $\End^{\bullet}(\bigoplus_{i=1}^3X_{\delta_i})$ which is isomorphic to the path algebra of the quiver shown on the right hand side of Figure \ref{FigureEndomorphismRingArcs}. The degree of a full cycle equals $2$, i.e.\ the inverse winding number around the puncture. The paths in the loop quiver in Theorem \ref{TheoremGeometricDescriptionMorphismPuncturedCase} of the pair $X_{\delta_0}, X_{\delta_2}$ correspond to the paths $(\beta \alpha) (\gamma \beta \alpha)^i$, $i \geq 0$.
\end{exa}
\begin{figure}[H]
	\[\begin{tikzpicture}[scale=0.7]
	\begin{scope}[shift={(0,0)}]
	\foreach \i in {0,1,2}
	{
		\draw ({3pt*cos((\i * 120))},{3pt*sin((\i*120))})--({2*cos((\i * 120))},{2*sin((\i*120))}) node[pos=1.2]{$\delta_{\i}$};
	}
	\draw (0,0) circle (3pt);
	\end{scope}
	
	\begin{scope}[shift={(8,0)}]
	\foreach \i in {0,1,2}
	{
		\draw ({3pt*cos((\i * 120))},{3pt*sin((\i*120))})--({2*cos((\i * 120))},{2*sin((\i*120))}) node[pos=1.2]{$\delta_{\i}$};
		\filldraw[black] ({1.5*cos((\i * 120))},{1.5*sin((\i*120))}) circle (2pt);
		\draw[<-] ({1.5*cos(((\i) * 120-7))},{1.5*sin(((\i)*120-7))})  arc ({\i*120-7}:{(\i-1)*120+7}:1.5);
	}
	\draw (0,0) circle (3pt);
	
	\draw ({1.75*cos((-1 * 120+60))},{1.75*sin((-1*120+60))}) node{$\gamma$};
	\draw ({1.75*cos((-2 * 120+60))},{1.75*sin((-2*120+60))}) node{$\beta$};
	\draw ({1.75*cos((-3 * 120+60))},{1.75*sin((-3*120+60))}) node{$\alpha$};
	\end{scope}
	\end{tikzpicture}
	\]
	\caption{}
	\label{FigureEndomorphismRingArcs}
\end{figure}

\begin{rem}\label{RemarkMappingConesVerdierQuotient}  Distinguished triangles in $\mathcal{D}^b(\Lambda_n) / \Dpart$ are isomorphic to images of distinguished triangles in $\mathcal{D}^b(\Lambda_n)$. This allows us to determine the curve of the mapping cone of any morphism which correspond to an intersection. The only new and interesting case appears for an intersection at a puncture. Consider arcs $\delta_1, \delta_2$ as in Example \ref{ExampleEndomorphismRingPuncture}. For $i \geq 0$, let $f^i$ denote the morphism which corresponds to the path $(\alpha\gamma)^i\beta$ in Figure \ref{FigureEndomorphismRingArcs}. Then, $C_{f^i}$ extends $\delta_1$ and $\delta_2$ outside a small neighbourhood of $p$ and winds around the puncture counter-clockwise $i$ times. Figure \ref{FigureMapppingConesPunctures} shows the curves of $C_{f^i}$ for $i=0,1,2$. The general case is analogous.
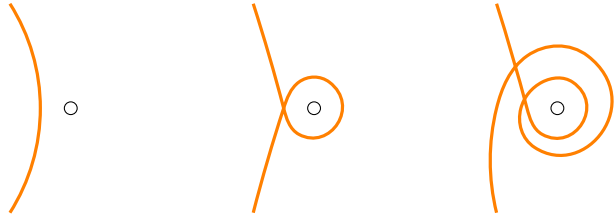
\begin{figure}[H]
	\[\begin{tikzpicture}[scale=0.8]
	\begin{scope}[shift={(0,0)}]
	\draw (0,0) circle (3pt);

	\draw[hobby, very thick, orange] plot  [] coordinates { ({2*cos((1 * 120))},{2*sin((1*120))}) ({0.5*cos((1.5 * 120))},{0.5*sin((1.5*120))}) ({2*cos((2 * 120))},{2*sin((2*120))}) };
	
	\end{scope}
	
	\begin{scope}[shift={(4,0)}]
	\draw (0,0) circle (3pt);

\draw[hobby, very thick, orange] plot  [] coordinates { ({2*cos((1 * 120))},{2*sin((1*120))}) ({0.5*cos((1.5 * 120))},{0.5*sin((1.5*120))}) ({0.5*cos((2 * 120))},{0.5*sin((2*120))}) ({0.5*cos((0.5 * 120))},{0.5*sin((0.5*120))}) ({0.5*cos((1* 120))},{0.5*sin((1*120))}) ({0.5*cos((1.5* 120))},{0.5*sin((1.5*120))}) ({2*cos((2 * 120))},{2*sin((2*120))}) };

	\end{scope}
	
	\begin{scope}[shift={(8,0)}]
	\draw (0,0) circle (3pt);

\draw[hobby, very thick, orange] plot  [] coordinates { ({2*cos((1 * 120))},{2*sin((1*120))}) ({0.5*cos((1.5 * 120))},{0.5*sin((1.5*120))}) ({0.5*cos((2 * 120))},{0.5*sin((2*120))}) ({0.5*cos((0.5 * 120))},{0.5*sin((0.5*120))}) ({0.5*cos((1* 120))},{0.5*sin((1*120))}) ({0.6*cos((1.5* 120))},{0.6*sin((1.5*120))}) ({0.75*cos((2* 120))},{0.75*sin((2*120))}) ({1*cos((0.5* 120))},{1*sin((0.5*120))}) ({1*cos((1* 120))},{1*sin((1*120))}) ({1*cos((1.5* 120))},{1*sin((1.5*120))}) ({2*cos((2 * 120))},{2*sin((2*120))}) };

	\end{scope}
	\end{tikzpicture}
	\]
	\caption{The mapping cones of $f^0, f^1$ and $f^2$ (from left to right).}
	\label{FigureMapppingConesPunctures}
\end{figure}

\end{rem}

\noindent As an easy application of Theorem \ref{TheoremGeometricDescriptionMorphismPuncturedCase} we can now describe the curves which represent the structure sheaves of singular points.
\begin{cor}\label{CorollaryCurvesOfSingularPoints} Let $z \in C_n$ be singular. Then  $\Bbbk(z)$ is represented by a vertical simple arc $\delta \subseteq \mathscr{T}^n$ which lies between two loops $\gamma_{\Bbbk(x)}^i$ and $\gamma_{\Bbbk(x)}^{i+1}$ as shown in Figure \ref{FigureSingularSkyscraper}.
\end{cor}
\begin{proof}
Denote by $\gamma$ a representative of $\Bbbk(z)$ in minimal position. Then, $\gamma$ is an arc as $\Bbbk(z)$ is not perfect. Since $\dim \Hom^{\bullet}(\Bbbk(z), \mathcal{O}_{C_n})=1= \dim \Hom^{\bullet}(\mathcal{O}_{C_n}, \Bbbk(z))$ and $\Hom^{\bullet}(\Bbbk(z), \Bbbk(x))=0=\Hom^{\bullet}(\Bbbk(x), \Bbbk(z))$ for all $x \neq z$, $\gamma$ must intersect $\gamma_{\Pic}$  exactly once and must be disjoint from any loop $\gamma_{\Bbbk(x)}^i$. The only homotopy classes of arcs which satisfy these constraints are of the claimed shape.
\end{proof}
	\begin{figure}[H]
			\begin{displaymath}
			\begin{tikzpicture}[scale=2.5]
			\draw[dashed] (0,0)--(4,0);
			\draw[dashed] (0,0)--(0,1);
			\draw[dashed] (0,1)--(4,1);
			\draw[dashed] (4,1)--(4,0);
			
			
			\filldraw[white] (2+0.5,0.5) circle (0.275);
		
			\draw[dotted, thick] (2+0.5-0.125,0.5)--(2+0.5+0.125,0.5);

			\draw[line width=1.5, black] (1,0)--(1, 0.75)  node[left]{$\gamma_{\Bbbk(x)}^i$};
			\draw[line width=1.5, black] (1,0.75)--(1,1);
			
			\draw[line width=1.5, black] (2,0)--(2, 0.75)  node[right]{$\gamma_{\Bbbk(x)}^{i+1}$};
			\draw[line width=1.5, black] (2,0.75)--(2,1);
			
			\draw[line width=1.5, black] (1.5,0)--(1.5, 0.75)  ;
			\draw[line width=1.5, black] (1.5,1)--(1.5,0.75) node[right]{$\delta$};
			
				\foreach \u in {0,1,3}
			{
				\filldraw[white] (\u+0.5,0.5) circle (1pt);
				\draw[black] (\u+0.5,0.5) circle (1pt);
			
			}
			\end{tikzpicture}
			\end{displaymath}
			\caption{The vertical arc $\delta$ represents the structure sheaf of a singular point. White circles indicate punctures.} \label{FigureSingularSkyscraper}
		\end{figure}
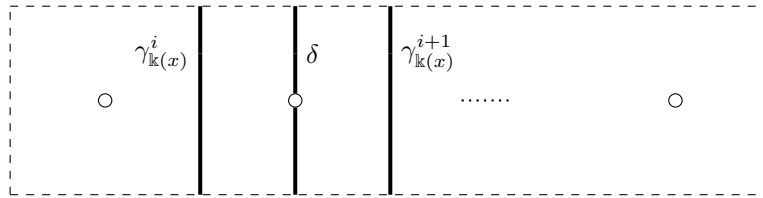

\begin{rem}\label{RemarkGeneralizations}
Theorem \ref{TheoremGeometricDescriptionMorphismPuncturedCase} and the arguments outlined in Remark \ref{RemarkCompositions} generalize to the following situation. Suppose $A$ is a gentle algebra and let $\Sigma=\Sigma_A$ denote its marked surface ($\Sigma=\mathbb{T}_n$ in case $A=\Lambda_n$). We assume that $\Sigma$ is not a cylinder or a disc. A construction of $\Sigma_A$ and its relation to $\mathcal{D}^b(A)$ are found in \cite{OpperPlamondonSchroll}. As in the case $A=\Lambda_n$ and for every subset $C$ of boundary components of $\Sigma$,  one can define the thick subcategory $\mathcal{D}_C$ which is generated by all objects which are represented by a boundary arc on a component of $C$. In special situations, $\mathcal{D}_C$ can contain $\tau$-invariant objects which correspond to boundary loops, namely if the winding number of the boundary components vanishes. As in the case of the present paper, one can express objects, morphisms, compositions and certain mapping cones in the Verdier quotient $\mathcal{D}^b(A) / \mathcal{D}_C$ through the surface $\Sigma_C$ which is obtained from $\Sigma$ by gluing punctured discs to every component in $C$. The proofs we presented here generalize with minor adaptations in order to account for the special situations described above.
\end{rem}

\subsection{Auto-equivalences and diffeomorphisms of punctured tori}\label{SectionDiffeomorphismsPuncturedTori} We study the relationship between auto-equivalences of $\mathcal{D}^b(C_n)$ and diffeomorphisms of  $\mathscr{T}^n$. Led by the homomorphism $\Psi:\Aut\left(\mathcal{D}^b(\Lambda_n)\right) \rightarrow \MCG(\mathbb{T}_n)$, we construct a surjective homomorphism of groups

\begin{displaymath} \begin{tikzcd}\Upupsilon: \Aut\left(\mathcal{D}^b(C_n)\right) \arrow{rr} && \MCG(\mathscr{T}^n). \end{tikzcd}\end{displaymath}

\noindent The group $\MCG(\mathscr{T}^n)$ will be defined below. The underlying idea of the construction exploits the relationship between the groups of automorphisms of the arc complex of a surface.

 \subsubsection{Arc complexes and their automorphisms}
 
\begin{definition}\label{DefinitionArcComplex}Let $\Sigma$ be a compact, oriented surface with or without boundary and a finite set of marked points. 
The \textbf{arc complex}  $\mathcal{A}(\Sigma)$ of $\Sigma$ is the abstract simplicial complex whose $m$-simplices are collections of $m+1$ homotopy classes of simple arcs  which are pairwise disjoint away from their end points.
\end{definition}
\noindent We refer to $0$-simplices as \textbf{vertices} and to $1$-simplices as \textbf{edges}.
\begin{rem}\label{RemarkFlagComplex}
$\mathcal{A}(\Sigma)$ is a \textit{flag complex} which means that $m+1$ vertices of $\mathcal{A}(\Sigma)$ form an $m$-simplex if and only if they are pairwise connected by an edge.
\end{rem}

\begin{rem}\label{RemarkSimplicialInclusionMap}
	The inclusion $\mathbb{T}_n \hookrightarrow \mathscr{T}^n$ (as defined in Section \ref{SectionToriMarkedPoints}) induces a surjective simplicial map $c:\mathcal{A}(\mathbb{T}_n) \rightarrow \mathcal{A}(\mathscr{T}^n)$ which maps the homotopy class of a simple arc $\gamma \subseteq \mathbb{T}_n$ to the homotopy class of an arc $\tilde{\gamma} \subseteq \mathscr{T}^n$ by connecting each end point of $\gamma$ on a component $B$ with the respective puncture which is enclosed by $B$.
\end{rem}

\noindent It is clear that every self-diffeomorphism $F$ of $\Sigma$ induces a simplicial automorphism $\Phi^{\Sigma}(F)$ of $\mathcal{A}(\Sigma)$. We write $\Aut(\mathcal{A}(\Sigma))$ for the group of simplicial automorphisms of $\mathcal{A}(\Sigma)$, so that $\Phi^{\Sigma}(-)$ becomes a group homomorphism  

$$\begin{tikzcd}\Phi^{\Sigma}:\MCG^{\pm}(\Sigma) \arrow{r} & \Aut(\mathcal{A}(\Sigma))\end{tikzcd}$$

\noindent defined on the \textbf{extended mapping class group} of $\Sigma$, i.e.\  the group of isotopy classes of \emph{all} self-diffeomorphisms of $\Sigma$ which preserve the set of marked points. An isotopy is supposed to leave the boundary and all marked points fixed.  The \textbf{mapping class group} $\MCG(\Sigma)$ is the subgroup of $\MCG^{\pm}(\Sigma)$ consisting of all orientation-preserving mapping classes. If $\Sigma=\mathbb{T}_n$, the previous definition reduces to Definition \ref{DefinitionMappingClassGroup}. We write $\Phi$ instead $\Phi^{\Sigma}$ if $\Sigma$ is apparent from the context.\medskip

\noindent In \cite{OpperDerivedEquivalences}, the map $\Phi=\Phi^{\mathbb{T}_n}$ was used to construct the homomorphism $\Psi:\Aut(\mathcal{D}^b(\Lambda_n)) \rightarrow \MCG(\mathbb{T}_n)$ in Proposition \ref{PropositionExactSequenceAutoequivalences}. Namely, for every $T \in \Aut(\mathcal{D}^b(\Lambda_n))$ one can prove that there exists a unique simplicial automorphism $\overline{T} \in \Aut(\mathcal{A}(\mathbb{T}_n))$ such that $\gamma_{T(X)} \simeq \overline{T}(\gamma_X)$ for every indecomposable object $X \in \mathcal{D}^b(\Lambda_n)$ which is represented by a simple arc, c.f.\ Theorem \ref{TheoremEquationForPsi}. One then shows that $\overline{T} \in \Phi(\MCG(\mathbb{T}_n))$. A result by Disarlo \cite{Disarlo} (for surfaces without punctures) shows that $\Phi$ is bijective and $\Psi(T)$ is the defined as the unique preimage of $\overline{T}$ under $\Phi$.\medskip

\noindent Disarlo's theorem has an analogue for punctured surfaces.  The following is a special case of a result by Irmak and McCarthy.

\begin{thm}[\cite{IrmakMcCarthy}]\label{TheoremIrmakMcCarthy}
For $n \neq 1$ and $\Sigma=\mathscr{T}^n$, $\Phi=\Phi^{\Sigma}$ is an isomorphism. For $\Sigma = \mathscr{T}^1$, $\Phi$ is surjective and induces an isomorphism $\PGL_2(\mathbb{Z}) \cong \Aut(\mathcal{A}(\mathscr{T}^1))$.
\end{thm}

\noindent We like to associate an automorphism of $\mathcal{A}(\mathscr{T}^n)$ to every auto-equivalence of $\mathcal{D}^b(C_n)$. This requires us to understand how simple loops and interior intersections are characterized within $\mathcal{D}^b(C_n)$. This will be achieved by means of the singularity category.

\subsubsection{The singularity category of $C_n$ and the ideal of interior morphisms} \ \medskip

\noindent  We recall that the \textbf{singularity category} $\Sing$ of $C_n$ is by definition the Verdier quotient of $\mathcal{D}^b(C_n)$ with respect to $\Perf(C_n)$. In this sense and due to Proposition \ref{PropositionVerdierQuotientIsDb}, $\Sing$ is obtained as the localization of the Verdier quotient $\quotient{\mathcal{D}^b(\Lambda_n)}{\Dpart}$ at the essential image of $\Dinv$ under the localization functor. One may also consider another category by first localizing $\mathcal{D}^b(\Lambda_n)$ at the subcategory $\Dinv$ and then localizing the resulting quotient at the essential image of the subcategory $\Dpart$. 
\begin{lem}
The two localizations above are canonically equivalent.
\end{lem}
\begin{proof}
By Proposition \ref{PropositionCharaterizationDpartDinv}Lemma \ref{LemmaInvPartOrthogonalCategories}and Proposition\ref{PropositionVerdierQuotientIsDb}, both localizations are Verdier quotients of $\mathcal{D}^b(\Lambda_n)$ at the triangulated subcategory $\Dinv \oplus \Dpart$.
\end{proof}
\noindent Given $X,Y \in \mathcal{D}^b(C_n)$, we denote by $\Hom_{\operatorname{Int}}^{\bullet}(X,Y)$ the kernel of the canonical map
$$\begin{tikzcd}
\Hom^{\bullet}_{\mathcal{D}^b(C_n)}(X,Y) \arrow{r} & \Hom^{\bullet}_{\Sing}(X,Y).
\end{tikzcd}$$
The set $\Hom_{\Int}(X,Y) \subseteq \Hom(X,Y)$ is defined analogously. We refer to elements in $\Hom_{\Int}(X,Y)$ as \textbf{interior morphisms}.\medskip

\noindent Interior morphisms  are stable under triangle equivalences in the following sense.
\begin{lem}\label{LemmaInteriorMorphismStable}
Let $T \in \Aut(\mathcal{D}^b(C_n))$. Then, $T$ induces an isomorphism
$$\begin{tikzcd}
\Hom^{\bullet}_{\operatorname{Int}}(X,Y) \arrow{r} & \Hom^{\bullet}_{\operatorname{Int}}(T(X),T(Y)).
\end{tikzcd}$$
\end{lem}
\begin{proof}$T$ preserves $\Perf(C_n)$ and induces an equivalence on the singularity category. In particular, it preserves the kernel of $\Hom^{\bullet}_{\mathcal{D}^b(C_n)}(X,Y) \rightarrow  \Hom^{\bullet}_{\mathcal{D}_{\Sing}}(X,Y)$.
\end{proof}
\noindent The assertion of the next lemma is analogous to the results of Section 2.9 in \cite{OpperDerivedEquivalences} and provides a topological characterization of interior morphisms.
\begin{lem}\label{LemmaInteriorMorphismsSpanningSet}
Let $X,Y \in \mathcal{D}^b(C_n)$ be indecomposable. Then, $\Hom_{\operatorname{Int}}^{\bullet}(X,Y)$ is the set of morphisms which are supported only at interior intersections of $\gamma_X$ and $\gamma_Y$.
\end{lem}
\begin{proof}Let $U, V \in \mathcal{D}^b(\Lambda_n)$ and let us denote by $\pi_{\operatorname{loop}}, \pi_{\partial}$ and $\pi_{\oplus}$ the localization functors  at $\Dinv$, $\Dpart$ and $\Dinv \oplus \Dpart$, respectively. Since $\Dinv$ and $\Dpart$ are orthogonal, it is not difficult to see that the kernel of the canonical map $\Hom^{\bullet}(U,V) \rightarrow \Hom^{\bullet}(\pi_{\oplus}(U), \pi_{\oplus}(V))$ splits into the sum of the kernels of the corresponding maps induced by $\pi_{\operatorname{loop}}$ and $\pi_{\partial}$. Let $f: X \rightarrow Y$ be interior and choose a preimage $g: X \rightarrow \vartheta^i(Y)$ in $\mathcal{D}^b(\Lambda_n)$ of $f$ under the map \eqref{structuremaps} on page \pageref{structuremaps}. Then, $g$ lies in the kernel of the map induced by $\pi_{\operatorname{loop}}$ and hence factors through $\Dinv$. By Lemma \ref{LemmaInteriorMorphismsClosedUnderComposition}, $g$ must be supported at interior intersections. Moreover, by definition of $\eta$ and $\Dinv^{\perp}=\Dpart$, it follows that $g$ factors through $(\eta^m)_{\vartheta^{-m+i}(Y)}: \vartheta^{-m+i}(Y) \rightarrow \vartheta^i(Y)$ for all $m \in \mathbb{N}$. This implies that all intersections in $\supp g$ correspond to interior intersections of $\gamma_X$ and $\gamma_Y$ (as arcs in $\mathscr{T}^n$).
	
Vice versa, given an interior intersection $p \in\gamma_X \overrightarrow{\cap}\gamma_Y$, we denote by $f: X \rightarrow Y$ the corresponding morphism and choose $g$ as before. Then, $g$ is supported at a single interior intersection and factors through an object which is represented by the loop $(\gamma_{X})_{\operatorname{loop}} \subseteq \mathbb{T}_n$ (see Remark \ref{RemarkDoubleLoopOfAnArcHasVanishingWindingNumber}) and $1$-dimensional local system. As $\Hom_{\operatorname{Int}}^{\bullet}(X,Y)$ is a subvector space, this finishes the proof.
\end{proof}

\noindent We are prepared to construct the map $\Upupsilon$.

\begin{prp}\label{PropositionExistenceGroupHom} For $n > 1$, there exists a surjective group homomorphism 
\begin{displaymath} \begin{tikzcd}\Upupsilon: \Aut\left(\mathcal{D}^b(C_n)\right) \arrow{rr} && \MCG(\mathscr{T}^n). \end{tikzcd}\end{displaymath}
 
\noindent Moreover, there exists a surjective homomorphism $\Aut\left(\mathcal{D}^b(C_1)\right) \rightarrow \operatorname{PSL}_2(\mathbb{Z})$.
\end{prp}
\begin{proof}
Let $F \in \Aut(\mathcal{D}^b(C_n))$. Subsequently, we shall call indecomposable, non-perfect objects $X, Y \in \mathcal{D}^b(\Coh C_n)$ \textit{disjoint} if their representing homotopy classes of arcs contain a pair of representatives with disjoint interior. If $X=Y$, being disjoint is equivalent to $\gamma_X$ being simple. Given two indecomposable, non-perfect objects $X, Y \in \mathcal{D}^b(C_n)$, it follows from Lemma \ref{LemmaInteriorMorphismsSpanningSet} that the canonical map $\Hom^{\bullet}(X,Y) \rightarrow \Hom_{\Sing}^{\bullet}(X,Y)$ is injective if and only if $X$ and $Y$ are disjoint.  By Lemma \ref{LemmaInteriorMorphismStable}, $F$ maps pairs of disjoint simple objects to disjoint simple objects and hence induces an element $\overline{F} \in \Aut(\mathcal{A}(\mathscr{T}^n))$ by Remark \ref{RemarkFlagComplex}. Note that $\overline{F}$ is uniquely determined by the property that $\overline{F}$ maps a vertex $\gamma_X$ corresponding to an object  $X \in \mathcal{D}^b(C_n)$ to the vertex $\gamma_{F(X)}$. In particular, $F \mapsto \overline{F}$ defines a group homomorphism $\Aut(\mathcal{D}^b(C_n)) \rightarrow \Aut(\mathcal{A}(\mathscr{T}^n))$. If $n \geq 2$, we set $\Upupsilon(F) \coloneqq \Phi^{-1}(\overline{F}) \in \MCG^{\pm}(\mathscr{T}^n)$, where $\Phi=\Phi^{\mathscr{T}^n}$. If $n=1$, $\Upupsilon(F) \in \PGL_2( \mathbb{Z})$ is the element which corresponds to $F$ under the isomorphism $\PGL_2(\mathbb{Z}) \cong \Aut(\mathcal{A}(\mathscr{T}^1))$ from Theorem \ref{TheoremIrmakMcCarthy}. 

 We claim that $\overline{F}$ lies in the image of $\MCG(\mathscr{T}^n)$ under $\Phi$. This is a consequence of the covariance of $F$. Namely,
let $\Delta$ denote the triangulation of $\mathscr{T}^n$ as shown in Figure \ref{FigureSpecialTriangulation}. Then, at any puncture $p$ there exist $6$ segments of arcs of $\Delta$ ending $p$. Three of these segments correspond to a chain of  non-zero morphisms $f:U \rightarrow V, g:V \rightarrow W$ and $h: W \rightarrow U[2]$ (see Example \ref{ExampleEndomorphismRingPuncture}) such that $h \circ g \circ f \neq 0$ and where  $U, V$ and $W$ are objects which are represented by pairwise distinct arcs of $\Delta$.
As pointed out above, we have $\overline{F}(\gamma_X)=\gamma_{F(X)}$ whenever $\gamma_X$ is simple. Thus, if $\overline{F}$ was the image of an orientation-reversing homeomorphism, the intersection of the arcs $\gamma_{F(U)}, \gamma_{F(V)}$ and $\gamma_{F(W)}$ and the puncture $F(p)$ would give rise to a non-zero composition $F(U)[2] \rightarrow F(W) \rightarrow F(V) \rightarrow F(U)$. However, there exists no chain of this kind as we can read off from the intersection pattern of the arcs in Figure \ref{FigureSpecialTriangulation}. We conclude that $\overline{F}$ is the image of an orientation-preserving homeomorphism. Note that $\Phi(\MCG(\mathscr{T}^1))$ is the subgroup $\operatorname{PSL}_2(\mathbb{Z}) \subseteq \operatorname{PGL}_2(\mathbb{Z})$.

It remains to prove surjectivity. We recall from \eqref{ShortExactSequenceMCGInclusion} on page \pageref{ShortExactSequenceMCGInclusion} that there is a surjective group homomorphism $\MCG(\mathbb{T}_n) \rightarrow \MCG(\mathscr{T}^n)$ which maps a homeomorphism to its radial extension. In fact, radial extension defines a group homomorphism $\overline{\pi}$ between the corresponding extended mapping class groups. Since $\MCG(\mathscr{T}^n)$ has index $2$ inside $\MCG^{\pm}(\mathscr{T}^n)$ it follows that $\overline{\pi}$ and hence the composition $\alpha \coloneqq \Phi^{\mathscr{T}^n} \circ \overline{\pi} \circ \left(\Phi^{\mathbb{T}_n}\right)^{-1}$ are surjective. The map $\alpha$ can be described as follows. Let $F \in \Aut(\mathcal{A}(\mathbb{T}_n))$ and let $c:\mathcal{A}(\mathbb{T}_n) \rightarrow \mathcal{A}(\mathscr{T}^n)$ denote the surjective simplicial map from Remark \ref{RemarkSimplicialInclusionMap}. Then, 
 
 $$
 \alpha(F)(\tilde{\gamma}) = c(F(\gamma))
 $$
 \noindent  for every simple arc $\tilde{\gamma} \subseteq \mathscr{T}^n$ and every preimage $\gamma \subseteq \mathbb{T}_n$ under $c$. If $n > 1$, it follows that the diagram

 \begin{displaymath}
 		\begin{tikzcd}
 				\Aut(\mathcal{D}^b(\Lambda_n)) \arrow{r} \arrow[twoheadrightarrow, swap]{d}{\Psi} & \Aut(\mathcal{D}^b(C_n)) \arrow{d}{\Phi\circ \Upupsilon} \\
 				\Aut(\mathcal{A}(\mathbb{T}_n)) \arrow[twoheadrightarrow]{r}{\alpha} & \Aut(\mathcal{A}(\mathscr{T}^n)).
 		\end{tikzcd}
 \end{displaymath}
\noindent commutes. Thus, $\Upupsilon$ is surjective due to $\Phi$ being invertible. Finally, if $n=1$, one replaces $\Phi\circ \Upupsilon$ by $\Upupsilon$ in the previous diagram.\end{proof}	

\begin{figure}
			\begin{displaymath}
			\begin{tikzpicture}[scale=2.5]
			\draw[dashed] (0,0)--(4,0);
			\draw[dashed] (0,0)--(0,1);
			\draw[dashed] (0,1)--(4,1);
			\draw[dashed] (4,1)--(4,0);
			
			
			\filldraw[white] (2+0.5,0.5) circle (0.275);
			\draw[dotted, thick] (2+0.5-0.125,0.5)--(2+0.5+0.125,0.5);
			
		\draw[line width=1.5, black] (0,0.5)--(0.5,0.5);
		\draw[line width=1.5, black] (0.5,0.5)--(1.5,0.5);
		\draw[line width=1.5, black] (1.5,0.5)--(2.3,0.5);
		\draw[line width=1.5, black] (2.7,0.5)--(3.5,0.5);
		\draw[line width=1.5, black] (3.5,0.5)--(4,0.5);
		
			\draw[line width=1.5, black] (0.5,0.5)--(1,1);
			\draw[line width=1.5, black] (1,0)--(1.5,0.5);

				\draw[line width=1.5, black] (1.5,0.5)--(2,1);
			\draw[line width=1.5, black] (2,0)--(2.3,0.3);
			
				\draw[line width=1.5, black] (2.7,0.7)--(3,1);
			\draw[line width=1.5, black] (3,0)--(3.5,0.5);
			
				\draw[line width=1.5, black] (0,0)--(0.5,0.5);
			\draw[line width=1.5, black] (3.5,0.5)--(4,1);
			
				\foreach \u in {0,1,3}
			{
			
					\draw[line width=1.5, black] (\u+0.5,0)--(\u+0.5, 0.75)  ;
			\draw[line width=1.5, black] (\u+0.5,1)--(\u+0.5,0.75) ;
				\filldraw[white] (\u+0.5,0.5) circle (1pt);
				\draw[black] (\u+0.5,0.5) circle (1pt);

			}
			\end{tikzpicture}
			\end{displaymath}
			\caption{} \label{FigureSpecialTriangulation}
		\end{figure}
		
\begin{cor}\label{CorollaryDefinitonUpupsilon}
Let $X \in \mathcal{D}^b(C_n)$ be indecomposable such that $\gamma_X$ is a simple arc and let $T \in \Aut(\mathcal{D}^b(C_n)$. Then,
\[
    \gamma_{T(X)} \simeq \Upupsilon(T)(\gamma_X).
\]
\end{cor}

\noindent In order to  describe the kernel of $\Upupsilon$, we will use the following result due to Sibilla.

\begin{prp}[{\cite[Lemma 3.3, Remark 3.4]{SibillaMappingClassGroupAction}}]\label{PropositionSibillaRigidityOfCycles}
Let $\{x_i\}_{i \in \mathbb{Z}_n}$ be a collection of smooth points such that $x_i \in \mathbb{P}^1_i$ for each $i \in \mathbb{Z}_n$. Suppose that $T \in \Aut(\mathcal{D}^b(C_n))$ satisfies

\begin{itemize}
    \item $T(\mathcal{O}_{C_n}) \cong \mathcal{O}_{C_n}$ and
    \item for all $i \in \mathbb{Z}_n$, $T(\Bbbk(x_i)) \cong \Bbbk(x_i)$.
\end{itemize}

\noindent Then, there exists an automorphism $f: C_n \rightarrow C_n$ such that $T$ and the induced equivalence $f^{\ast}: \mathcal{D}^b(C_n) \rightarrow \mathcal{D}^b(C_n)$  are naturally isomorphic. 
\end{prp}

\begin{rem}\label{RemarkAutomorphismsCycle}
\noindent Suppose that $f: C_n \rightarrow C_n$ is automorphism such that $f^{\ast} \in \Aut(\mathcal{D}^b(C_n))$ satisfies the assumptions in Proposition \ref{PropositionSibillaRigidityOfCycles}. Then, as pointed out in \cite{SibillaMappingClassGroupAction},  $f=\operatorname{Id}$ whenever $n \geq 3$. If $n \leq 2$, then $f^2=\operatorname{Id}$ and $f$ is induced from a permutation of the preimages of the singular points in a normalization of $C_n$. Assuming that $x_0$ and $x_1$ correspond to $1 \in \mathbb{P}^1$ and the singular points to $0$ and $\infty$,  $f$ acts on all irreducible components either as the identity or as the involution $x \mapsto x^{-1}$, where $0^{-1} \coloneqq \infty$. In particular, the group of all such $f$ is isomorphic to $\mathbb{Z}_2$. If $n=1$, then $f \in \ker \Upupsilon$. If $n=2$ and $f \neq \text{Id}$, then $f^{\ast}$ permutes the skyscrapers of singular points and hence $f \not \in \ker \Upupsilon$.
\end{rem}

\begin{prp}\label{PropositionKernel}
There exists an isomorphism 

$$
    \ker \Upupsilon \cong \left(\mathcal{N} \ltimes \left(\Bbbk^{\times}\right)^{n}\right) \times \mathbb{Z} \times \Pic^{\mathbb{0}}(C_n),
$$
\noindent where 
$$  \mathcal{N} \cong \begin{cases} \mathbb{Z}_2, & \text{if }n=1; \\ \mathbf{1}, & \text{otherwise.} \end{cases}
$$

\noindent If $n=1$, the structure map of the semi-direct product maps $\overline{1}$ to the automorphism $\lambda \mapsto \lambda^{-1}$.

\end{prp}
\begin{proof} First of all, $[1] \in \ker \Upupsilon$.
By changing coordinates we may assume that $0, \infty \in \mathbb{P}^1$ are preimages of singular points in $\mathbb{P}^1_i$.
Rescaling every irreducible component $\mathbb{P}^1_i$ of $C_n$ by $\lambda_i \in \Bbbk^{\times}$ defines a faithful action of $\left(\Bbbk^{\times}\right)^n$ on $\mathcal{D}^b(C_n)$. Let $\Delta$ denote the triangulation from Figure \ref{FigureSpecialTriangulation}. The shifts of the objects $\Bbbk(x_i)$ are characterized up to isomorphism as those objects $Z \in \Perf(C_n)$ with the following three properties:
\begin{itemize}
    \item $Z$ is spherical; 
    \item There exists a vertical arc $\gamma_1$ and a diagonal arc $\gamma_2$ in $\Delta$ such that 
    \[\dim \Hom^{\bullet}(Z,X_{\gamma_1})=1=\dim \Hom^{\bullet}(Z,X_{\gamma_2});\]

    \item  If  $\gamma \in \Delta \setminus \{\gamma_1, \gamma_2\}$, then $\Hom^{\bullet}(Z,X_{\gamma})=0$.
\end{itemize}
\noindent Line bundles in $\Pic^{\mathbb{0}}(C_n)$ such as the structure sheaf are characterized via all vertical arcs in a similar way. Let $f^{\ast} \in \Aut(\mathcal{D}^b(C_n))$ be an equivalence which is induced by scaling or the tensor product with a line bundle $\mathcal{L} \in \Pic^{\mathbb{0}}(C_n)$. Then, $f^{\ast}$ leaves all singular skyscraper sheaves invariant and maps all smooth skyscrapers to such with support in the same component. By Corollary \ref{CorollaryDefinitonUpupsilon} and Corollary \ref{CorollaryCurvesOfSingularPoints}, it follows that $\Upupsilon(f^{\ast})$  acts trivially on the homotopy classes of the arcs of $\Delta$ and hence $\Upupsilon(f^{\ast})$ is trivial. 
Next, suppose  $\Upupsilon(T)$ is trivial for some equivalence $T \in \Aut(\mathcal{D}^b(C_n))$. Then again, from Corollary \ref{CorollaryDefinitonUpupsilon} and the characterization of $\Bbbk(x_i)$ and the Picard group above we conclude that $T(\Bbbk(x_i)) \cong \Bbbk(y_i)[m_i]$ for some $y_i \in \mathbb{P}^1_i$ and $m_i \in \mathbb{Z}$ for each $i \in \mathbb{Z}_n$ and $T(\mathcal{O}_{C_n})[d] \in \Pic^{\mathbb{0}}(C_n)$ for some $d \in \mathbb{Z}$. As $\Hom^{\bullet}(\mathcal{O}_{C_n}, \Bbbk(x_i))$ is concentrated in degree $0$, we have $m_i=d$ for $i \in \mathbb{Z}_n$.  Hence, after composition with $[-d]$, rescaling and composition with  ${-} \otimes^{\mathbb{L}}T(\mathcal{O}_{C_n})^{\vee}$ we obtain an auto-equivalence $T'$ which satisfies the conditions of Lemma \ref{PropositionSibillaRigidityOfCycles}. Thus, $T' \simeq f^{\ast}$ for some automorphism of $C_n$ and the assertion follows from Remark \ref{RemarkAutomorphismsCycle}.
\end{proof}

\noindent Proposition \ref{PropositionExistenceGroupHom} and Proposition \ref{PropositionKernel} imply the following (Theorem \ref{IntroTheoremAutoGroup} in the introduction).

\begin{cor}\label{CorollaryAutGroup}
	Let $n \geq 1$. There exists a short exact sequence
	
\begin{displaymath}
\begin{tikzcd} \mathbf{1} \arrow{r} & \left(\mathcal{N} \ltimes \left(\Bbbk^{\times}\right)^{n}\right) \times \mathbb{Z} \times \Pic^{\mathbb{0}}(C_n) \arrow{r} & \Aut(\mathcal{D}^b(C_n)) \arrow{r}{\Upupsilon} & G \arrow{r} & \mathbf{1}.  \end{tikzcd}
\end{displaymath}
\noindent where $G=\MCG(\mathscr{T}^n)$ and $\mathcal{N}$ is trivial, if $n > 1$, and $\mathcal{N}=\mathbb{Z}_2$ and $G=\PSL_2(\mathbb{Z})$, if $n=1$. 
	\end{cor}

\begin{rem}\label{RemarkBK}	Our results for the case $n=1$ are in line with \cite[Corollary 5.8]{BurbanKreusslerGenusOne} which states that there exists a short exact sequence
	\begin{displaymath}
	\begin{tikzcd} \mathbf{1} \arrow{r} & \mathcal{K}  \arrow{r} & \Aut\left(\mathcal{D}^b(C_1)\right) \arrow{r}{\varphi} & \operatorname{SL}_2(\mathbb{Z}) \arrow{r} & \mathbf{1},\end{tikzcd}
	\end{displaymath}
	
	\noindent where $\mathcal{K}$ is generated by $[2]$, $\Aut(C_1)$ and tensor products with line bundles of degree $0$.
	The map $\varphi$ associates to an auto-equivalence $T$ its induced map on the Grothendieck group $\mathcal{K}_0(\mathcal{D}^b(C_1)) \cong \mathbb{Z}^2$. The isomorphism is given by the map $X \mapsto (\deg X, \rk X)$.  Note that $\Aut(C_1)\cong \mathbb{Z}_2 \ltimes \Bbbk^{\times}$, where the first factor corresponds to the involution $(x \mapsto x^{-1})$ and the second factor to the scaling automorphisms. Let $p: \operatorname{SL}_2(\mathbb{Z}) \rightarrow \operatorname{PSL}_2(\mathbb{Z})$ denote the projection. Using the explicit isomorphism $\Aut(\mathcal{A}(\mathscr{T}^1)) \cong \operatorname{PGL}_2(\mathbb{Z})$ described in the proof of \cite[Theorem 2.1]{IrmakMcCarthy}, one sees that $p \circ \varphi=\Upupsilon$.
\end{rem}

\begin{rem}\label{RemarkGeneralizationsEquivalences}
We expect Corollary \ref{CorollaryAutGroup} to generalize to Verdier quotients $\mathcal{D}$ of other gentle algebras $\Lambda$ as explained in Remark \ref{RemarkGeneralizations}. In general, $\MCG(\mathscr{T}^n)$ should be replaced by the stabilizer of a line field in the mapping class group of $\Sigma_C$ (in the notation of Remark \ref{RemarkGeneralizations})  and the existence of a map $\Upupsilon$ from  $\Aut(\mathcal{D})$ into the stabilizer seems to follow in a similar way except when $\Sigma_C$ is one of finitely many surfaces, where the automorphism group of the arc complex and the mapping class group of $\Sigma_C$ do not agree. We expect the kernel of $\Upupsilon$ to be isomorphic to $\ker \Psi$.
\end{rem}

\subsection{Faithfulness of group actions}\label{SectionFaithfulness}

\noindent For each collection of smooth points $(x_i)_{i \in \mathbb{Z}_n}$ in $C_n$ such that $x_i \in \mathbb{P}^1_i$, Sibilla \cite{SibillaMappingClassGroupAction} defined a group action of a central extension $\PMCG_{\operatorname{gr}}(\mathscr{T}^n)$ of $\PMCG(\mathscr{T}^n)$ on $\mathcal{D}^b(C_n)$. The group $\PMCG_{\operatorname{gr}}(\mathscr{T}^n)$ fits into a short exact sequence

$$
\begin{tikzcd}
    \mathbf{1} \arrow{r} & \mathbb{Z} \arrow{r}{\iota} & \PMCG_{\operatorname{gr}}(\mathscr{T}^n) \arrow{r} & \PMCG(\mathscr{T}^n) \arrow{r} & \mathbf{1}.
\end{tikzcd}$$
and is generated by the Dehn twists $\{D_{\gamma_{\Pic}}\} \cup \{ D_{\gamma_{\Bbbk(x)}^i} \, | \, i \in \mathbb{Z}_n\}$ and a central element $t\coloneqq \iota(1)$. The projection in the short exact sequence above maps the Dehn twists to themselves and $t$ to the identity. The structure homomorphism $\mathfrak{X}:\PMCG_{\operatorname{gr}}(\mathscr{T}^n) \rightarrow \Aut\left(\mathcal{D}^b(C_n)\right)$ of the group action maps $D_{\gamma_{\Pic}}$ to $T_{\mathcal{O}_{C_n}}$, $D_{\gamma_{\Bbbk(x)}^i}$ to $T_{\Bbbk(x_i)}$ and $t$ to the shift functor.\medskip

\noindent Sibilla conjectured that the action is faithful. We confirm this in the following theorem.

\begin{thm}\label{PropositionActionSplits} The homomorphism $\mathfrak{X}$ is injective. Thus, the group action is faithful.
\end{thm}
\begin{proof}
Consider the diagram of short exact sequences

	$$\begin{tikzcd}[ampersand replacement=\&, column sep=2em]
            \mathbf{1} \arrow{r} \& \mathbb{Z} \arrow{r}  \arrow[hookrightarrow]{dd} \& \PMCG_{\operatorname{gr}}(\mathscr{T}^n) \arrow{r} \arrow{dd}{\mathfrak{X}} \& \PMCG(\mathscr{T}^n) \arrow[hookrightarrow]{dd} \arrow{r} \& \mathbf{1} \\ \\
          \mathbf{1} \arrow{r}\& \ker \Upupsilon \arrow{r} \& \Aut\left(\mathcal{D}^b(C_n)\right) \arrow{r}{\Upupsilon} \& \MCG(\mathscr{T}^n) \arrow{r} \& \mathbf{1},
         \end{tikzcd}$$
        \noindent where the left vertical map is the isomorphism onto the $\mathbb{Z}$-component of $\ker \Upupsilon$. The diagram commutes by Proposition \ref{PropositionSphericalTwistsBecomeDehnTwists}. As the outer vertical maps are injective, so is $\mathfrak{X}$.\end{proof}

\begin{rem}
With a little bit more effort and use of gradings on arcs, it is not difficult to see that $\Upupsilon$ lifts to a group homomorphism $\widetilde{\Upupsilon}: \Aut\left(\mathcal{D}^b(C_n)\right) \rightarrow \MCG_{\operatorname{gr}}(\mathscr{T}^n)$, where $\MCG_{\operatorname{gr}}(\mathscr{T}^n)$ is a central extension of $\MCG(\mathscr{T}^n)$ and one finds that $\widetilde{\Upupsilon} \circ \mathfrak{X}$ is the inclusion $\PMCG_{\operatorname{gr}}(\mathscr{T}^n) \hookrightarrow \MCG_{\operatorname{gr}}(\mathscr{T}^n)$.
\end{rem}

	\section{Curves of vector bundles and simple vector bundles}\label{SectionVectorBundles}
	\noindent In this section, we identify those simple loops on  the $n$-punctured torus which correspond to simple vector bundles on $C_n$ under the bijection in Theorem \ref{IntroTheoremClassificationSphericalObjects}. We further characterize the loops which represent vector bundles and provide an easy way to determine the multi-degree and rank of a vector bundle from its representing loop.
	
	\subsection{Simple vector bundles on cycles of projective lines}
	
	\noindent Simple vector bundles on cycles of projective lines were studied by Burban-Drozd-Greuel \cite{BurbanDrozdGreuel} and Bodnarchuk-Drozd-Greuel \cite{BodnarchukDrozdGreuel}. The latter showed that the rank, the multi-degree and the determinant form a complete set of invariants for the isomorphism class of a simple vector bundle.  We recall that the \textbf{multi-degree} of a vector bundle $\mathcal{E}$ over $C_n$ is the function $\mathbb{d} \in \mathbb{Z}^{\mathbb{Z}_n}$ whose value  $\mathbb{d}(i)$ is the degree of the restricted  pull back bundle $\left(\pi^{\ast}\mathcal{E}\right)|_{\pi^{-1}(\mathbb{P}^1_i)}$, where $\pi$ denotes a normalization map. For any $\mathbb{d} \in \mathbb{Z}^{\mathbb{Z}_n}$, we refer to the sum $\underline{\mathbb{d}}=\sum_{x \in \mathbb{Z}_n}\mathbb{d}(x)$ as the \textbf{total degree} of $\mathbb{d}$.
	 
	 	\begin{thm}[\cite{BodnarchukDrozdGreuel}]\label{TheoremBodnarchukDrozdGreuel}
	 	Let  $r \geq 1$ be a natural number, $\mathbb{d} \in \mathbb{Z}^{\mathbb{Z}^n}$ and $\mathcal{L} \in \Pic^{\mathbb{0}}(C_n) \cong \Bbbk^{\times}$. Then, there exists a simple vector bundle $\mathcal{E}$ on $C_n$ of rank $r$, multi-degree $\mathbb{d}$ and determinant $\mathcal{L}$ if and only if $r$ and $\underline{\mathbb{d}}$ are coprime. Moreover, in this case, the isomorphism class of $\mathcal{E}$ is uniquely determined by the triple $(r, \mathbb{d}, \mathcal{L})$.
	 \end{thm}
	 
	\noindent We provide an alternative proof of Theorem \ref{TheoremBodnarchukDrozdGreuel} in Section \ref{SectionAlternativeProof} which is based on the topological model of $\mathcal{D}^b(C_n)$.  It is natural to ask how a simple vector bundle with a given set of invariants $(r, \mathbb{d}, \mathcal{L})$ can be constructed and a short report on the matter is given in the next section.\medskip
	
	\subsubsection{Vector bundles as glued line bundles and cyclic sequences} \ \medskip

	\noindent Throughout this section, we fix a normalization map $\pi:\widetilde{C}_n \rightarrow C_n$ and denote by $\widetilde{U}_i$ the preimage of $\mathbb{P}^1_i$ under $\pi$.\medskip

	\noindent Let $\mathcal{E}$ be a vector bundle over $C_n$ of rank $r$ and multi-degree $\mathbb{d}$. Then $\pi^{\ast}(\mathcal{E})$ decomposes into a direct sum of $n \cdot r$ line bundles over the components $\widetilde{U}_i \cong \mathbb{P}^1$. In other words, there exists a unique multi-set of integers $D_{\mathcal{E}}=\{m_i^j | \, i \in \mathbb{Z}_n, j \in [0,r)\}$  such that
	
	\begin{displaymath}
	\pi^*(\mathcal{E})\cong \bigoplus_{i \in \mathbb{Z}_n} \mathcal{E}_i,   \qquad \mathcal{E}_i=\bigoplus_{j=0}^{r-1} 
	{\mathcal{O}_{\widetilde{U}_i}\left(m_i^j\right)}.
	\end{displaymath}
	
	\noindent The vector bundle $\mathcal{E}$ then can be thought of as being glued from the vector bundles $\mathcal{E}_i$ by identifying their stalks along the preimages of singular points. The map which identifies the stalks of $\mathcal{E}_i$ and $\mathcal{E}_{i+1}$ over the corresponding singular point in $\mathbb{P}^1_i \cap \mathbb{P}^1_{i+1}$ encodes a matrix $M_i$ over $\Bbbk$. The resulting sequence $(M_i)_{i \in \mathbb{Z}_n}$ describes the vector bundle $\mathcal{E}$. 
	
	\begin{thm}[\cite{BurbanDrozdGreuel}]
	Let $\mathcal{E}$ be an indecomposable vector bundle. There exist $m \geq 1$ such that $m | r$ and splittings of $\mathcal{E}_0, \dots, \mathcal{E}_{n-1}$ for which the matrices $M_i$ are of the following shape. 
	\begin{enumerate}
		\item All $M_i$ are block matrices with blocks of size $m \times m$ and in each row and column of the block division of $M_i$ there exists a single non-zero block.
		\item All non-zero blocks contain the identity matrix except for a single block in $M_0$ which contains a Jordan block $J_m(\lambda)$ for some $\lambda \in \Bbbk^{\times}$.
	\end{enumerate} 
	\end{thm}

	\noindent The scalar $\lambda$ describes the determinant of $\mathcal{E}$ whereas $m$ specifies the position of $\mathcal{E}$ in the corresponding homogeneous tube in the Auslander-Reiten quiver. In particular, $m=1$ if $\mathcal{E}$ is simple. Specializing to $m=1$, we see that the non-zero entries in the matrices $M_i$ determine a cyclic order on $D_{\mathcal{E}}$ by declaring $m_{i+1}^{j'}$ as the successor of $m_{i}^{j}$ if $\left(M_{i}\right)_{j j'} \neq 0$. Thus, every vector bundle of rank $r$ which sits at the mouth of a homogeneous tube defines a cyclic sequence $\seq \in \mathbb{Z}^{\mathbb{Z}_{n r}}$\footnote{In the same way, every vector bundle defines a cyclic sequence.}.\medskip
	
\noindent Burban, Drozd and Greuel found necessary and sufficient conditions for the simplicity of $\mathcal{E}$ in terms of the sequence $\seq$. They read as follows.

\begin{thm}{\cite[Theorem 5.3]{BurbanDrozdGreuel}}\label{TheoremNecessarySufficientConditions}
	Let $\mathcal{E}$ be an indecomposable vector bundle  over $C_n$ of rank $r$ which sits at the mouth of a homogeneous tube. Let $\seq$ denote the corresponding cyclic sequence. Then, $\mathcal{E}$ is simple if and only if all of the following conditions are satisfied.
	
	\begin{enumerate}
	\item  The rank and the total degree of $\mathcal{E}$ are coprime.
	
	\item \label{Condition2TheoremNecessarySufficientConditions} The difference between any two entries in $\seq$ is at most $1$. 
	
	\item \label{Condition3TheoremNecessarySufficientConditions}  Regarding $\seq$ as an infinite sequence $(m_i)_{i \in \mathbb{Z}}$ with period $n \cdot r$, for all $t \in \mathbb{Z}$ the sequence 
	
	$$(m_i - m_{i+t})_{i \in \mathbb{Z}},$$
	
	\noindent contains no subsequence of the form $1, 0, \dots, 0, 1$ or $-1, 0, \dots, 0, -1$.  
		\end{enumerate}
		
\end{thm}

\noindent The conditions in Theorem \ref{TheoremNecessarySufficientConditions} can be used to derive the multi-set $D_{\mathcal{E}}$ and hence the entries of $\seq$, see Remark \ref{RemarkMultiSet}. However, apart from the case $n=1$ (\cite{BurbanStableBundles}) the cyclic order on $D_{\mathcal{E}}$ and therefore the sequence $\seq$ seems to be unknown in general. We provide a closed formula for $\seq$ in Section \ref{SectionClosedFormulaCyclicSequence}.\medskip

\begin{rem}\label{RemarkMultiSet}
 Condition (2) in Theorem \ref{TheoremNecessarySufficientConditions} can be used to derive the multi-set $D_{\mathcal{E}}$. Suppose $\seq$ satisfies the conditions of Theorem \ref{TheoremNecessarySufficientConditions}. Let $i \in [0,n)$ and let $q$ denote the minimum in the sequence $\seq_i, \seq_{i+n}, \seq_{i+2n}, \dots, \seq_{i+(r-1) \cdot n}$. If $a$ denotes the number of times the number $q$ occurs in the preceding sequence, then the Condition (2) implies that

$$
a \cdot q + (r-a) \cdot (q+1)=\sum_{j=0}^{r-1}{\seq_{i+j \cdot n}}=\mathbb{d}(i).
$$
\noindent from which we deduce that $(r-a) \in [0, r)$ is the residue of $\mathbb{d}(i)$ modulo $r$ and  $q$ is the unique integer such that $r \cdot q + (r-a)= \mathbb{d}(i)$. 
\end{rem}

	\subsubsection{The cyclic sequence of a simple vector bundle}\label{SectionClosedFormulaCyclicSequence}
	\ 
	\medskip
	
	\noindent From now on, we write $\seq_x$ instead of $\seq(x)$ for all $\seq \in \mathbb{Z}^{\mathbb{Z}_{n r}}$ and all $x \in \mathbb{Z}_{n r}$. We say that two sequences $\seq, \seq' \in \mathbb{Z}^{\mathbb{Z}_{n r}}$ are \textbf{equivalent} and write $\seq \sim \seq'$ if they agree up to rotation, i.e.\ if there exists $t \in n\mathbb{Z}$ such that $\seq_x= \seq'_{x + t}$ for all $x \in \mathbb{Z}_{n r}$. 
	
	\begin{definition}\label{DefinitionSequencesSimpleVectorBundles}
		Let $\mathbb{d} \in \mathbb{Z}^{\mathbb{Z}_n}$ and let $r \geq 1$ be a natural number such that $\underline{\mathbb{d}}$ and $r$ are coprime. For $i \in \mathbb{Z}$, denote by $S_i=S_i(r,\mathbb{d})$ the sequence of integers defined by $S_0=0$ and the recursive formula $S_{i+1}=S_{i}+ \mathbb{d}(i)$. For $i \in \mathbb{N}$ with residue class $x=\overline{i} \in \mathbb{Z}_{n r}$, let  
		
		$$
		\seq(r, \mathbb{d})_x \coloneqq \begin{cases}\phantom{-} \big|r\mathbb{Z} \cap (S_i, S_{i+1}] \big|, & \text{if }S_{i+1} \geq S_i; \\ - \big|r\mathbb{Z} \cap (S_{i+1}, S_{i}]\big|, & \text{otherwise.}  \end{cases}
		$$ 
		\noindent Define $\seq(r, \mathbb{d}) \in \mathbb{Z}^{\mathbb{Z}_{n r}}$ as the cyclic sequence with entries $\seq(r, \mathbb{d})_x$ as defined above.
	\end{definition}
	\noindent We note that since $\underline{\mathbb{d}}$ and $r$ are coprime, $\seq(r, \mathbb{d})$ is well-defined. Also note that for each $i \in \mathbb{Z}$, 
	$$\sum_{j=0}^{r-1}{\seq(r, \mathbb{d})_{i+j \cdot n}}=\mathbb{d}(i).$$
	
	\noindent   Occasionally, we omit $r$ and $\mathbb{d}$ from the notation and just write $\seq$ instead if the two are apparent from the context.
	
	It is convenient to think of the cyclic sequence $\seq(r, \mathbb{d})$ as a matrix with $n$ columns and $r$ cyclically ordered rows such that we recover the cyclic order by reading the entries from left to right and top to bottom.
	
	\begin{exa}\label{ExampleMatrixSequence}
	Let $n=2=r$, $\mathbb{d}\left(0\right)=2$ and $\mathbb{d}\left(1\right)=-1$. Then, $\underline{\mathbb{d}}=1$, $(S_0, \dots, S_4)=(0, 2, 1, 3, 2)$ and $\seq=\seq(2, \mathbb{d})$ is given by $(\seq_{0}, \dots, \seq_{3})=(1, -1, 1, 0)$. The $2 \times 2$-matrix of $\seq$ is
	
				$$ \left( \begin{matrix}  1 &  -1 \\ 1 & 0  \end{matrix} \right).
					$$	
\end{exa}	 	
	\noindent The following is an immediate consequence of Definition \ref{DefinitionSequencesSimpleVectorBundles}.
	
	\begin{lem}\label{LemmaShiftsByMultiplesOfR}Let $r \geq 1$ and $\mathbb{d} \in \mathbb{Z}^{\mathbb{Z}_n}$ such that $\underline{\mathbb{d}}$ and $r$ are coprime. Let $l \in \mathbb{Z}$, $i \in \mathbb{Z}_n$ and define $\mathbb{e} \in \mathbb{Z}^{\mathbb{Z}_n}$ by 
		\[\mathbb{e}(x)= \begin{cases} \mathbb{d}(x), & \textrm{if $x \neq i$;} \\ \mathbb{d}(x)+l\cdot r, & \textrm{if $x= i$.} \end{cases}\]
		
		\noindent Then, the matrix of $\seq(r, \mathbb{e})$ is obtained from the matrix of $\seq(r, \mathbb{d})$ by adding the vector $ (l \cdot r, \dots,l \cdot r)$ to its $i$-th column.		
	\end{lem}
\begin{proof}
Omitted.
\end{proof}

	\noindent By the previous lemma, we may always assume that $\mathbb{d}(x) \in [0, r)$ for all $x \in \mathbb{Z}_n$ and hence $\seq(r, \mathbb{d})_y \in \{0, 1\}$ for all $y \in \mathbb{Z}_{n r}$.\medskip

	\noindent In the following lemma, we prove that the cyclic sequence $\seq(r,\mathbb{d})$ satisfies the conditions of Theorem \ref{TheoremNecessarySufficientConditions}. 
	
	\begin{prp}\label{LemmaSequencesHaveIntersectionsProperties}
		Let $\mathbb{d} \in \mathbb{Z}^{\mathbb{Z}_n}$ and $r \geq 1$ such that $\underline{\mathbb{d}}$ and $r$ are coprime. Then $\seq=\seq(r, \mathbb{d})$ satisfies the conditions of Theorem \ref{TheoremNecessarySufficientConditions}.
		
	\end{prp}
	\begin{proof}By Lemma \ref{LemmaShiftsByMultiplesOfR}, we may assume that $\mathbb{d}(x) \in [0, r)$ for all $x \in \mathbb{Z}_{n}$.
		If $i, j \in [0, n r]$ such that $j-i \in n \mathbb{Z}$, then in the notation of Definition \ref{DefinitionSequencesSimpleVectorBundles}, 
		
		\[S_{i+1}-S_i=\mathbb{d}\left(i\right)=\mathbb{d}\left(j\right)=S_{j+1}-S_j.\] 
		
		\noindent Thus, $(S_i, S_{i+1}]$ and $(S_j, S_{j+1}]$ have the same length and the cardinality of $r \mathbb{Z} \cap (S_i, S_{i+1}]$ and $r \mathbb{Z} \cap (S_j, S_{j+1}]$ differ by at most one. This proves Condition \eqref{Condition2TheoremNecessarySufficientConditions}.

		Let $i, l\in \mathbb{Z}$ and let $p \geq 1$. The cardinality of $r\mathbb{Z} \cap (S_{i}, S_{i+p}]$ is $\sum_{j=0}^{p-1}{\seq_{i+j}}$. As before, we regard $\seq$ as an infinite sequence $(m_i)_{i \in \mathbb{Z}}$. Now, suppose the infinite sequence $(m_i - m_{i+l})_{i \in \mathbb{Z}}$ contains a subsequence of the form $ \pm {1}, {0}, \dots, {0}, \pm {1}$. It means that for some $i \in \mathbb{Z}$ and some $p \geq 1$, $r\mathbb{Z} \cap (S_{i}, S_{i+p}]$ contains two more (resp. two less) elements than $r\mathbb{Z} \cap (S_{i+l n}, S_{i+l n +p}]$. But the intervals $(S_{i}, S_{i+p}]$ and $(S_{i+l n}, S_{i+l n +p}]$ have the same cardinality which yields a contradiction. This completes the proof of  Condition \eqref{Condition3TheoremNecessarySufficientConditions}. \end{proof}
		
	\begin{cor}[Theorem \ref{IntroThmPullback}]\label{CorollarySimpleVCB}
		Let $\mathcal{E}$ be a simple vector bundle on $C_n$ of multi-degree $\mathbb{d}$ and rank $r$. If $\seq \in \mathbb{Z}^{\mathbb{Z}_{n r}}$ satisfies the conditions in Theorem \ref{TheoremNecessarySufficientConditions}, then $\seq \sim \seq(r, \mathbb{d})$.
	\end{cor}
	\begin{proof}
		This follows from Theorem \ref{TheoremBodnarchukDrozdGreuel}, Theorem \ref{TheoremNecessarySufficientConditions} and Proposition \ref{LemmaSequencesHaveIntersectionsProperties}.
	\end{proof}

	\subsection{Loops and their intersections via ribbon graphs} \ \medskip
	
\noindent Our goal is to translate certain homotopy classes of  loops into a cyclic integer sequence so that for every loop which represents a simple vector bundle of rank $r$ and multi-degree $\mathbb{d}$ this sequence is precisely the sequence $\seq(r, \mathbb{d})$ from Corollary \ref{CorollarySimpleVCB}. This is achieved by interpreting loops on $\mathscr{T}^n$ as walks in a \emph{ribbon graph} which we define in terms of an embedded quiver $\Gamma \subseteq \mathscr{T}^n$.\medskip
	
\noindent Consider the following collection of oriented arcs on $\mathscr{T}^n$:
\begin{figure}[H]
			\begin{displaymath}
			\begin{tikzpicture}[scale=2.5]
			\draw[dashed] (0,0)--(4,0);
			\draw[dashed] (0,0)--(0,1);
			\draw[dashed] (0,1)--(4,1);
			\draw[dashed] (4,1)--(4,0);

			
			\filldraw[white] (2+0.5,0.5) circle (0.275);
		\draw[dotted, thick] (2+0.5-0.125,0.5)--(2+0.5+0.125,0.5);	
				
		
		\begin{scope}[decoration={
    markings,
    mark=at position 0.5 with {\arrow{<}}}
    ] 
    	\draw[ black, postaction={decorate}] (0,0.5)--(0.5,0.5);
		\draw[ black, postaction={decorate}] (0.5,0.5)--(1.5,0.5);
		\draw[ black, postaction={decorate}] (1.5,0.5)--(2.3,0.5);
		\draw[ black, postaction={decorate}] (2.7,0.5)--(3.5,0.5);
		\draw[ black, postaction={decorate}] (3.5,0.5)--(4,0.5);
    \end{scope}
    
         \draw (0.25,0.5) node[above]{$\kappa_0$};
     \draw (1,0.5) node[above]{$\kappa_1$};
	     \draw (1.85,0.5) node[above]{$\kappa_2$};
	          \draw (3.1,0.5) node[above]{$\kappa_{n-1}$};

				\foreach \u in {0,1,3}
			{
			
					\draw[ black, pos=0.75, ->] (\u+0.5,0)--(\u+0.5, 0.75);
			\draw[ black] (\u+0.5,1)--(\u+0.5,0.75);
				\filldraw[white] (\u+0.5,0.5) circle (1pt);
				\draw[black] (\u+0.5,0.5) circle (1pt);

			}
		
		     \draw (0.5,0.25) node[right]{$\varepsilon_0$};
		     		     \draw (1.5,0.25) node[right]{$\varepsilon_1$};
		     		     		     \draw (3.5,0.25) node[right]{$\varepsilon_{n-1}$};
			\end{tikzpicture}
			\end{displaymath}
				\caption{}
 			\label{FigureEpsilonsKappas}
		
 			\end{figure}
			
\noindent We label the horizontal arcs in Figure \ref{FigureEpsilonsKappas} by $\kappa_i$ ($i \in \mathbb{Z}_n$) and denote the vertical arc which intersects $\kappa_i$ and $\kappa_{i+1}$ by $\varepsilon_{i}$. The collection of these arcs cut $\mathscr{T}^n$ into discs. We denote by $\Gamma$ the dual quiver of this dissection: the set of vertices in $\Gamma$ is in bijection with the discs, $\Gamma_1=\{\varepsilon_i, \kappa_i \, | \, i \in \mathbb{Z}_n\}$ and an arrow $\alpha \in \Gamma_1$, which corresponds to an arc $\gamma$, starts (resp. ends) at the vertex which corresponds to the disc which lies on the left (resp. right) hand side of $\gamma$.
	
	\begin{figure}[H]
		\begin{displaymath}
		\begin{tikzpicture}[scale=1]
		 \pgfmathsetmacro{\sz}{2.5};

         \foreach \v in {0, 1} 
		  {
		  \draw[thick, latex-] ($(-173:{\sz/8} and {0.75*\sz})+(\v*\sz,0)-(-\sz/8,0)$) arc (-173:175:{\sz/8} and {0.75*\sz}) node[pos=0.95, xshift=-3pt, left]{$\kappa_{\v}$};
        }   
         \foreach \v in {3} 
		  {
		  \draw[thick, latex-] ($(-173:{\sz/8} and {0.75*\sz})+(\v*\sz,0)-(-\sz/8,0)$) arc (-173:175:{\sz/8} and {0.75*\sz}) node[pos=0.95, xshift=-3pt, left]{$\kappa_{n-1}$};
        } 
        \fill[white] (0,-0.15) rectangle (4*\sz, 0.3);
		    
	        \draw (0*\sz,0) circle (3pt);
	        \filldraw (1*\sz,0) circle (3pt);
	        \draw (2*\sz,0) node {$\cdots$};
	        \filldraw (3*\sz,0) circle (3pt);
		    \draw (4*\sz,0) circle (3pt);
		    
		  \foreach \v in {0, 1} 
		  {
		  \draw[-Latex] ({\sz*\v+0.25},0)--({\sz*\v+\sz-0.25},0) node[midway, below]{$\varepsilon_{\v}$};
        }		
        		  \foreach \v in {2, 3} 
		  {
		    \pgfmathsetmacro{\u}{int(\v-4)};
		  \draw[-Latex] ({\sz*\v+0.25},0)--({\sz*\v+\sz-0.25},0) node[midway, below]{$\varepsilon_{n\u}$};
        }

		\end{tikzpicture}
		\end{displaymath}
		\caption{The quiver $\Gamma$. The two copies of the vertex $\circ$ are identified.} 
		\label{FigureQuiverDissection}
	\end{figure}
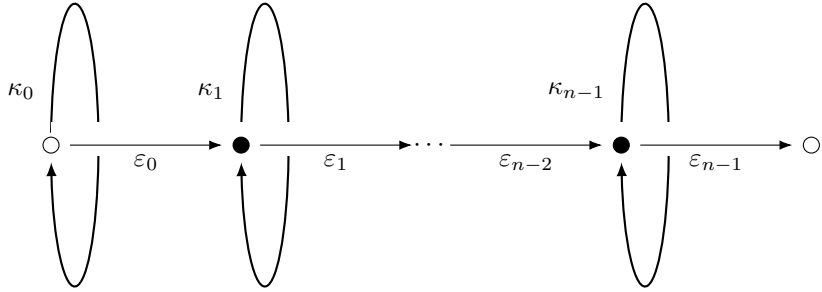

\noindent By construction, $\Gamma$ can be embedded into $\mathscr{T}^n$ by mapping a  vertex to a point inside of its corresponding disc and the arrow associated with an arc $\gamma$ to a simple path $p_{\gamma}$ which crosses $\gamma$ exactly once and no other arc of $\{\varepsilon_i, \kappa_i\}$. The orientation of $\mathscr{T}^n$  defines cyclic orders on the set of half-edges emanating from any vertex of $\Gamma$ and hence turns $\Gamma$ into a \textit{ribbon graph}. By a \textbf{half-edge} we mean any of the two segments of $p_{\gamma}$ between its crossing  with $\gamma$ and one of its end points.

 Let $\Gamma_1^{-1}\coloneqq \{\gamma^{-1} \, | \, \gamma \in \Gamma_1\}$ denote the set of formal inverses of $\Gamma_1$. As usual, $s(\gamma^{-1})\coloneqq t(\gamma)$ and $t(\gamma^{-1})\coloneqq s(\gamma)$. One defines a bijection between the half-edges and elements in $\Gamma_1 \sqcup \Gamma_1^{-1}$ by identifying a segment $\delta$ of $p_{\gamma}$, i.e.\ a half-edge, with $\gamma \in \Gamma_1$, if $\delta$ crosses $\gamma$ from the left, and with $\gamma^{-1}$ otherwise. Regarding the underlying graph of $\Gamma$ as a topological space in the usual way, the embedding $\Gamma \hookrightarrow \mathscr{T}^n$ is a deformation retract and induces an isomorphism $\pi_1(\Gamma) \cong \pi_1(\mathscr{T}^n)$ of fundamental groups.\medskip

\noindent This has a few useful and well-known consequences. First of all, the isomorphism of fundamental groups allows us to regard loops on $\mathscr{T}^n$ as certain words in the alphabet $\Gamma_1 \sqcup \Gamma_1^{-1}$ which we call \textit{admissible walks}. Second, every intersection between loops can be expressed as a common subword of their admissible walks.

\begin{definition}
 An  \textbf{admissible walk} of length $l$ in $\Gamma$ is a function $W:\mathbb{Z}_l \rightarrow \Gamma_1 \sqcup \Gamma_1^{-1}$ such that for all $i \in \mathbb{Z}_l$, $W(i) \neq W(i+1)^{-1}$ and $t(W(i))=s(W(i+1))$, and $W$ is primitive, i.e.\ it is not of the form $\mathbb{Z}_l \xrightarrow{\varphi} \mathbb{Z}_m \xrightarrow{W'} \Gamma_1 \sqcup \Gamma_1^{-1}$, where $l$ is a proper multiple of $m$ and $\varphi$ is the canonical surjection.
\end{definition}

\noindent By definition, we consider two walks $W, W'$ of  length $l$ as equivalent if $W$ and $W'$ agree up to rotation and inversion, i.e.\ there exists $\sigma \in \{\pm 1\}$ and $j \in \mathbb{Z}$ such that $W(i)=W'(\sigma \cdot i + j)$ for all $i \in \mathbb{Z}_l$. Usually, we describe a walk $W$ through a sequence $W(i) W(i+1) \cdots W(i-1)$, where $i \in \mathbb{Z}_l$. \medskip

\noindent The relationship between loops and walks can now be phrased as follows.

\noindent 
	
	\begin{prp}\label{PropositionBijectionCurvesWalks}
	There exists a bijection between homotopy classes of unoriented loops on $\mathscr{T}^n$ and equivalence classes of admissible walks in $\Gamma$.
	\end{prp}
	\noindent 	  Given an admissible walk $W$ we obtain its associated loop simply by regarding $W$ as a closed path in the quiver $\Gamma$: if $W= \alpha_1 \dots \alpha_l$ with $\alpha_i \in \Gamma_1 \sqcup \Gamma_1^{-1}$, we start at the vertex $s(\alpha_1)$ and then walk along the arrows $\alpha_1, \dots, \alpha_l$ of $\Gamma$ by which we mean that if $\alpha_i \in \Gamma_1^{-1}$, then we walk $\alpha_i^{-1} \in \Gamma_1$ backwards from its target to its source. Via the embedding $\Gamma \hookrightarrow \mathscr{T}^n$, we obtain a loop in $\mathscr{T}^n$.\medskip
	
		\noindent If $\gamma$ is a loop, we denote by $W_{\gamma}$ the equivalence class of walks associated with $\gamma$ or, by abuse of notation, a representative thereof.

    \noindent Intersections between loops can be described by maximal common subsequences of their walks. To be precise, suppose $\gamma, \gamma' \subset \mathscr{T}^n$ are loops in minimal position and regard $W=W_\gamma$ and $W'=W_{\gamma'}$ as infinite cyclic sequences with period given by their lengths. Then, every intersection $p \in \gamma \overrightarrow{\cap} \gamma'$ corresponds uniquely to one of the following two situations:\smallskip
		
		\begin{itemize}\label{IntersectionCombinatorics}
			\setlength\itemsep{0.5em}
			\item [a)] After interchanging the roles of $W$ and $W'$ if necessary, there exist decompositions $W=w_1 u c_1 \dots c_l v w_2$ and $W'=w_1' u' c_1 \dots c_l v' w_2'$, where $l \geq 1$ and $c_1, \dots, c_l, u, u', v, v' \in \Gamma_1 \sqcup \Gamma_1^{-1}$, such that $t(u)=t(u')$ but $u \neq u'$ and  $s(v)=s(v')$ but $v \neq v'$, and in the cyclic order of half-edges around $t(u)=t(u')$, $c_1$ lies in between $u'$ and $u$ and in the the counter-clockwise cyclic order of half-edges around $s(v)=s(v')$, $c_l$ lies between $v'$ and $v$.
			
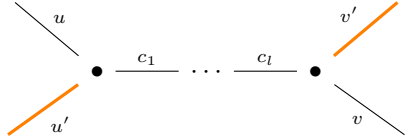
			\begin{figure}[H]
				\begin{tikzcd}
				 \text{} \arrow[dash]{dr}[black]{u} & & \text{} & \text{} & \text{} \\
				 & \bullet \arrow[dash]{r}{c_1} \arrow[dash, orange, very thick]{dl}[black]{u'} & \cdots  \arrow[dash]{r}{c_l} & \bullet \arrow[dash, swap]{dr}{v} \arrow[dash, orange, very thick]{ur}[black]{v^\prime} \\
				 \text{} & & & & \text{} 
				\end{tikzcd}
			\end{figure}
			
			\item[b)] After interchanging the roles of $W$ and $W'$ if necessary, there exist decompositions $W= w_1 u v w_2$ and $W'=w_1 u' v' w_2'$, where $u,u', v, v' \in \Gamma_1 \sqcup \Gamma_1^{-1}$ are pairwise distinct, such that $x\coloneqq t(u)=t(u')=s(v)=s(v')$ and in the counter-clockwise cyclic order of half-edges around $x$ the order is $u, u', v, v'$. 
	
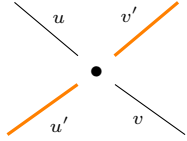
						\begin{figure}[H]
				\begin{tikzcd}
				\text{} \arrow[dash]{dr}[black]{u} &  & \text{} \\
				& \bullet \arrow[dash, orange, very thick]{dl}[black]{u'} \arrow[dash, swap]{dr}{v} \arrow[dash, orange, very thick]{ur}[black]{v'} \\
				\text{} & & \text{} 
				\end{tikzcd}
			\end{figure}
			
		\end{itemize}

	\begin{exa}
		The loop $\gamma^i_{\Bbbk(x)}$ corresponds to the walk $\kappa_i$ and $\gamma_{\Pic}$ corresponds to $\varepsilon_{0} \varepsilon_{1} \dots \varepsilon_{n-1}$. The unique intersection of $\gamma^i_{\Bbbk(x)}$ and $\gamma_{\Pic}$ is an example for the case b) and corresponds to the decompositions $ \dots \kappa_i \kappa_i \dots$ and $ \dots \varepsilon_i \varepsilon_{i+1} \dots $, i.e. $(u,u',v, v')=(\kappa_{i}, \varepsilon_i, \kappa_i^{-1}, \varepsilon_{i-1}^{-1})$.
	\end{exa}
	
	\subsection{The class of loops which represent vector bundles}
	\ \medskip
	
	 \noindent In this section, we describe the set of loops in $\mathscr{T}^n$ which represent vector bundles over $C_n$. \medskip
	 
	 \noindent The following definitions will play a central role in this section.
	 
	 \begin{definition}
	  Let  $\CVb{n}$ denote the set of all homotopy classes of primitive loops on $\mathscr{T}^n$ which are represented by a walk in $\Gamma$ without any letters of the form $\varepsilon_i^{-1}$ ($i \in \mathbb{Z}_n$). 
	 \end{definition}
 
 \noindent If $\gamma$ is a loop we write $\gamma \in \CVb{n}$ to indicate that the homotopy class of $\gamma$ lies in $\CVb{n}$.
 \begin{definition}
  Let $\Vect^n$ denote the additive closure of all indecomposable objects $X \in \Perf(C_n)$ such that $\gamma_X \in \CVb{n}$.	
 \end{definition}
 
	\noindent Elements of $\CVb{n}$ are in bijection with equivalence classes of primitive sequences $\seq \in \mathbb{Z}^{\mathbb{Z}_{n r}}$, where $r \geq 1$. Indeed, for any  sequence $\seq \in \mathbb{Z}^{\mathbb{Z}_{n r}}$ the corresponding homotopy class $\gamma(\mathbbm{m})$ belongs to the walk
	\[\cdots \varepsilon_0 \kappa_0^{\seq_0} \varepsilon_1 \kappa_1^{\seq_1} \cdots \varepsilon_{nr-1} \kappa_{nr -1}^{\seq_{nr-1}} \varepsilon_0 \cdots.\]

\noindent Note that $\gamma(\seq)$ is primitive if and only if $\seq$ is primitive. For a possibly non-primitive loop $\gamma$ which is a power of a loop in $\CVb{n}$ we shall denote by $\seq_{\gamma}$ its associated integer sequence.

 \begin{rem}
	The set $\CVb{n}$ admits a simple topological characterization: an immersed loop $\gamma$ is contained in $\CVb{n}$ if and only if $\gamma$ can be oriented in such a way that its derivative has everywhere positive latitudinal coordinate, c.f.\ Theorem \ref{IntroTheoremSimpleVectorBundles}. In other words, the curve $\gamma$ always travels towards the right in Figure \ref{FigureEpsilonsKappas}.
\end{rem}	
	
\begin{definition}
Let $r \geq 1$ and $\mathbb{d} \in \mathbb{Z}^{\mathbb{Z}_{n r}}$ such that $\underline{\mathbb{d}}$ and $r$ are coprime. Define $\gamma(r, \mathbb{d}) \in \CVb{n}$ as the homotopy class of $\seq(r, \mathbb{d})$.
\end{definition}
\noindent As reader probably guesses at this point $\gamma(r, \mathbb{d})$ contains simple loops and represents the simple vector bundles of rank $r$ and multi-degree $\mathbb{d}$. A proof is given in this section.
	
	\begin{rem}\label{RemarkDehnTwistActionOnSequences}
		With $\mathbb{e}$ and $\mathbb{d}$ as in Lemma \ref{LemmaShiftsByMultiplesOfR} the $l$-th power of the Dehn twist about $\gamma_{\Bbbk(x)}^i$ transforms the homotopy class $\gamma(r, \mathbb{d})$ into the homotopy class of $\gamma(r, \mathbb{e})$. An analogous statement is true for arbitrary loops $\gamma(\seq)$.
	\end{rem}

	\noindent 	For $\gamma$ a power of a loop in $\CVb{n}$, we denote by $\seq_{\gamma} \in \mathbb{Z}^{\mathbb{Z}_{nr}}$ the sequence associated with $\gamma$ and define the \textbf{rank} $\rk \gamma\coloneqq \rk \seq_{\gamma}$ and the \textbf{multi-degree} $\mathbb{d}(\gamma)\coloneqq \mathbb{d}(\seq_{\gamma})$. As usual, the \textbf{total degree} $\underline{\mathbb{d}}(\gamma)$ of $\gamma$ is the sum of all entries of $\mathbb{d}(\gamma)$. \medskip

\noindent The three invariants $\rk \gamma, \mathbb{d}(\gamma)$ and $\underline{\mathbb{d}}(\gamma)$ have a straightforward interpretation on the surface.

\begin{lem}
Let $\gamma \subseteq \mathscr{T}^n$ be a loop. Then, $\mathbb{d}(\gamma)(i)$ equals the number of intersections (counted with signs) of $\gamma$ with the arc $\kappa_i$ in Figure \ref{FigureEpsilonsKappas}.
\end{lem}
\begin{proof}
Omitted.
\end{proof}

\noindent In a similar way, rank and total degree of a loop $\gamma$ have a topological interpretation too. One embeds $\gamma$ into the closed torus $\mathbb{R}^2/ \mathbb{Z}^2 \cong S^1 \times S^1$. Then, composing $\gamma$ with the projection onto the first component of the product determines a self-map of $S^1$ and hence an integer $r_\gamma \in  \mathbb{Z} \cong \pi_1(S^1)$. The number $r_{\gamma}$ measures how many times $\gamma$ wraps around the latitudinal axis of the  torus. The projection onto the second component determines an integer $d_{\gamma}$ which counts the number of full positive turns of $\gamma$ around the longitudinal axis.\medskip

\noindent Comparing the previous construction with the definitions of rank and total degree, we observe the following:

\begin{lem}\label{LemmaGeometricInterpretationRankMultiDegree}
Let $\gamma \subseteq \mathscr{T}^n$ be a loop. Then, $r_{\gamma}$ agrees with $\rk \gamma$ and $d_{\gamma}$ equals $\underline{\mathbb{d}}(\gamma)$. 
\end{lem} 
\begin{proof}
Omitted.
\end{proof}

	\noindent The importance of $\CVb{n}$, $\Vect^n$ and $\gamma(r, \mathbb{d})$ becomes clear in the main result of this section:
	\begin{thm} The following is true. 
		\label{TheoremImagesOfSimpleVectorBundles}
		\begin{enumerate}
		\item A curve $\gamma \subseteq \mathscr{T}^n$ represents a vector bundle over $C_n$ if and only if $\gamma \in \CVb{n}$. Thus, $\Vect^n$ coincides with the additive closure of all vector bundles and their shifts. 
		\item Let $\mathcal{E}$ be a vector bundle of multi-degree $\mathbb{d}$ over $C_n$ which sits at the $k$-th level of its homogeneous tube in the Auslander-Reiten quiver. Then $\mathbb{F}(\mathcal{E}) \cong \P(\seq)(\mathcal{V})$, where  
		\begin{displaymath}
		\begin{array}{ccc}
		\dim \mathcal{V}=k & \text{and} & (\mathbb{d}(\seq), \rk \seq)=\big(\mathbb{d}, \frac{\rk \mathcal{E}}{k})\big).
		\end{array}
		\end{displaymath}
		\item Any loop in $\gamma(r, \mathbb{d})$ represents the simple vector bundles of rank $r$ and multi-degree $\mathbb{d}$.
		\end{enumerate}  
	\end{thm}
	
\noindent	The proof of Theorem \ref{TheoremImagesOfSimpleVectorBundles} can be found at the end of this section and can be summarized as follows.\medskip 

\noindent \textbf{Outline of the proof:} \,  We show that any object in $\mathbb{F}(\Vect^n) \subseteq \mathcal{D}^n(\Lambda_n)$ is a shift of an iterated extension (of degree $1$) of objects which correspond to line bundles over $C_n$. As the category of vector bundles is closed under extensions this shows the ``if''-part in Theorem \ref{TheoremImagesOfSimpleVectorBundles} (1). Theorem \ref{TheoremImagesOfSimpleVectorBundles} (2) follows along the way due to the additivity of rank and multi-degree. The ``only-if''-direction in Theorem \ref{TheoremImagesOfSimpleVectorBundles} (1) is implied by a result of Ballard which characterizes shifts of vector bundles inside the derived category by means of their morphism spaces to skyscraper sheaves. At the heart of all our arguments lies the description of certain morphisms in $\mathbb{F}(\Vect^n)$ via intersections which we rephrase in terms of the associated integer sequences. This also enables us to  derive simplicity of $\gamma(r, \mathbb{d})$. \medskip
	
\noindent The first lemma describes intersections of loops in $\CVb{n}$ in a convenient combinatorial way which resembles Condition \eqref{Condition2TheoremNecessarySufficientConditions} and \eqref{Condition3TheoremNecessarySufficientConditions} in Theorem \ref{TheoremNecessarySufficientConditions}. As before, we regard a sequence $\seq \in \mathbb{Z}^{\mathbb{Z}_{n r}}$ as an infinite sequence with period $n \cdot r$. For such an infinite sequence $\seq=(m_i)_{i \in \mathbb{Z}}$ and $l \in \mathbb{Z}$, we write $\seq[l]$ for the infinite sequence $(m_{i+l})_{i \in \mathbb{Z}}$.

	\begin{lem}\label{LemmaCombinatorialIntersections}
	Suppose that $\gamma, \gamma' \in \CVb{n}$ are in minimal position with ranks $r=\rk \gamma$ and $r'=\rk \gamma'$. Let $\mathrm{lcm}(r,r')$ denote the least common multiple of $r$ and $r'$. Then, the set of intersections of $\gamma$ and $\gamma'$ is in bijection with the union of the following two sets. 
		\begin{enumerate}
			\item The set of subsequences of $\seq_{\gamma} - \seq_{\gamma'}[l]$ up to shift by a multiple of $\mathrm{lcm}(r, r')$, of the form 
			$$a, 0, \dots,0, b,$$ 
			
			\noindent where $a, b > 0 $ or $a, b < 0$ and $l \in \mathbb{Z}$.
			
			\item The set of triples $(x, y, q)$, where $x  \in \mathbb{Z}_{n r}$, $y \in  \mathbb{Z}_{ nr'}$ and $0 \leq q \leq |\left(\seq_{\gamma}\right)_x-\left(\seq_{\gamma'}\right)_y|-2$.
			
		\end{enumerate} 
	\end{lem}

	\begin{proof}
		This is just a paraphrase of the bijection between intersections and maximal common subsequences as described on page \pageref{IntersectionCombinatorics}. 
		
		Suppose $\seq_{\gamma} - \seq_{\gamma'}[l]$  contains a subsequence $a, 0, \dots, 0, b$ with $a, b > 0$ or $a,b < 0$. Then, $W=W_\gamma$ and $W'=W_{\gamma'}$ contain sequences $\kappa_x c \kappa_y$ and $\varepsilon_x c \varepsilon_{y}$, where $c= \kappa_x^u \varepsilon_{x+1} \cdots \varepsilon_{y-1} \kappa_y^v$ and $u, v \geq 0$. One verifies easily that the common subsequence $c$ corresponds to an intersection.
		
		Next, suppose $|\left(\seq_{\gamma}\right)_x-\left(\seq_{\gamma'}\right)_y| \geq 2$. Set $m \coloneqq \left(\seq_{\gamma}\right)_x$ and $m' \coloneqq \left(\seq_{\gamma'}\right)_y$.  Without loss of generality $m > m'$. The corresponding subsequence $\varepsilon_{x} {\kappa_{x}}^{m} \varepsilon_{x+1}$ of $W$ contains $m' - m - 1$ subsequences of the form $\kappa_x c \kappa_x$, where $c=\kappa_x^{m'}$ if $m' \neq 0$ and $c$ is empty otherwise. Since $W'$ contains a subsequence $\varepsilon_x c \varepsilon_{x+1}$, these subsequences give rise to $m' - m - 1$ different intersections.
		
		Finally, we observe that the previous cases cover all possible decompositions of $W$ and $W'$ which correspond to an intersection of $\gamma$ and $\gamma'$.
	\end{proof}

	\begin{cor}\label{CorollaryThetaIsSimple}
		Let $r \geq 1$ and $\mathbb{d} \in \mathbb{Z}^{\mathbb{Z}_n}$ such that $r$ and $\underline{\mathbb{d}}$ are coprime. Then, $\gamma(r, \mathbb{d})$ contains a simple loop.
	\end{cor}
	\begin{proof}
		This follows from Proposition \ref{LemmaSequencesHaveIntersectionsProperties} and Lemma \ref{LemmaCombinatorialIntersections}.
	\end{proof}

	\begin{exa}\label{ExampleLoopOfSimpleVectorBundle}
		As in Example \ref{ExampleMatrixSequence}, let $n=2=r$, $\mathbb{d}\left(0\right)=2$ and $\mathbb{d}\left(1\right)=-1$. Then, $\seq=\seq(2, \mathbb{d})$ is given by $(\seq_{0}, \dots, \seq_{3})=(1, -1, 1, 0)$ and a representative $\gamma$ of $\gamma(2, \mathbb{d})$ is depicted in Figure \ref{FigureLoopOfSimpleVectorBundle}. The intersections of $\gamma$ with the vertical lines in Figure \ref{FigureLoopOfSimpleVectorBundle} divide $\gamma$ into $n \cdot r=4$ segments. The entries of the sequence $\seq$ encode the signed intersection number of each segment with the arcs $\kappa_i$.
		
			\begin{figure}
			\begin{displaymath}
			\begin{tikzpicture}[scale=3.25]
			\draw (0,0)--(2,0);
			\draw (0,0)--(0,1);
			\draw (0,1)--(2,1);
			\draw (2,1)--(2,0);

		\draw (0,0.5)--(2,0.5);

			\draw (0.5, 0.875)--(0.5+0.125,1);
			\draw[->] (0.625, 0)--(1.35,0.875-0.15);
			\draw (1.35,0.875-0.15)--(1.5,0.875);
			\draw[->] (0.5+1,0.875)--(0.5+1+0.25, 0.875-0.1875);
			\draw (0.5+1+0.25, 0.875-0.1875)--(0.5+1+0.5, 0.875-0.375);
			\draw[->] (0, 0.875-0.375)--(0.25, 0.875-0.375-0.1875);
			\draw (0.25, 0.875-0.375-0.1875)--(0.5, 0.125);
			\draw[->] (0.5, 0.125)--(1.2, 1-0.175);
			\draw (1.2, 1-0.175)--(1.375, 1);
			\draw (1.375, 0)--(1.5, 0.125);
			\draw[->] (1.5, 0.125)--(1.75, 0.0625);
			\draw (1.75,0.0625)--(2,0);
			\draw[->] (0, 1)--(0.25, 1-0.0625);
			\draw (0.25,1-0.0625)--(0.5,0.875);
			
				\foreach \u in {0, 1}
			{
			
				\draw ({\u+0.5}, 0)--({\u+0.5}, 1);
				
				\filldraw[white] (\u+0.5,0.5) circle (0.75pt);
				\draw[black] (\u+0.5,0.5) circle (0.75pt);
				
				\filldraw[blue] (\u+0.5, 0.875) circle (0.75pt);
				\filldraw[blue] (\u+0.5, 0.125) circle (0.75pt);
			}
			
			\end{tikzpicture}
			\end{displaymath}
			\caption{The loop $\gamma=\gamma(2, \mathbb{d})$. The points $\color{blue}\bullet$ divide $\gamma$ into $4$ segments.}
			\label{FigureLoopOfSimpleVectorBundle}
		\end{figure}
	\end{exa}
	
	\begin{rem}[Construction of simple representatives] For a general homotopy class of a loops on a surface it is not obvious how to construct a representative in minimal position right away. A simple representative of $\gamma(r, \mathbb{d})$ can be constructed from the sequence $S_i=S_i(r, \mathbb{d})$ (see Definition \ref{DefinitionSequencesSimpleVectorBundles}) in the following way.
     \begin{enumerate}
         \item Draw $r$ distinct points in the interior of every arc $\varepsilon_i$ which we label by $z_i^0, \dots, z_i^{r-1}$ while following the orientation of $\varepsilon_i$. If $s \in \mathbb{Z}$, set $z_x^s \coloneqq z_x^t$, where $t \in [0,r)$ such that $s \equiv t \mod r$.
         
         \item  For each pair $(i, j) \in \mathbb{Z}_n \times [0, n r]$,  connect $z_i^{S_j}$ and $z_{i+1}^{S_{j+1}}$ by the projection of a straight line in $\mathbb{R}^2$.
     \end{enumerate}
	The resulting loop is a simple representative of $\gamma(r, \mathbb{d})$. The loop in Figure \ref{FigureLoopOfSimpleVectorBundle} was constructed by means of this procedure.
	\end{rem}
	
	\begin{lem}\label{LemmaQuivers}
		Let $\mathbbm{m}  \in \mathbb{Z}^{\mathbb{Z}_{n r}}$ and let $\mathcal{V}$ be an indecomposable local system on $\gamma(\seq)$. Then, $\P_{\gamma(\mathbbm{m})}(\mathcal{V})$ is concentrated in two or three cohomological degrees.
	\end{lem}
	\begin{proof}
		Each subsequence $\varepsilon_x \kappa_x^{\mathbbm{m}_x} \varepsilon_{x+1}$ corresponds to a subquiver of $Q(\gamma(\mathbbm{m}))$ of the form
		\begin{equation}\label{equationquiver1}
		\begin{tikzcd}
		\bullet \arrow{r}{b_i}  & \bullet \arrow{r}{a_i} & \bullet & \arrow[swap]{l}{c_i} \bullet \arrow{r}{a_i} & \cdots & \arrow[swap]{l}{c_i} \bullet & \arrow[swap]{l}{d_i} \bullet,  
		\end{tikzcd}
		\end{equation}

		\noindent if $\seq_x \geq 0$ and 
		
		\begin{equation}\label{equationquiver2}    
		\begin{tikzcd}
		\bullet \arrow{r}{b_ic_i} & \bullet & \bullet \arrow[swap]{l}{a_i} \arrow{r}{c_i} & \cdots \arrow{r}{c_i} & \bullet & \arrow[swap]{l}{d_ia_i} \bullet,  
		\end{tikzcd}
		\end{equation}
		
		\noindent if $\seq_x < 0$. 
	\end{proof}

	\begin{rem}\label{RemarkNumberOfArrows}
	The number of arrows with labels $a_i$ (resp.\ $c_i$) in (\ref{equationquiver1}) and (\ref{equationquiver2}) is $|\mathbbm{m}_x|$.
\end{rem}

	\noindent In what follows, for any (not necessarily) primitive sequence $\seq$, we write $\P(\seq)$ as short for any representative $\P_{\gamma(\seq)}(\mathcal{V})$ of $\gamma(\seq)$, where $\dim \mathcal{V}=1$, which is concentrated in degrees $-1$ and $0$, or, has entries in degrees $-1,0$ and $1$. It follows from the discussion preceeding Theorem \ref{TheoremResolutionCrossings} that if $\seq=\mathbb{e}^t$ for some $t > 1$ and a primitive sequence $\mathbb{e}$, then $\P(\seq) \cong \bigoplus_{i=1}^m \P_{\gamma(\mathbb{e})}(\mathcal{V}_i)$ for indecomposable local systems $\mathcal{V}_1, \dots, \mathcal{V}_m$ with $\sum_{i=1}^m \dim \mathcal{V}_i=t$.\medskip

	\noindent The next lemma is the key to showing that $\P(\seq)$ is an iterated extension of images of line bundles.
	\begin{lem}\label{LemmaMorphismsVectorBundles}
		Let $r > 1$, $\seq \in \mathbb{Z}^{\mathbb{Z}_{ n r}}$ and $\ell \in \mathbb{Z}^{\mathbb{Z}_n}$. Then, the following is true.
		\begin{enumerate}

			\item \label{list1} A morphism $f \in \Hom^{\bullet}\big(\P(\ell), \P(\seq)\big)$  which corresponds to a subsequence $a, 0, \dots, 0, b$, $a, b, > 0$,  of $\ell - \seq[p]$ for some $p \in \mathbb{Z}$ is of degree $0$.
			\item If the above sequence has length $n+1$, then $C_f$ is isomorphic to a complex $\P(\seq')$ for a sequence $\seq'$ of rank $r-1$ and multi-degree $\mathbb{d}(\seq)-\mathbb{d}(\ell)$. In other words, $\P(\seq)$ is an extension of $\P(\ell)$ and $\P(\seq')$.
		\end{enumerate} 
	\end{lem}
	\begin{proof}We may assume that $p=0$ and that all values of $\ell, \seq$ are non-negative. Let $a, 0, \dots, 0, b$ be a subsequence of $\ell - \seq$ with $a, b > 0$. Then, following the proof of Lemma \ref{LemmaQuivers} (see also Remark \ref{RemarkNumberOfArrows}) and the definition of $f$ from \cite{OpperPlamondonSchroll}, $f$ is a  graph map as described in the proof of Lemma \ref{LemmaMappingConesAtTwoEnds}. Its  diagram has the following shape:
		
		\begin{displaymath}
		\begin{tikzcd}
		\bullet   & \bullet \arrow[equal]{d} \arrow[swap]{l}{c_i} \arrow{r}{a_i} & \bullet \arrow[equal]{d} & \arrow[swap]{l}{c_i} \bullet \arrow[equal]{d} \arrow{r}{a_i} & \cdots & \arrow[swap]{l}{c_i} \bullet \arrow[equal]{d} \arrow{r}{a_i} &  \bullet
		\\
		\bullet \arrow{r}{b_i}  & \bullet \arrow{r}{a_i} & \bullet & \arrow[swap]{l}{c_i} \bullet \arrow{r}{a_i} & \cdots & \arrow[swap]{l}{c_i} \bullet & \arrow[swap]{l}{d_i} \bullet.  
		\end{tikzcd}
		\end{displaymath}
		\noindent The top line corresponds to the subquiver of $Q(\gamma(\ell))$ associated with the subsequence \[\kappa_i^{\mathbbm{m}_i+1} \varepsilon_{i+1} \cdots \varepsilon_{i} \kappa_i^{\mathbbm{m}_{i+n}+1}\]
		\noindent of $Q(\gamma(\ell))$; the lower line corresponds to  the subquiver of $Q(\gamma(\mathbbm{m}))$ which is associated with 
		
		\[\varepsilon_i\kappa_i^{\mathbbm{m}_i} \varepsilon_{i+1} \cdots \varepsilon_{i} \kappa_i^{\mathbbm{m}_{i+n}} \varepsilon_{i+1}.\]
		
		\noindent  By Theorem \ref{TheoremResolutionCrossings} we obtain a loop which represents the mapping cone of $f$ by resolving the corresponding intersection. The reduced walk of the resolved curve is the reduction of
 $$
 \cdots \varepsilon_i \kappa_i^{ \seq_i} \left( \kappa_i^{-\ell_i} \varepsilon_i^{-1} \kappa_{i-1}^{-\ell_{i-1}} \varepsilon_{i-1}^{-1} \cdots \kappa_{i+1}^{-\ell_{i+1}} \varepsilon_{i+1}^{-1} \right) \varepsilon_{i+1} \kappa_{i+1}^{\seq_{i+1}} \cdots \varepsilon_{i} \kappa_i^{\seq_{i+n}} \cdots.
$$
\noindent Since $\ell_{i+s}= \seq_{i+s}$ for all $s \in (1, n)$, the reduction is
 $$
\cdots \varepsilon_{i-1}\kappa_{i-1}^{\seq_{i-1}} \varepsilon_i \kappa_i^{ \seq_i-\ell_i+\seq_{i+n}} \varepsilon_{i+1} \kappa_i^{\seq_{i+n+1}} \cdots.
$$

\noindent  Thus, the resolved curve is a power of a loop in $\CVb{n}$, has rank $r-1$ and multi-degree $\mathbb{d}(\mathbbm{m})- \mathbb{d}(\ell)$. 
	\end{proof}

	\begin{prp}\label{PropositionExtensions}
	Every object in $\Vect^n$ is a direct sum of shifted vector bundles. Moreover, the rank (resp.\ multi-degree) of a vector bundle coincides with the rank (resp.\ multi-degree) of its corresponding homotopy class in $\CVb{n}$.
	\end{prp}
\begin{proof}
	
	Let $\seq \in \mathbb{Z}^{\mathbb{Z}_{n r}}$ be a primitive cyclic sequence of rank $r \geq 2$. Let $i \in \mathbb{Z}$. Define $\ell \in \mathbb{Z}^{\mathbb{Z}_n}$ by $\ell_i= \max\{\seq_i, \seq_{i+n}\} +1$ and $\ell_{i+j}=\seq_{i+j}$ for all $j \in (0,n)$. Then by Lemma \ref{LemmaMorphismsVectorBundles},  $\P(\seq)$ is an extension of $\P(\ell)$ and $\P(\seq')$ for some $\seq'$ of rank $r-1$ and multi-degree $\mathbb{d}(\seq)-\mathbb{d}(\ell)$.
By induction on $r$, it follows that $\P(\seq)$ is an iterated extension of objects of the form $\P(\mathbb{d})$ with $\mathbb{d}$ primitive and $\rk \mathbb{d}=1$. Here we used the fact that every object in the homogeneous tube of a vector bundle is again a vector bundle which in turn is a consequence of the fact that the class of vector bundles is closed under extensions. From Theorem \ref{TheoremImages} and Remark \ref{RemarkTwistFunctorSmoothPointTensorProduct} we conclude that $\P(\mathbb{d})$ is the image of a line bundle of multi-degree $\mathbb{d}$. The class of vector bundles is closed under extensions. Thus, the assertion follows from additivity of rank and multi-degree as well as the fact that $\mathbb{F}$ is a triangulated embedding.
\end{proof}
	
	\noindent We are finally prepared to  prove Theorem \ref{TheoremImagesOfSimpleVectorBundles} and hence Theorem \ref{IntroTheoremSimpleVectorBundles} from the introduction.
	
	\begin{proof}[Proof of Theorem \ref{TheoremImagesOfSimpleVectorBundles}]
		By Proposition \ref{PropositionExtensions} a loop $\gamma \in \CVb{n}$ represents vector bundles of rank $\rk \gamma$ and multi-degree $\mathbb{d}(\gamma)$.
		By Theorem \ref{IntroTheoremClassificationSphericalObjects}, simple loops in $\CVb{n}$ correspond to simple vector bundles. For every smooth point $x \in C_n$,  we consider the family of functors 
	
	\[\mathcal{F}_y \coloneqq - \otimes^{\mathbb{L}} \left(\mathcal{L}(x) \otimes \mathcal{L}(y)^{\vee}\right).\]
	
	\noindent  By Proposition \ref{PropositionSphericalTwistsBecomeDehnTwists}, we see that for all indecomposable $Z \in \Perf(C_n)$, $\mathbb{F}(\mathcal{F}_y(Z))$ and $\mathbb{F}(Z)$ are represented by the same loop.
		
		It remains to prove that every vector bundle is represented by a homotopy class of $\CVb{n}$. Following \cite[Lemma 6.11]{Ballard}, a perfect object $X \in \Perf(C_n)$ is isomorphic to the shift of a locally-free sheaf of rank $n$ if and only if there exists $m \in \mathbb{Z}$ such that $\Hom^{\bullet}(X, \Bbbk(z)) \cong \Bbbk^n[m]$ as graded vector spaces for all closed points $z \in C_n$. The proof of \cite[Lemma 6.11]{Ballard} relies on the assumption that $C_n$ is projective and connected but does not require irreducibility. Without loss of generality, we may assume that $X$ sits at the mouth of its homogeneous tube. Suppose now, that the walk $W_{\gamma}$ of a loop $\gamma=\gamma_{\mathcal{E}}$ for a vector bundle $\mathcal{E}$ contains a subsequence of the form $\varepsilon_i \kappa_i^l \varepsilon_{i+1}^{-1}$ or its inverse. This condition is equivalent to $\gamma \not \in \CVb{n}$. Using Lemma \ref{LemmaDegreeOfIntersections}, it is not difficult to see that under these conditions the arc of the skyscraper sheaf of a singular point $z \in \mathbb{P}^1_i \cap \mathbb{P}_1^{i+1}$ has two intersections with $\gamma$ which correspond to morphisms in $\Hom^{\bullet}(X, \Bbbk(z))$ in different degrees. Hence such a loop cannot represent a vector bundle of $C_n$. 	This finishes the proof.\end{proof}
		
\subsection{An alternative proof for the classification of simple vector bundles}\label{SectionAlternativeProof}
\noindent We give an alternative proof of Theorem \ref{TheoremBodnarchukDrozdGreuel}. By Theorem  \ref{TheoremImagesOfSimpleVectorBundles}, it is sufficient to prove the following.

	\begin{prp}\label{PropositionNewProof}
		Let $\mathbbm{n} \in \mathbb{Z}^{\mathbb{Z}_{n r}}$. If $\gamma(\mathbbm{n})$ contains a simple loop, then $r$ and $\underline{\mathbbm{n}}$ are coprime and $\mathbbm{n} \sim \seq(r, \mathbb{d}(\mathbbm{n}))$.
	\end{prp}
	\begin{proof}[Proof]
		Given a sequence $\seq \in \mathbb{Z}^{\mathbb{Z}_{n r}}$ and $q \in \mathbb{Z}_n$, we can define the contraction $C^q(\seq) \in \mathbb{Z}^{\mathbb{Z}_{(n-1)r}}$ of $\seq$. Regarding $\seq$ as element of $\Mat_{n \times r}(\mathbb{Z})$ with cyclic rows and  $C^q(-)$ as an operator $C^q: \Mat_{n \times r}(\mathbb{Z}) \rightarrow \Mat_{(n-1) \times r}(\mathbb{Z})$, $C^q$ stores the sum of the $q$-th and the $q+1$-th row of $\seq$ in the $q$-th row of $C^q(\seq)$ and leaves all the other rows untouched while preserving the cyclic order of the rows.  On $\mathscr{T}^n$ the matrix $C^q(\seq)$ corresponds to the loop $\gamma(\seq)$ regarded as a loop in $\mathscr{T}^{n-1}$. The operation $C^q$ induces a map $\CVb{n} \rightarrow \CVb{n-1}$ which preserves simplicity and hence preserves properties (1) and (2) in Lemma \ref{LemmaSequencesHaveIntersectionsProperties}.

		Conversely, given a simple loop $\gamma=\gamma(\seq) \subseteq \mathscr{T}^{n-1}$, we may assume that $\gamma$ and $\delta \coloneqq \gamma_{\Bbbk(x)}^q \subseteq \mathscr{T}^{n-1}$ are in minimal position. In particular, $\gamma \overrightarrow{\cap}\delta$ contains exactly $r= \rk \seq$ intersections and is cyclically ordered by the orientation of $\delta$. Choose an intersection $u \in \gamma \overrightarrow{\cap} \delta$ and denote by $u^+ \in \gamma \overrightarrow{\cap} \delta$ its successor. Let $p \in \delta \setminus \{u, u^+\}$ be a point on the segment between $u$ to $u^+$. Let $\varphi:\mathscr{T}^{n-1} \rightarrow \mathscr{T}^n$ be a homeomorphism which maps all punctures of $\mathscr{T}^{n-1}$ and $p$ to punctures in $\mathscr{T}^n$. Moreover, we require that it maps the set $\{\delta\} \cup \{\varepsilon_i\}_{i \in \mathbb{Z}_{n-1}}$ bijectively onto the set of curves $\{\varepsilon_j\}_{j \in \mathbb{Z}_n}$ in $\mathscr{T}^n$ and maps $\kappa_j$ to $\kappa_j$ for all $j \neq q$ as well as $\kappa_q$ to the concatenation of $\kappa_{q}$ and $\kappa_{q+1}$. The isotopy class of such a map $\varphi$ is not unique but appears as part of a $\mathbb{Z}$-family of isotopy classes of homeomorphisms. Namely, every two possible choices for $\varphi$ as above are related by composition with a homeomorphism 

		$$
			D_{\gamma_{\Bbbk(x)}^q}^m \circ D_{\gamma_{\Bbbk(x)}^{q+1}}^{-m} \in \PMCG(\mathscr{T}^n)
		$$
		 \noindent for some unique $m \in \mathbb{Z}$. For every  $\varphi$ as above let us denote by  $\gamma^{\varphi}_u \subseteq \mathscr{T}^n$ the image of $\gamma$ under $\varphi$. It follows from Remark \ref{RemarkDehnTwistActionOnSequences}, that we obtain all homotopy classes of simple loops $\gamma' \subseteq \mathscr{T}^n$ in $\CVb{n}$ such that $C^q(\seq(\gamma')) \simeq \seq$. Moreover, we observe that the multi-degrees of $\gamma^{\varphi}_u$ and $\gamma^{\varphi'}_{u'}$ agree if and only if $\varphi$ and $\varphi'$ are isotopic and $u=u'$. In particular, this is the case if the two loops are homotopic.
		
		By induction, our arguments reduce the uniqueness of sequences $\seq$ as in Lemma \ref{LemmaSequencesHaveIntersectionsProperties} to the uniqueness of such sequences in the case $n=1$. Here the result follows from the well-known fact that homotopy classes of non-separating simple loops on $\mathscr{T}^1$ are in bijection with pairs of coprime integers. The bijection projects a loop onto its associated class in $H_1(\mathscr{T}^1, \mathbb{Z})\cong \mathbb{Z}^2$. Choosing the classes of $\gamma_{\Pic}$ and $\gamma_{\Bbbk(x)}^0$ as the basis of $H_1(\mathscr{T}^1, \mathbb{Z})$, $\gamma=\gamma(\seq)$ is mapped to the pair $(\rk \seq, \mathbb{d}(\seq))$.	\end{proof}	
	
\appendix
	\section{Images of smooth skyscraper sheaves and the Picard group}\label{ChapterAppendix}
	\label{SectionImagesSkyscraperLineBundles}
	\noindent This section contains the proof of Theorem \ref{TheoremImages} which was explained to us by Igor Burban. Our main reference is \cite{BurbanDrozdTilting}. As before let $\Pic^{\mathbb{0}}(C_n)$ denote the set of line bundles of multi-degree $(0, \dots, 0)$ over $C_n$ and let $\mathbb{F}: \Perf(C_n) \hookrightarrow \mathcal{D}^b(\Lambda_n)$ denote the embedding.  Theorem \ref{TheoremImages} asserts the following. 
	
		\begin{thm}
		\ \begin{enumerate}
			\setlength\itemsep{1ex}

			\item The essential image of $\Pic^{\mathbb{0}}(C_n)$ under $\mathbb{F}$ consists of the isomorphism classes of the complexes $\mathcal{O}(\lambda)$, where $\lambda \in \Bbbk^{\times}$.
			
			\item  The essential image of the skyscraper sheaves of smooth points $x \in \mathbb{P}^1_i$ under $\mathbb{F}$ consists of the isomorphism classes of the complexes $\Bbbk(i, \lambda)$, where $\lambda \in \Bbbk^{\times}$.
		\end{enumerate}
		
	\end{thm}
\noindent The definitions of the complexes $\mathcal{O}(\lambda)$ and $\Bbbk(i, \lambda)$ are found in Figure \ref{FigurePicardGroupComplex} on page \pageref{FigurePicardGroupComplex} and the paragraph succeeding it. The definition of $\mathbb{F}$ is recalled in Section \ref{AppendixResolutionOfCycles} below.\medskip

	\noindent For the remainder of this section fix $n \geq 1$ and write $\Lambda=\Lambda_n$, $C=C_n$ as well as $\mathbb{P}^1_i$, $i \in \mathbb{Z}_n$, for the irreducible components of $C_n$.  Let $\pi: \widetilde{C} \rightarrow C$ be a normalization map. Then, $\widetilde{C}=\sqcup_{i=0}^{n-1}{\widetilde{U}_i}$, where $\widetilde{U}_i=\pi^{-1}(\mathbb{P}^1_i) \cong \mathbb{P}^1$. By changing coordinates, we may assume that $0, \infty \in \mathbb{P}^1$ are the preimages of the singular points in $U_i$. Set $\widetilde{\mathcal{O}}\coloneqq \pi_{\ast}(\mathcal{O}_{\widetilde{C}})$. Then, $\mathcal{O}\coloneqq\mathcal{O}_C$ is a subsheaf of $\widetilde{\mathcal{O}}$ and $\widetilde{\mathcal{O}}$ decomposes as 
	\[\widetilde{\mathcal{O}} = \bigoplus_{i \in \mathbb{Z}_n}{\widetilde{\mathcal{O}}_i},\]
	
	\noindent where $\widetilde{\mathcal{O}}_i \coloneqq \pi_{\ast}(\mathcal{O}_{\widetilde{C}}|_{\widetilde{U_i}})$. We denote by $\mathcal{I}$ the ideal sheaf of the singular locus of $C$. 
	
	\subsection{Non-commutative resolutions of cycles}\label{AppendixResolutionOfCycles}
	In \cite{BurbanDrozdTilting}, Burban and Drozd introduced the so-called \textit{Auslander-sheaf} $\mathscr{A}$. It is the sheaf of $\mathcal{O}$-orders on $C$ given by
	\[\mathscr{A} =  \left(  \begin{array}{cc} \mathcal{O} & \widetilde{\mathcal{O}} \\ \mathcal{I} & \widetilde{\mathcal{O}} \end{array} \right).\]
	The diagonal embedding of $\mathcal{O}$ defines an action of $\mathcal{O}$ on $\mathscr{A}$ and endows $\mathscr{A}$ with the structure of an $\mathcal{O}$-algebra. \medskip
	
	\noindent Burban and Drozd proved that the bounded derived category $\mathcal{D}^b(\Coh \mathbb{X})$ of coherent sheaves over the non-commutative ringed space $\mathbb{X}=(C, \mathscr{A})$ admits a tilting object $\mathscr{H}$ defined as follows.\medskip
	
	\noindent Let $\mathcal{S}$ denote the torsion sheaf of $\mathscr{A}$-modules defined by the cokernel of the canonical inclusion
	\begin{equation}\label{EquationProjectiveResolution}\begin{tikzcd}[ampersand replacement=\&] \left( \begin{array}{c} \mathcal{I} \\ \mathcal{I} \end{array} \right) \arrow{r} \& \left(\begin{array}{c}\mathcal{O} \\ \mathcal{I} \end{array}\right). \end{tikzcd}\end{equation}
	
	\noindent  Furthermore, denote by
	\[\begin{array}{ccc}{\mathcal{F} \coloneqq \left( \begin{array}{c} \mathcal{O} \\ \mathcal{I} \end{array} \right)} & \text{and}& {\mathcal{P} \coloneqq \left( \begin{array}{c} \widetilde{\mathcal{O}} \\ \widetilde{\mathcal{O}} \end{array} \right),} \end{array}\]
	\noindent the $\mathscr{A}$-modules corresponding to the columns of $\mathscr{A}$.
	\noindent The tensor product $\mathcal{F} \otimes_{\mathcal{O}} -$ defines a fully faithful functor $\mathbb{J}: \Coh C \rightarrow \Coh \mathbb{X}$.

	By definition, we have $\mathcal{P}= \bigoplus_{i \in \mathbb{Z}_n}{\mathcal{P}_i}$, where $\mathcal{P}_i\coloneqq \mathbb{J}(\widetilde{\mathcal{O}}_i)$. Moreover, $\mathcal{S}= \bigoplus_{i \in \mathbb{Z}_n}{\mathcal{S}_i}$, where $\mathcal{S}_i$ is a skyscraper sheaf supported at the singular point corresponding to the image of $0 \in \widetilde{U}_i$ under the normalization map $\pi$. For any integer $m \in \mathbb{Z}$ and any line bundle $\mathcal{L}$ of multi-degree $(m, \dots, m)$, define 
	\[\mathcal{P}(m)\coloneqq \bigoplus_{i \in \mathbb{Z}_n}{\mathcal{P}_i \otimes_{\mathcal{O}} \mathcal{L}}.\]
	
	 \noindent It was shown in \cite[Lemma 4]{BurbanDrozdTilting} that the isomorphism class of $\mathcal{P}(m)$ is independent of the choice of $\mathcal{L}$. In particular, up to isomorphism, the object
	\[\mathscr{H}:= \mathcal{S}[-1] \oplus \mathcal{P}(-1) \oplus \mathcal{P}.\]
\noindent is well-defined. The object $\mathscr{H}$ is a tilting and $\Lambda$ is isomorphic to the opposite endomorphism algebra of $\mathscr{H}$.  The derived functor of $\mathbb{J}$ restricts to an embedding $\Perf(C) \hookrightarrow \mathcal{D}^b(\mathbb{X})$ and gives rise to an embedding $\mathbb{B}:\Perf(C) \hookrightarrow\mathcal{D}^b(\Lambda)$. For $i \in [0,n)$, let $x_i \in U_i$ be smooth. The desired embedding $\mathbb{F}: \Perf(C) \rightarrow \mathcal{D}^b(\Lambda)$, as cited in the main text of this paper, is given by
	\[\mathbb{F}=\mathbb{B} \circ \prod_{i \in \mathbb{Z}_n} T_{\Bbbk(x_i)} \cong \mathbb{B} \circ \otimes^{\mathbb{L}} \mathcal{L}\left(x_0 + \cdots + x_{n-1}\right).\] 
		
	\subsection{Images of skyscraper sheaves of smooth points}
	\noindent For any $i \in \mathbb{Z}_n$, let $x \in U_i$ be smooth corresponding to a point $(\lambda: \mu) \in \mathbb{P}^1$. By our assumption on the chosen coordinates, $\lambda, \mu \neq 0$. Let $\widetilde{x} \in \widetilde{U}_i$ denote the unique preimage of $x$ under $\pi$. Then, $\pi_{\ast}(\Bbbk(\widetilde{x}))\cong \Bbbk(x)$ and if $z_0^i, z_{\infty}^i: \mathcal{O}_{\mathbb{P}^1} \rightarrow \mathcal{O}_{\mathbb{P}^1}(1)$ correspond to the chosen coordinates on $\widetilde{X}_i$ vanishing at $0$ and $\infty$ respectively, then the $\Bbbk(x)$ is the cokernel of $\alpha_i: \widetilde{\mathcal{O}}_i(-1) \rightarrow \widetilde{\mathcal{O}}_i$, where  
	$$\alpha_i\coloneqq \pi_\ast \left(\mu z_0^i(-1) - \lambda z_{\infty}^i(-1)\right).$$ 
	
	\noindent	Keeping in mind that $\mathbb{J} \circ \pi_{\ast}(\mathcal{O}_{\widetilde{U_i}}(m))=\mathcal{P}_i(m)$, the isomorphism $\Lambda \cong \End(\mathscr{H})^{\operatorname{op}}$ identifies the arrows $a_i$ and $c_i$ in $\Lambda$ with the maps $\mathbb{J}(\pi_\ast \circ z_{0}^i), \mathbb{J}(\pi_\ast \circ z_{\infty}^i) \in \Hom(\mathcal{P}_i(-1), \mathcal{P}_i)$. 
    We deduce that there is a short exact sequence 
	\[\begin{tikzcd}0 \arrow{r} & \mathcal{P}_i(-1) \arrow{rr}{\mathbb{J}(\alpha_i)} & & \mathcal{P}_i \arrow{r} & \mathbb{J}(\Bbbk(x)) \arrow{r} & 0.\end{tikzcd}\] 
	The tilting functor associated with $\mathscr{H}$ sends $\mathcal{P}_i(-1)$ to the indecomposable projective module of $s(a_i)$ and $\mathcal{P}_i$ to the indecomposable projective module of $t(a_i)$  proving that $\mathbb{B}(\Bbbk(x))$ is isomorphic to $\Bbbk(i, \lambda^{-1}\mu)$. Since $\Hom^{\bullet}(\Bbbk(x), \Bbbk(y))=0$ for all $x \neq y$ and $\Bbbk(x)$ is $1$-spherical, it follows from Corollary \ref{CorollaryTwistFunctorsUnderEmbeddings} and Lemma \ref{LemmaImageSphericalObjectsUnderItsTwist} that $\mathbb{F}(\Bbbk(x)) \cong\Bbbk(i,\lambda^{-1}\mu)$. \medskip

	\subsection{Image of the Picard group} The image of $\Pic^{\mathbb{0}}(C)$ under $\mathbb{F}$ coincides with the image of the line bundles of degree $(1, \dots, 1)$ under $\mathbb{B}$.\medskip
	
	\noindent Let $\mathcal{L}$  be a line bundle on $C$ of multi-degree $(1, \dots, 1)$ and set $\mathcal{G}:=\mathbb{J}(\mathcal{L})=\mathcal{F} \otimes_{\mathcal{O}} \mathcal{L}$. We want to compute the dimension of the vector spaces $\Ext^i(\mathcal{S}, \mathcal{G})$, $\Ext^i(\mathcal{P}, \mathcal{G})$ and $\Ext^i(\mathcal{P}(-1), \mathcal{G})$ for $ i \geq 0$. Note that by  \cite[Theorem 2]{BurbanDrozdTilting}, the global dimension of $\Coh \mathbb{X}$ is $2$.

	Since $\mathcal{S}$ is a torsion sheaf and $\mathcal{G}$ is torsion-free, it follows that ${\homom}_{\mathscr{A}}(\mathcal{S}, \mathcal{G})\cong 0$. Thus, $\Hom_{\mathscr{A}}(\mathcal{S}, \mathcal{G})=0$. By \eqref{EquationProjectiveResolution}, $\mathcal{S}$ has a locally projective resolution of length one. Thus, $\SExt_{\mathscr{A}}^i(\mathcal{S}, \mathcal{G})=0$ for all $i \geq 2$ and one deduces from the local computations of $\SExt_{\mathscr{A}}^1(\mathcal{S}_j,\mathcal{G})$ and the local-to-global spectral sequence that $\Ext^1_{\mathscr{A}}(\mathcal{S}_j, \mathcal{G})\cong \Bbbk$ for all $j \in \mathbb{Z}_n$ and $\Ext_{\mathscr{A}}^i(\mathcal{S}, \mathcal{G})=0$ for all $i \geq 2$.\medskip
	
	\noindent By \cite[Corollary 4]{BurbanDrozdTilting}, there exists an isomorphism of $\mathcal{O}$-modules, 
	$$\mathcal{I} \otimes \mathcal{L} \cong {\homom}_{\mathscr{A}}(\mathcal{P}, \mathcal{G}),$$ 
	and by \cite[Corollary 3]{BurbanDrozdTilting}, $\Ext^i_{\mathscr{A}}(\mathcal{P}, \mathcal{G})=0$ for all $i \geq 0$. 
	We have a sequence of isomorphisms
	\[{\homom}_{\mathscr{A}}(\mathcal{P}(-1), \mathcal{G}) \cong {\homom}_{\mathscr{A}}(\mathcal{P}(-1) \otimes_{\mathcal{O}} \mathcal{L}^{\vee}, \mathcal{F}) \cong {\homom}_{\mathscr{A}}(\mathcal{P}(-2), \mathcal{F}),\]
	where $\mathcal{L}^{\vee}$ denotes the dual of $\mathcal{L}$.
	Using that $\mathcal{P}(-2) \cong \left( \begin{array}{c} \mathcal{I} \\ \mathcal{I} \end{array}\right)$ and that both $\mathcal{F}$ and $\mathcal{P}(-2)$ are torsion free, it follows from  \cite[Proposition 6]{BurbanDrozdTilting}, that 
	${\homom}_{\mathscr{A}}(\mathcal{P}(-1), \mathcal{G}) \cong \widetilde{\mathcal{O}}$ as $\mathcal{O}$-modules. Since $\Bbbk \cong \Gamma(\widetilde{U_i}, \mathcal{O}_{\widetilde{U_i}})\cong \Gamma(X, \widetilde{\mathcal{O}}_i)$, we deduce $\Hom_{\mathscr{A}}(\mathcal{P}_i(-1), \mathcal{G}) \cong \Bbbk$. Moreover, it follows from \cite[Corollary 3]{BurbanDrozdTilting}, that $\Ext^i(\mathcal{P}(-1), \mathcal{G})=0$ for all $i \geq 1$.\medskip
	
	\noindent We have shown that $\Hom^{\bullet}(\mathscr{H}, \mathcal{G})$, and hence the cohomology of $\mathbb{B}(\mathcal{G})$, is concentrated in degree $0$. In particular, $\mathbb{B}(\mathcal{G})$ is quasi-isomorphic to a $\Lambda$-module.
	Since $\mathbb{B}(\mathcal{L})$ is indecomposable, we conclude from the isomorphisms $\Hom_{\mathscr{A}}(\mathcal{P}_i(-1), \mathcal{G}) \cong \Bbbk \cong \Hom_{\mathscr{A}}(\mathcal{S}_i[-1], \mathcal{G})$ for all $i \in \mathbb{Z}_n$ and from the classification of $\tau$-invariant indecomposable objects  in $\mathcal{D}^b(\Lambda)$ from Section \ref{SectionComplexOfLoops} that $\mathbb{F}(\mathcal{L}) \cong \mathcal{O}(\lambda)$ for some $\lambda \in \Bbbk^{\times}$.
	
	\bibliography{Bibliography}{}
	\bibliographystyle{alpha}
	
	\ \medskip
	
\end{document}